\documentclass[12pt,a4paper]{amsart}

\usepackage[final=true]{hyperref}

\usepackage[T1]{fontenc}              
\usepackage[latin1]{inputenc}         
\usepackage{geometry}                 
\usepackage{amsmath,amsfonts,amssymb,amsrefs} 
\usepackage{amsthm}  
\usepackage{stmaryrd}
\usepackage{comment} 
\usepackage{graphicx}                 
\usepackage{subfig}                   
\usepackage{listings}                 
\usepackage{enumerate}
\usepackage{bbm}
\usepackage{tikz-cd}
\usepackage{tikz}
\usetikzlibrary{calc}
\usetikzlibrary{arrows}
\usetikzlibrary{matrix}
\usetikzlibrary{decorations.markings}
\usepackage{rotating}
\usepackage{faktor}
\usepackage{soul} 
\usepackage[capitalize]{cleveref}
\usepackage{xpatch}
\usepackage{chngcntr}
\counterwithin{equation}{section}
\usepackage{mathtools}
\usepackage{ulem}

\DefineSimpleKey{bib}{myurl}
\newcommand\myurl[1]{\url{#1}}
\BibSpec{webpage}{
  +{}{\PrintAuthors} {author}
  +{,}{ \textit} {title}
  +{}{ \parenthesize} {date}
  +{,}{ \myurl} {myurl}
  +{,}{ } {note}
  +{.}{ } {transition}
}

\makeatletter
\newcounter{proofpart}
\xpretocmd{\proof}{\setcounter{proofpart}{0}}{}{}
\newcommand{\proofpart}[1]{
  \par
  \addvspace{\medskipamount}
  \stepcounter{proofpart}
  \noindent\emph{Part \theproofpart: #1}\par\nobreak\smallskip
  \@afterheading
}
\makeatother

\newtheorem{thmu}{Theorem}
\newtheorem{coru}[thmu]{Corollary}

\newtheorem{lem}{Lemma}[section]
\newtheorem{prop}[lem]{Proposition}
\newtheorem{cor}[lem]{Corollary}
\newtheorem{thm}[lem]{Theorem}

\theoremstyle{definition}
\newtheorem{defin}[lem]{Definition}
\newtheorem*{acknowledgements}{Acknowledgements}
\newtheorem{remark}[lem]{Remark}

\newtheorem{setup}[lem]{Setup}
\newtheorem*{motivating question}{Motivating question}

\newtheorem{example}[lem]{Example}

\DeclareMathOperator{\Rad}{Rad} 

\DeclareMathOperator{\projdim}{proj.\!dim}
\DeclareMathOperator{\Soc}{Soc}

\DeclareMathOperator{\Thick}{Thick}
\DeclareMathOperator{\Top}{Top}
\DeclareMathOperator{\gldim}{gl.\!dim}
\DeclareMathOperator{\Ext}{Ext}
\DeclareMathOperator{\Gr}{Gr}
\DeclareMathOperator{\gr}{gr}
\DeclareMathOperator{\grproj}{grproj}
\DeclareMathOperator{\qgr}{qgr}

\DeclareMathOperator{\Hom}{Hom}
\DeclareMathOperator{\RHom}{\textbf{R}Hom}

\DeclareMathOperator{\End}{End}
\DeclareMathOperator{\Modu}{Mod}
\DeclareMathOperator{\add}{add}
\DeclareMathOperator{\modu}{mod}

\DeclareMathOperator{\stgrmodu}{\setul{0.6ex}{0.1ex}
\textup{\ul{gr}}}

\DeclareMathOperator{\Hm}{H}

\DeclareMathOperator{\op}{op}

\DeclareMathOperator{\perf}{perf}
\DeclareMathOperator{\fd}{fd}
\DeclareMathOperator{\dg}{dg}
\DeclareMathOperator{\D}{\mathcal{D}}
\DeclareMathOperator{\Dif}{Dif}
\DeclareMathOperator{\Z}{\mathbb{Z}}
\DeclareMathOperator{\A}{\mathcal{A}}
\DeclareMathOperator{\B}{\mathcal{B}}
\DeclareMathOperator{\U}{\mathcal{U}}
\DeclareMathOperator{\T}{\mathcal{T}}

\newcommand{\Thicks}{\operatorname{Thick}^{\scriptscriptstyle{\langle - \rangle}\!}}

\title{Higher Koszul duality and connections with $n$-hereditary algebras}
\author{Johanne Haugland and Mads Hustad Sand\o y}

\begin{document}


\address{Department of mathematical sciences, NTNU, NO-7491 Trondheim, Norway}
\email{johanne.haugland@ntnu.no}
\email{mads.sandoy@ntnu.no}

\begin{abstract}
We establish a connection between two areas of independent interest in representation theory, namely Koszul duality and higher homological algebra. This is done through a generalization of the notion of $T$-Koszul algebras, for which we obtain a higher version of classical Koszul duality. Our approach is motivated by and has applications for $n$-hereditary algebras. In particular, we characterize an important class of $n$-$T$-Koszul algebras of highest degree $a$ in terms of $(na-1)$-representation infinite algebras. As a consequence, we see that an algebra is $n$-representation infinite if and only if its trivial extension is $(n+1)$-Koszul with respect to its degree $0$ part. Furthermore, we show that when an $n$-representation infinite algebra is $n$-representation tame, then the bounded derived categories of graded modules over the trivial extension and over the associated $(n+1)$-preprojective algebra are equivalent. In the $n$-representation finite case, we introduce the notion of almost $n$-$T$-Koszul algebras and obtain similar results. 
\end{abstract}

\maketitle

\tableofcontents

\section{Introduction}
Global dimension is a useful measure for the objects one studies in representation theory of finite dimensional algebras. However, while algebras of global dimension $0$ and $1$ are exceptionally well understood, it seems quite difficult to develop a general theory for algebras of higher global dimension. This is a background for studying the class of \textit{$n$-hereditary algebras} \cites{Iyama_2007,Iyama-Oppermann APR,Iyama-Oppermann,Herschend-Iyama-Oppermann,Herschend-Iyama,Herschend-Iyama2,Iyama,Darpo-Iyama}. These algebras play an important role in higher Auslander--Reiten theory \cites{Iyama_2007_2,Iyama_2008,Kvamme-Jasso}, which has been shown to have connections to commutative algebra, both commutative and non-commutative algebraic geometry, combinatorics, conformal field theory, and homological mirror symmetry \cites{Amiot-Iyama-Reiten, Herschend-Iyama-Minamoto-Oppermann, Iyama-Wemyss, Oppermann-Thomas, Evans-Pugh, DJL}. An $n$-hereditary algebra has global dimension less than or equal to $n$ and is either \textit{$n$-representation finite} or \textit{$n$-representation infinite}. As one might expect, these notions coincide with the classical definitions of representation finite and infinite hereditary algebras in the case $n=1$. 

Like in the classical theory, $n$-hereditary algebras have a notion of (higher) preprojective algebras. If $A$ is $n$-representation infinite and the $(n+1)$-preprojective $\Pi_{n+1}A$ is graded coherent, there is an equivalence \mbox{$\D^b(\modu A) \simeq \D^b(\qgr \Pi_{n+1}A)$}, where \mbox{$\qgr \Pi_{n+1}A$} denotes the category of finitely presented graded modules modulo finite dimensional modules \cites{Minamoto 2012, Minamoto & Mori 2011}. On the other hand, the bounded derived category of a finite dimensional algebra of finite global dimension is always equivalent to the stable category of finitely generated graded modules over its trivial extension \cite{Happel}. Combining these two equivalences, and using the notation $\Delta A$ for the trivial extension of $A$, one obtains
\begin{equation} \label{motivating equivalence}
    \stgrmodu (\Delta A) \simeq \D^b(\qgr \Pi_{n+1}A).
\end{equation}

The equivalence above brings to mind the acclaimed Bern\v{s}te\u{\i}n--Gel\cprime fand--Gel\cprime fand-correspondence, which can be formulated as \mbox{$\stgrmodu \Lambda \simeq \D^b(\qgr \Lambda^!)$} for a finite dimensional Frobenius Koszul algebra $\Lambda$ and its graded coherent Artin--Schelter regular Koszul dual $\Lambda^!$ \cite{Bernstein-Gelfand-Gelfand}. The BGG-correspondence is known to descend from the Koszul duality equivalence between bounded derived categories of graded modules over the two algebras, as indicated in the following diagram
\[
    \begin{tikzcd}[column sep=22, row sep=20] 
   \D^b(\gr \Lambda) \arrow[r,"\simeq"] \arrow[d] & \D^b(\gr \Lambda^!) \arrow[d]  \\
    \stgrmodu \Lambda \arrow[r,dashed,"\simeq"] & \D^b(\qgr \Lambda^!).
    \end{tikzcd}
\]
It is natural to ask whether something similar is true in the $n$-representation infinite case, i.e.\ if the equivalence (\ref{motivating equivalence}) is a consequence of some higher Koszul duality pattern. This is a motivating question for this paper.

\begin{motivating question}
Is the equivalence (\ref{motivating equivalence}) a consequence of some higher Koszul duality pattern?
\end{motivating question}

One reasonable approach to this question is to study generalizations of the notion of Koszulity. A positively graded algebra $\Lambda$ generated in degrees $0$ and $1$ with semisimple degree $0$ part is known as a \textit{Koszul algebra} if $\Lambda_0$ is a graded self-orthogonal module over $\Lambda$ \cites{Beilinson-Ginzburg-Soergel, Priddy}. This means that \mbox{$\Ext^{i}_{\gr \Lambda}(\Lambda_0, \Lambda_0 \langle j \rangle) = 0$} whenever \mbox{$i \neq j$}, where $\langle - \rangle$ denotes the graded shift. Using basic facts about Serre functors and triangle equivalences, one can show that a similar statement holds for $\Delta A$ with respect to its degree $0$ part \mbox{$(\Delta A)_0 = A$} in the case where $A$ is $n$-representation infinite. Here, the algebra $A$ is clearly not necessarily semisimple, but it is of finite global dimension. 

In \cite{Green-Reiten-Solberg} Green, Reiten and Solberg present a notion of Koszulity for more general graded algebras, where the degree $0$ part is allowed to be an arbitrary finite dimensional algebra. Their work provides a unified approach to Koszul duality and tilting equivalence. Koszulity in this framework is defined with respect to a module $T$, and thus the algebras are called \textit{$T$-Koszul}. Madsen \cite{Madsen 2011} gives a simplified definition of $T$-Koszul algebras, which he shows to be a generalization of the original one whenever the degree $0$ part is of finite global dimension.

We generalize Madsen's definition to obtain the notion of \textit{$n$-$T$-Koszul algebras}, where $n$ is a positive integer and $n=1$ returns Madsen's theory. In \cref{Existence of T-Koszul equivalence} we prove that an analogue of classical Koszul duality holds in this generality, and we recover a version of the BGG-correspondence in \cref{answer motivating question}. Moreover, we provide a characterization of an important class of $n$-$T$-Koszul algebras of highest degree $a$ in terms of $(na-1)$-representation infinite algebras. This characterization is given as the following theorem.

\begin{thmu} [see \cref{Characterization}] \label{thm1}
Let $\Lambda=\oplus_{i \geq 0}\Lambda_i$ be a finite dimensional graded Frobenius algebra of highest degree $a \geq 1$ with $\gldim \Lambda_0 < \infty$. Consider a basic graded $\Lambda$-module $T$ which is concentrated in degree $0$ and a tilting module over $\Lambda_0$. We assume $T_\mu \simeq T$ as $\Lambda$-modules for the Nakayama automorphism $\mu$ of $\Lambda$.
The following statements are equivalent:
\begin{enumerate}
    \item $\Lambda$ is $n$-$T$-Koszul. 
    \item $\widetilde{T} = \oplus_{i=0}^{a-1}\Omega^{-ni}T \langle i \rangle$ is a tilting object in $\stgrmodu \Lambda$ and \mbox{$B= \End_{\stgrmodu \Lambda}(\widetilde{T})$} is $(na-1)$-representation infinite.
\end{enumerate}
\end{thmu}

As a consequence of \cref{thm1}, we see that an algebra is $n$-representation infinite if and only if its trivial extension is $(n+1)$-Koszul with respect to its degree $0$ part. 

\begin{coru} [see \cref{motivating result}] \label{cor2}
Let $\Lambda = \Lambda_0 \oplus \Lambda_1$ be a finite dimensional graded Frobenius algebra of highest degree $1$ with $\gldim \Lambda_0 < \infty$.
Then $\Lambda$ is $(n+1)$-Koszul with respect to $T=\Lambda_0$ if and only if $\Lambda_0$ is $n$-representation infinite. In particular, we obtain a bijective correspondence given by $A \mapsto \Delta A$
\[
\left\{
\begin{tabular}{@{}l@{}}
    isomorphism classes \\
    of basic $n$-representa- \\ 
    tion infinite algebras
\end{tabular}
\right\}
{\rightleftarrows}
\left\{
\begin{tabular}{@{}l@{}}
    isomorphism classes of graded symmetric finite \\
    dimensional algebras of highest degree $1$ which are \\
    $(n+1)$-Koszul with respect to their degree $0$ part
\end{tabular}
\right\}.
\]
\end{coru}
\noindent Furthermore, we show in \cref{cor: dual of triv.ext.} that when $A$ is $n$-representation infinite, then the higher Koszul dual of its trivial extension is given by the associated $(n+1)$-preprojective algebra. Combining this with our version of the BGG-correspondence, \cref{n-rep tame} gives an affirmative answer to our motivating question. In particular, we see that when an $n$-representation infinite algebra $A$ is \mbox{$n$-representation} tame, then the bounded derived categories of graded modules over $\Delta A$ and over $\Pi_{n+1}A$ are equivalent, and that this descends to give an equivalence \mbox{$\stgrmodu (\Delta A) \simeq \D^b(\qgr \Pi_{n+1}A)$}. 

Having developed our theory for one part of the higher hereditary dichotomy, we ask and provide an answer to whether something similar holds in the higher representation finite case. Inspired by and seeking to generalize the notion of almost Koszul algebras as developed by Brenner, Butler and King \cite{Brenner-Butler-King}, we define \textit{(minimally) almost $n$-$T$-Koszul algebras}. 
This enables us to show a similar characterization result as in the $n$-$T$-Koszul case, namely the following theorem.

\begin{thmu} [see \cref{nrepfinchar}] \label{thm3}
Let $\Lambda$ and $T$ be as in \cref{thm1}. The following statements are equivalent:
\begin{enumerate} 
    \item $\Lambda$ is minimally almost $n$-$T$-Koszul.  
    \item $\widetilde{T} = \oplus_{i=0}^{a-1}\Omega^{-ni}T \langle i \rangle$ is a tilting object in $\stgrmodu \Lambda$ and \mbox{$B = \End_{\stgrmodu \Lambda}(\widetilde{T})$}
    is $(na-1)$-representation finite.
\end{enumerate}
\end{thmu}

\noindent This yields the corollary below, which is a higher representation finite analogue of \cref{cor2}.

\begin{coru}[see \cref{cor sec 7}] 
Let $\Lambda = \Lambda_0 \oplus \Lambda_1$ be a finite dimensional graded Frobenius algebra of highest degree $1$ with $\gldim \Lambda_0 < \infty$.
Then $\Lambda$ is minimally almost $(n+1)$-Koszul with respect to $T=\Lambda_0$ if and only if $\Lambda_0$ is $n$-representation finite. In particular, we obtain a bijective correspondence given by $A \mapsto \Delta A$
\[
\left\{
\begin{tabular}{@{}l@{}}
    isomorphism classes of  \\
    basic $n$-representation \\ 
    finite algebras
\end{tabular}
\right\}
{\rightleftarrows}
\left\{
\begin{tabular}{@{}l@{}}
    isomorphism classes of graded symmetric finite \\
    dimensional algebras of highest degree $1$ which  \\
    are minimally almost $(n+1)$-Koszul with \\ respect to their degree $0$ parts
\end{tabular}
\right\}.
\]
\end{coru}

Altogether, we establish a connection between two areas of independent interest in representation theory, namely Koszul duality and higher homological algebra. Notice that a relationship between Koszulity and $n$-hereditary algebras is also studied in \cite{BP}, and more recently in \cite{Grant-Iyama}. In some sense, parts of the theory we develop is a generalized Koszul dual version of results in \cites{Minamoto & Mori 2011,Grant-Iyama}. Note that many of our results are novel already in the case $n=1$. This demonstrates that questions arising from higher homological algebra can lead to new insight also in the classical case.

This paper is organized as follows. In \cref{preliminaries} we highlight relevant facts about graded algebras, before giving an overview of the notions of tilting subcategories and Serre functors. The definition and general theory of $n$-$T$-Koszul algebras is presented in \cref{Graded n-self-orth and n-T-Koszul}. As a foundation for the rest of the paper, \cref{n-hereditary} is devoted to recalling definitions and known facts about $n$-hereditary algebras. Note that this section does not contain new results. In \cref{T-Koszul and n-rep inf} we state and prove our results on the connections between $n$-$T$-Koszul algebras and higher representation infinite algebras. Finally, almost $n$-$T$-Koszul algebras are introduced in \cref{section almost n-T}, and we develop their theory along the same lines as was done in \cref{T-Koszul and n-rep inf}.

\subsection{Conventions and notation}
Throughout this paper, let $k$ be an algebraically closed field and $n$ a positive integer. All algebras are algebras over $k$. We denote by $D$ the duality \mbox{$D(-) = \Hom_k(-,k)$}. 

Notice that $A$ and $B$ always denote ungraded algebras, while the notation $\Lambda$ and $\Gamma$ is used for graded algebras. We work with right modules, homomorphisms act on the left of elements, and we write the composition of morphisms \mbox{$X \xrightarrow{f} Y \xrightarrow{g} Z$} as \mbox{$g \circ f$}. We denote by $\Modu A$ the category of $A$-modules and by $\modu A$ the category of finitely presented $A$-modules. The derived categories of $\Modu A$ and $\modu A$ are denoted by $\D(\Modu A)$ and $\D(\modu A)$, respectively.

We write the composition of arrows
\begin{tikzcd}
i \xrightarrow{\alpha}j \xrightarrow{\beta} k
\end{tikzcd}
in a quiver as $\alpha\beta$. In our examples, we use diagrams to represent indecomposable modules. This convention is explained in more detail in \cref{ex: section 6}.

Given a set of objects $\U$ in an additive category $\A$, we denote by $\add \U$ the full subcategory of $\A$ consisting of direct summands of finite direct sums of objects in $\U$. If $\A$ is triangulated, we use the notation $\Thick_{\A}(\U)$ for the smallest thick subcategory of $\A$ which contains $\U$. When it is clear in which category our thick subcategory is generated, we often omit the subscript $\A$.

Note that we have certain standing assumptions given at the beginning of \cref{Graded n-self-orth and n-T-Koszul} and \cref{T-Koszul and n-rep inf}.

\section{Preliminaries} \label{preliminaries}

In this section we begin by recalling some facts about graded algebras. Following this, we provide an introduction to a class of algebras which will be studied in \cref{T-Koszul and n-rep inf} and \cref{section almost n-T}, namely the graded Frobenius algebras. We finish by giving an overview of the notions of tilting subcategories and Serre functors, and discuss an equivalence which will be heavily used later on.

\subsection{Graded algebras, modules and extensions}\label{Graded algebras}
Consider a graded $k$-algebra \mbox{$\Lambda=\oplus_{i \in \Z} \Lambda_i$}. The category of graded $\Lambda$-modules and degree $0$ morphisms is denoted by $\Gr \Lambda$ and the subcategory of finitely presented graded $\Lambda$-modules by $\gr \Lambda$. Recall that $\gr \Lambda$ is abelian if and only if $\Lambda$ is graded right coherent, i.e.\ if every finitely generated homogeneous right ideal is finitely presented.

Given a graded module $M= \oplus_{i\in\Z}M_i$, we define the \textit{$j$-th graded shift} of $M$ to be the graded module $M\langle j \rangle$ with \mbox{$M\langle j \rangle_i = M_{i-j}$}. It should be noted that the graded shift $\langle 1 \rangle$ coincides with what is often denoted by $(-1)$ in the literature. The following basic result relates ungraded extensions to graded ones.

\begin{lem}[See \cite{Nastasescu-Van Oystaeyen}*{Corollary 2.4.7}] \label{Ext for graded introres.}
Let $M$ and $N$ be graded $\Lambda$-modules. If $M$ is finitely generated and there is a projective resolution of $M$ such that all syzygies are finitely generated, then
\[
\Ext_{\Lambda}^i(M,N) \simeq \bigoplus_{j\in \Z}\Ext_{\Gr \Lambda}^i(M,N\langle j \rangle)
\]
for all $i \geq 0$.
\end{lem}

A non-zero graded module $M = \oplus_{i\in\Z}M_i$ is said to be \textit{concentrated in degree $m$} if \mbox{$M_i  = 0$} for \mbox{$i \neq m$}. When $\Lambda$ is finite dimensional and $M$ finitely generated, there is an integer $h$ such that $M_h \neq 0$ and $M_i = 0$ for every $i > h$. We call $h$ the \textit{highest degree} of $M$. In the same way, the \textit{lowest degree} of $M$ is the integer $l$ such that $M_l \neq 0$ and $M_i = 0$ for every $i < l$. 

\subsection{Graded Frobenius algebras}

Recall that \textit{twisting} by a graded algebra automorphism $\phi$ of a graded algebra $\Lambda$ yields an autoequivalence $(-)_{\phi}$ on $\gr \Lambda$. Given $M$ in $\gr \Lambda$, the module $M_\phi$ is defined to be equal to $M$ as a vector space with right $\Lambda$-action \mbox{$m\cdot \lambda = m \phi(\lambda)$}, while $(-)_{\phi}$ acts trivially on morphisms. 

We make use of the following observation.
\begin{lem}\label{prop: twist of algebra iso as one-sided modules to algebra}
Let $\Lambda$ be a graded algebra with $\phi$ a graded algebra automorphism. 
Then one has an isomorphism $\Lambda_{\phi} \simeq \Lambda$ as right graded $\Lambda$-modules given by sending $\lambda \in \Lambda_{\phi}$ to $\phi^{-1}(\lambda)$. 
\end{lem}

The definition below plays an important role in this paper.

\begin{defin}
    A finite dimensional positively graded algebra $\Lambda$ is called \textit{graded Frobenius} if \mbox{$D\Lambda \simeq \Lambda \langle -a \rangle$} as both graded left and graded right $\Lambda$-modules for some integer $a$. 
\end{defin}

Notice that if $\Lambda$ in the definition above is concentrated in degree $0$, we recover the usual notion of a Frobenius algebra. Observe also that the integer $a$ in the definition must be equal to the highest degree of $\Lambda$, as \mbox{$(D\Lambda)_i = D(\Lambda_{-i})$}. We will usually assume $a \geq 1$.

Being graded Frobenius is equivalent to being Frobenius as an ungraded algebra and having a grading such that the socle is contained in the highest degree. 

\begin{lem}\label{graded frobenius lemma}
Let $\Lambda = \oplus_{i \geq 0}\Lambda_i$ be a finite dimensional algebra of highest degree $a$. 
The following statements are equivalent: 
\begin{enumerate}
    \item $\Lambda$ is graded Frobenius.
    \item There exists a graded automorphism $\mu$ of $\Lambda$ such that \mbox{${}_{1}\Lambda_{\mu}\langle -a \rangle \simeq D\Lambda$} as graded $\Lambda$-bimodules.
    \item $\Lambda$ is Frobenius as an ungraded algebra and satisfies \mbox{$\Soc \Lambda \subseteq \Lambda_a$}.
\end{enumerate}
\end{lem} 

\begin{proof}
If $\Lambda$ is graded Frobenius, \cite{Minamoto & Mori 2011}*{Lemma 2.9} implies that there exists a graded automorphism $\mu$ of $\Lambda$ such that 
\[
D\Lambda \simeq {}_{1}\Lambda_\mu \langle -a \rangle \simeq {}_{\mu^{-1}}\Lambda_1 \langle -a \rangle
\] 
as graded $\Lambda$-bimodules. It is hence clear that \textit{(1)} is equivalent to \textit{(2)}.

To see that \textit{(1)} is equivalent to \textit{(3)}, use that graded lifts of finite dimensional modules are unique up to isomorphism and graded shift \cite{Beilinson-Ginzburg-Soergel}*{Lemma 2.5.3} together with the fact that \mbox{$\Soc D\Lambda \subseteq (D\Lambda)_0$}.
\end{proof}

The automorphism $\mu$ of a graded Frobenius algebra $\Lambda$ as in the lemma above, is unique up to composition with an inner automorphism and is known as the \textit{graded Nakayama automorphism} of $\Lambda$. We call $\Lambda$ \textit{graded symmetric} if $\mu$ can be chosen to be trivial, and note that this notion also descends to the ungraded case. 

One class of examples which will be important for us, is that of trivial extension algebras. Recall that given a finite dimensional algebra $A$, the \textit{trivial extension of $A$} is \mbox{$\Delta A \coloneqq A \oplus DA$} as a vector space. The trivial extension is an algebra with multiplication \mbox{$(a, f) \cdot (b, g) = (ab, ag + fb)$} for \mbox{$a,b \in A$} and \mbox{$f,g \in DA$}. We consider $\Delta A$ as a graded algebra by letting $A$ be in degree $0$ and $DA$ be in degree $1$. Observe that $\Delta A$ is graded symmetric as it is symmetric as an ungraded algebra and satisfies \mbox{$\Soc \Delta A \subseteq (\Delta A)_1$}.

The stable category of finitely presented graded modules over a graded algebra $\Lambda$ is denoted by $\stgrmodu \Lambda$. If $\Lambda$ is self-injective, the category $\gr \Lambda$ is a Frobenius category, and $\stgrmodu \Lambda$ is triangulated with shift functor $\Omega^{-1}(-)$. Notice that every Frobenius algebra is self-injective. Observe that twisting by a graded automorphism $\phi$ of $\Lambda$ descends to an autoequivalence $(-)_{\phi}$ on $\stgrmodu \Lambda$. This functor commutes with taking syzygies and cosyzygies, as well as with graded shift.

We will often consider syzygies and cosyzygies of modules over self-injective algebras even when we do not work in a stable category. Whenever we do so, we assume having chosen a minimal projective or injective resolution, so that our syzygies and cosyzygies do not have any non-zero projective summands.
Because of our convention with respect to (representatives of) syzygies and cosyzygies, the notions of highest and lowest degree make sense for these too.  

Throughout the paper, we often need to consider basic degree arguments, as summarized in the following lemma. We include a short proof for the convenience of the reader.

\begin{lem} \label{lemma}
Let $\Lambda=\oplus_{i \geq 0}\Lambda_i$ be a finite dimensional self-injective graded algebra of highest degree $a$ and $\Soc\Lambda \subseteq \Lambda_a$. Consider $M,N,P \in \gr\Lambda$, where $P$ is indecomposable projective of highest degree $h$, and $M$ is of highest degree $h_M$ and lowest degree $l_M$. The following statements hold:
\begin{enumerate}
    \item Given any non-zero element $x \in P$, there exists $\lambda \in \Lambda$ such that \mbox{$x \lambda \in P_h$} is non-zero.
    \item For every non-zero morphism \mbox{$f \in \Hom_{\gr \Lambda}(M,P)$}, there exists an element \mbox{$x \in M$} such that \mbox{$f(x) \in P_h$} is non-zero. 
    \item Assume $a \geq 1$, and let $M$ and $N$ be concentrated in degree $0$. Then 
    \[
    \Hom_{\stgrmodu \Lambda}(M, N) \simeq \Hom_{\gr \Lambda}(M, N).
    \]  
    \item If $M$ is non-projective, then the lowest degree of $\Soc \Omega^{i}M$ is greater than or equal to $l_{M} + a$ for $i > 0$.
    \item If $M$ is non-projective, then the highest degree of $\Top \Omega^{i}M$ is less than or equal to $h_{M} - a$ for $i < 0$.
    \item If $M$ is non-projective, then the highest degree of $\Omega^{i}M$ is less than or equal to $h_{M}$ in the case $i\leq 0$ and greater than or equal to $l_{M} + a$ in the case $i > 0$. 
    \item Let $M$ be concentrated in degree $0$. Then 
    \[
    \Hom_{\gr \Lambda}(M,\Omega^iM\langle j \rangle) = 0
    \]
    for $i,j < 0$.
    \item Let $M$ be concentrated in degree $0$. Then 
    \[
    \Hom_{\gr \Lambda}(M,\Omega^iM\langle j \rangle) = 0
    \]
    for $i > 0$ and $j \geq 1 - a$. 
\end{enumerate}
\end{lem}

\begin{proof}
Since $\Soc \Lambda$ is concentrated in one degree and $P$ is indecomposable, we see that $\Soc P \subseteq P_h$ holds. Since $0 \neq \Soc x\Lambda \subseteq \Soc P \subseteq P_h$, we have $(x\Lambda)_h \neq 0$, as desired, and we thus obtain \textit{(1)}.

For \textit{(2)}, let $y \in M$ such that $f(y) \neq 0$. By \textit{(1)}, there exists an element \mbox{$\lambda \in \Lambda$} such that $f(y)\lambda \in P_h$ is non-zero. Consequently, the element \mbox{$x=y\lambda$} yields our desired conclusion.

To verify \textit{(3)}, notice that there can be no non-zero homomorphism \mbox{$M \rightarrow N$} factoring through a $\Lambda$-projective. Otherwise, one would have non-zero homomorphisms \mbox{$M \rightarrow \Lambda \langle i\rangle$} and \mbox{$\Lambda\langle i\rangle \rightarrow N$} for some integer $i$. The former is possible only if \mbox{$i = -a$} by \textit{(2)}. However, if $i=-a$, the latter is impossible as $\Lambda\langle -a\rangle$ is generated in degree $-a$. 

To show \textit{(4)}, we proceed as follows. Let $P$ be a projective cover of $M$. Since the lowest degree of $P$ is $l_M$, the lowest degree of \mbox{$\Soc P = \Soc \Omega M$} is \mbox{$l_M + a$}. Since $l_{\Omega M} \geq l_M$ holds, the claim follows by induction.

Since the argument for \textit{(4)} works just as well for left $\Lambda$-modules, we deduce \textit{(5)} from the left $\Lambda$-module version of \textit{(4)} by using the graded $k$-dual $D(-)$, that $D(M)_n = D(M_{-n})$, and that $D(\Soc D(M)) \simeq \Top M$ for any finitely generated graded right $\Lambda$-module $M$.

As the highest degree of $\Omega^iM$ is equal to the highest degree of $\Soc\Omega^iM$, we see from \textit{(4)} that the highest degree of $\Omega^iM$ is greater than or equal to $l_{M}+a$ for $i > 0$. Moreover, it follows from \textit{(5)} that the generators of $\Omega^iM$ are in degrees less than or equal to $h_{M}-a$ for $i < 0$, which means that the highest degree of $\Omega^iM$ is less than or equal to $h_{M}$. Since the claim in the case $i = 0$ is simply that $M$ has highest degree $h_{M}$, this finishes the proof of \textit{(6}).

Observe that \textit{(7)} is immediate in the case where $M$ is projective. Otherwise, note that the highest degree of $\Omega^{i}M$ is at most $0$ by \textit{(6)}. Hence, the highest degree of $\Omega^{i}M \langle j \rangle$ is less than or equal to $j$. As $j < 0$, this yields our desired conclusion.

For \textit{(8)}, it again suffices to consider the case where $M$ is non-projective. Applying \textit{(4)}, our assumptions yield that the lowest degree of $\Soc \Omega^{i}M \langle j \rangle$ is greater than or equal to $1$. By \textit{(2)}, this gives \mbox{$\Hom_{\gr \Lambda}(M,\Omega^iM\langle j \rangle) = 0$}, as syzygies are submodules of projectives. 
\end{proof} 

\subsection{Tilting subcategories, equivalences and Serre functors} \label{tilting section}

Tilting subcategories and the equivalences they provide play a crucial role throughout this paper. In this section we recall relevant notions and discuss an equivalence which will be heavily used in \cref{T-Koszul and n-rep inf} and \cref{section almost n-T}. We also describe the correspondence of Serre functors induced by this equivalence.

\begin{defin}\label{def: tilting subcat}
    Let $\T$ be a triangulated category. 
    A subcategory $\mathcal{M}$ of $\T$ is a \textit{tilting subcategory} if the following conditions hold:
\begin{enumerate}
    \item $\Hom_{\T}(M,M'[i])=0$ for $i \neq 0$ and $M, M' \in \mathcal{M}$.
    \item $\Thick_{\T}(\mathcal{M})=\T$.
\end{enumerate}
If $\mathcal{M} = \add M$ is a tilting subcategory for an object $M$ in $\T$, we say that $M$ is a \textit{tilting object}.
\end{defin}

The first condition in the definition above is often referred to as \textit{rigidity}. 

A triangulated category is called \textit{algebraic} if it is triangle equivalent to the stable category of a Frobenius category; for definitions, see e.g.\ \cite{Keller04}*{Section 3.6}. We now recall how essentially all algebraic triangulated categories can be described using dg-categories, and in this we follow closely the presentation in \cite{Keller04}*{Section 3.6}. However, since we only make use of dg-categories and the techniques of dg-homological algebra in this section and in \cref{Graded n-self-orth and n-T-Koszul}, we refer the reader to \cite{Keller} for an introduction to dg-homological algebra. Because of this choice, we have more or less adopted the notation of that source for the reader's convenience. In particular, recall from \cite{Keller} that given a dg-category $\A$, we define the category $\Hm^{0} \A$ (resp.\ $\Hm^{*} \A$) to have the same objects as $\A$ and morphisms given by taking the $0$-th cohomology (resp.\ the cohomology) of the morphism spaces in $\A$. Similarly, the category $\tau_{\leq 0} \A$ has the same objects as $\A$, and morphisms given by taking subtle truncation. Recall also the definition of the dg-category $\Dif \mathcal{A}$, which we denote instead by $C_{\dg}(\mathcal{A})$. If $\mathcal{A}$ is an ordinary algebra concentrated in cohomological degree $0$, the objects of the category $C_{\dg}(\mathcal{A})$ are complexes of modules over $\mathcal{A}$, and the morphisms are given by homogeneous maps which do not necessarily respect the differentials. 
We also recall that $\D (\A)$ denotes the derived category \mbox{of $\A$}. Note especially that we write $\D^{\perf}(\A)$ for the perfect derived category of $\A$ and that our choice of notation here differs from \cite{Keller04}.

The following is a special case of Keller's Morita theorem for algebraic triangulated categories. 

\begin{thm}\label{thm: keller's morita thm for algebraic triangulated categories}
\cite{Keller04}*{Theorem 3.8 b)}
For an idempotent complete algebraic triangulated category $\mathcal{T}$ with a full subcategory $\mathcal{M}$ satisfying $\Thick \mathcal{M} = \mathcal{T}$, one can choose a pretriangulated dg-category $\mathcal{A}$ with a full dg-subcategory $\mathcal{B}$ such that $\mathcal{T} \simeq \Hm^0 (\mathcal{A})$ and $\mathcal{M} \simeq \Hm^0(\mathcal{B})$. 
Then the dg-functor $\mathcal{A} \rightarrow C_{\dg} (\mathcal{B})$ given by $X \mapsto \mathcal{A}(\mathcal{B},X)$ gives a triangle equivalence $\mathcal{T} \simeq \D^{\perf}(\mathcal{B})$.
If $\mathcal{M}$ is a tilting subcategory, then we have quasi-equivalences
\[
\begin{tikzcd}
\mathcal{M} \simeq \Hm^{0} \mathcal{B} & \arrow[two heads]{l} \tau_{\leq 0} \mathcal{B} \arrow[hookrightarrow]{r} & \mathcal{B}
\end{tikzcd}
\]
which give triangle equivalences $\D^{\perf}( \mathcal{M}) \simeq \D^{\perf}(\Hm^0 \mathcal{B}) \simeq \mathcal{T}$.
\end{thm}

Recall that when $\Lambda$ is a self-injective graded algebra, the category $\gr \Lambda$ is Frobenius, and consequently the stable category $\stgrmodu \Lambda$ is an algebraic triangulated category. By \cref{thm: keller's morita thm for algebraic triangulated categories}, we hence know that if $T$ is a tilting object in $\stgrmodu \Lambda$ and \mbox{$B=\End_{\stgrmodu \Lambda}(T)$} has finite global dimension, then there exists a triangle equivalence \mbox{$G \colon \D^b(\modu B) \rightarrow \stgrmodu \Lambda$} given by the quasi-inverse of the equivalence obtained via \cref{thm: keller's morita thm for algebraic triangulated categories}. In \cref{T-Koszul and n-rep inf} and \cref{sec 7}, we will use that projective $B$-modules correspond to summands of $T$ under this equivalence, as described in \cref{eksplisitt ekvivalens} below. This fact follows from the construction of the equivalence in \cref{thm: keller's morita thm for algebraic triangulated categories}. 

Given a decomposition $T \simeq \oplus_{i=1}^t T^i$ of $T$, we let $e_i \colon T \twoheadrightarrow T^i \hookrightarrow T$ denote the $i$-th projection followed by the $i$-th inclusion. This yields a decomposition \mbox{$B \simeq \oplus_{i=1}^t P^i$} of $B$ into projectives $P^i = e_i B$. Note moreover that if \mbox{$\gldim \Lambda_0 < \infty$}, then \mbox{$\gldim B < \infty$} by \cite{Yamaura 2013}*{Corollary 3.12}.

\begin{prop} \label{eksplisitt ekvivalens}
Let $\Lambda$ be finite dimensional self-injective graded algebra and assume that \mbox{$\gldim \Lambda_0 < \infty$}. Consider a tilting object $T$ in $\stgrmodu \Lambda$ and denote its endomorphism algebra by \mbox{$B = \End_{\stgrmodu \Lambda}(T)$}. Then the equivalence
\[
G \colon \D^b(\modu B) \rightarrow \stgrmodu \Lambda
\]
satisfies $G(e_i B) \simeq T^i$.
\end{prop}

From \cref{T-Koszul and n-rep inf} and on, the following notion will be crucial.

\begin{defin}
Let $\mathcal{T}$ be a $k$-linear $\Hom$-finite triangulated category. An additive autoequivalence $\mathcal{S}$ on $\mathcal{T}$ is called a \textit{Serre functor} provided there exists a bifunctorial isomorphism
\[
\Hom_{\mathcal{T}}(X, Y) \simeq D\Hom_{\mathcal{T}}(Y, \mathcal{S}X)
\]
for all objects $X$ and $Y$ in $ \mathcal{T}$.
\end{defin}

We want to compare the Serre functor on $\D^b (\modu B)$ to that of $\stgrmodu \Lambda$ when $\Lambda$ is a graded Frobenius algebra of highest degree $a$ with Nakayama automorphism $\mu$. In this case, it follows from Auslander--Reiten duality, see \cite{Auslander-Reiten-Smalo} and \cite{Reiten-Van den Bergh}*{Proposition I.2.3}, combined with the characterization in \cref{graded frobenius lemma} that $\Omega(-)_{\mu}\langle -a \rangle$ is a Serre functor on $\stgrmodu \Lambda$. As $B$ is a finite dimensional algebra of finite global dimension, the derived Nakayama functor \mbox{$\nu(-) = - \otimes^{\mathbf{L}}_{B} DB$} is a Serre functor on $\D^b (\modu B)$. By uniqueness of the Serre functor, the equivalence \mbox{$G \colon \D^b(\modu B) \rightarrow \stgrmodu \Lambda$} yields a commutative diagram
\begin{center}\label{correspondence of Serre functors}
    \begin{tikzcd}[column sep=22, row sep=20]
   \D^b(\modu B) \arrow[r,"G"] \arrow[d,"\nu"] & \stgrmodu \Lambda \arrow[d,"\Omega(-)_{\mu}\langle -a\rangle"]  \\
    \D^b(\modu B) \arrow[r,"G"] & \stgrmodu \Lambda.
    \end{tikzcd}    
\end{center}

Note that throughout the rest of this paper, we often use the triangle equivalence \mbox{$G \colon \D^b(\modu B) \rightarrow \stgrmodu \Lambda$} and the correspondence of the Serre functors described in the diagram above without making the reference explicitly.

\section{Higher Koszul duality} \label{Graded n-self-orth and n-T-Koszul}

Throughout the rest of this paper, let $\Lambda = \oplus_{i \geq 0} \Lambda_i$ be a positively graded algebra, where $\Lambda_0$ is a finite dimensional basic algebra. We assume that $\Lambda$ is locally finite dimensional, i.e.\ that $\Lambda_i$ is finite dimensional as a vector space over $k$ for each $i \geq 0$.

In this section we define more flexible notions of what it means for a module $T$ to be graded self-orthogonal and an algebra to be $T$-Koszul than the ones introduced by Madsen \cite{Madsen 2011}*{Definition 3.1.1 and 4.1.1}. This enables us to talk about $T$-Koszul duality for a more general class of algebras. In particular, we obtain a higher Koszul duality equivalence in \cref{Existence of T-Koszul equivalence} and we recover a higher version of the BGG-correspondence in \cref{answer motivating question}. Note that the ideas in this section are similar to the ones in \cite{Madsen 2011}. For the convenience of the reader, we nevertheless give concise proofs of this section's main results, to show that the arguments work also in our generality. 

In order to state our main definitions, let us first recall the notion of a tilting module. 

\begin{defin}\label{tilting module def}
Let $A$ be a finite dimensional algebra. A finitely generated $A$-module $T$ is called a \textit{tilting module} if its projective resolution is a tilting object in $\D^{\perf}(A)$. 

In other words, $T$ is a tilting module if it satisfies the following conditions: 
\begin{enumerate}
    \item $\projdim_A T < \infty$;
    \item $\Ext^i_A(T,T)=0$ for $i>0$;
    \item There is an exact sequence
    \[
    0 \rightarrow A \rightarrow T^0 \rightarrow T^1 \rightarrow \cdots \rightarrow T^l \rightarrow 0
    \]
    with $T^i \in \add T$ for $i=0,\dots,l$.
\end{enumerate}
\end{defin}

We now define what it means for a module to be graded $n\mathbb{Z}$-orthogonal. 

\begin{defin}\label{def:gr so}
	Let $T$ be a finitely generated basic graded $\Lambda$-module concentrated in degree $0$. We say that $T$ is \textit{graded $n\mathbb{Z}$-orthogonal} if 
	\[
    \Ext^{i}_{\gr \Lambda}(T,T\langle j \rangle) = 0 
    \]
	for $i \neq nj$.
\end{defin}

Notice that the definition of being graded $n\mathbb{Z}$-orthogonal is more general than the notion of graded self-orthogonality given in \cite{Madsen 2011}. More precisely, the two definitions coincide exactly when $n$ is equal to $1$. In this case, examples of graded $n\mathbb{Z}$-orthogonal modules are given by $\Lambda_0$ in the classical Koszul situation or tilting modules if $\Lambda=\Lambda_0$. Moreover, we see in \cref{T-Koszul and n-rep inf} that $n$-representation infinite algebras provide examples of modules which are graded $n\mathbb{Z}$-orthogonal for any choice of $n$.

In general, a graded $n\mathbb{Z}$-orthogonal module might have syzygies which are not finitely generated, so \cref{Ext for graded introres.} does not apply. However, the following proposition gives a similar result for graded $n\mathbb{Z}$-orthogonal modules. This is an analogue of \cite{Madsen 2011}*{Proposition 3.1.2}. The proof is exactly the same, except that we use our more general definition of graded $n\mathbb{Z}$-orthogonality. 

\begin{prop} \label{Ext for T graded self-orthogonal}
Let $T$ be a graded $n\mathbb{Z}$-orthogonal $\Lambda$-module. Then
    \[
    \Ext_{\Lambda}^{ni}(T,T) \simeq \Ext_{\gr \Lambda}^{ni}(T,T\langle i \rangle)
    \]
    for all $i \geq 0$.
\end{prop}

Using our definition of a graded $n\mathbb{Z}$-orthogonal module $T$, we also get a more general notion of what it means for an algebra to be Koszul with respect to $T$.  

\begin{defin}
	Assume $\gldim \Lambda_0 < \infty$ and let $T$ be a graded $\Lambda$-module concentrated in degree $0$. We say that $\Lambda$ is \textit{$n$-$T$-Koszul} or \textit{$n$-Koszul with respect to $T$} if the following conditions hold:
	\begin{enumerate}
		\item $T$ is a tilting $\Lambda_0$-module.
		\item $T$ is graded $n\mathbb{Z}$-orthogonal as a $\Lambda$-module.
	\end{enumerate}
\end{defin}

Note that an algebra is $T$-Koszul in the sense of \cite{Madsen 2011} if and only if it is $n$-$T$-Koszul for $n=1$. The following two remarks each discuss an aspect of how this definition relates to the corresponding one in \cite{Madsen 2011}.

\begin{remark}\label{rem: on regrading}
If $\Lambda = \oplus_{i\geq 0} \Lambda_{i}$ is $n$-$T$-Koszul, one can rescale the grading of $\Lambda$ so that the regraded algebra $\Lambda^{\rho}$ is $T$-Koszul in the sense of \cite{Madsen 2011}*{Definition 4.1.1} by defining
$\Lambda^{\rho}_{i} = \Lambda_{j}$ if $i = nj$ for some integer $j$ and $\Lambda^{\rho}_{i} = 0$ otherwise. Then one also obtains that the category $\gr \Lambda$ embeds into $\gr \Lambda^{\rho}$ as the full subcategory consisting of modules which are non-zero only in degrees multiples of $n$. 

Note, however, that we cannot always work directly with the regraded algebras. For instance, the category $\gr \Lambda^{\rho}$ is `too big' for the motivating question in the introduction since regrading $\Delta \coloneqq \Delta A$ by putting $DA$ in degree $a > 1$ yields that $\stgrmodu \Delta^{\rho}$ is \textit{not} equivalent to $\D^b(\modu A)$ by e.g.\ \cite{Yamaura 2013}*{Proposition 3.11}.
\end{remark}

\begin{remark} 
In \cref{def:gr so} we require a graded $n\mathbb{Z}$-orthogonal module to be basic for consistency with \cite{Madsen 2011}. Consequently, we later assume that certain algebras are basic, for instance in \cref{motivating result}. Note that usually this is of limited importance for our proofs.
\end{remark}

Like in the classical theory, we want a notion of a Koszul dual of a given $n$-$T$-Koszul algebra. 

\begin{defin}
Let $\Lambda$ be an $n$-$T$-Koszul algebra. The \textit{$n$-$T$-Koszul dual of $\Lambda$} is given by $\Lambda^! =  \oplus_{i \geq 0} \Ext_{\gr \Lambda}^{ni}(T,T\langle i \rangle)$.
\end{defin}

Note that while the notation for the $n$-$T$-Koszul dual is potentially ambiguous, it will in this paper always be clear from context which $n$-$T$-Koszul structure the dual is computed with respect to.

By \cref{Ext for T graded self-orthogonal}, we get the following equivalent description of the $n$-$T$-Koszul dual.

\begin{cor} \label{cor: n-T-Koszul dual}
Let $\Lambda$ be an $n$-$T$-Koszul algebra. Then there is an isomorphism of graded algebras $\Lambda^! \simeq  \oplus_{i \geq 0}\Ext_{\Lambda}^{ni}(T,T)$.
\end{cor}

The next result shows that the $n$-$T$-Koszul property behaves well with respect to tensor products.

\begin{prop}\label{prop: n-T-Koszul closed under tensor products}
If $\Lambda^{i}$ is $n$-$T^{i}$-Koszul for $1 \leq i \leq m$, then $\Lambda^1 \otimes_k \cdots \otimes_k \Lambda^m$ is $n$-$T$-Koszul for $T \coloneqq T^1 \otimes_k \cdots \otimes_k T^m$.
\end{prop}

\begin{proof}
By induction, it suffices to show this for $m = 2$. 
Let thus $\Lambda \coloneqq \Lambda^1 \otimes_k \Lambda^2$.  
It will be useful to apply a K\"{u}nneth formula for $\Ext_{\gr \Lambda^{\rho}}$, where $\Lambda^{\rho}$ is a regraded version of $\Lambda$ endowed with a $\mathbb{Z}\times \mathbb{Z}$-grading. The grading of $\Lambda^{\rho}$ is given by letting \mbox{$\Lambda^{\rho}_{(i,j)} = \Lambda^1_i \otimes_k \Lambda^2_j$}. We now have
\[
\Ext^{i}_{\gr \Lambda^{\rho}} (T^1 \otimes_k T^2, T^1 \otimes_k T^2 \langle (j,k) \rangle)
\simeq 
\bigoplus_{i_1 + i_2 = i}
\Ext^{i_1}_{\gr \Lambda^1} (T^1, T^1 \langle j \rangle)
\otimes_k 
\Ext^{i_2}_{\gr \Lambda^2} (T^2, T^2 \langle k \rangle).
\]
Combining this with $T^{i} \in \gr \Lambda^{i}$ being graded $n \mathbb{Z}$-orthogonal, we obtain 
\[
\Ext^{i}_{\gr \Lambda} (T^1 \otimes_k T^2, T^1 \otimes_k T^2 \langle l \rangle)
\simeq 
\bigoplus_{j + k = l}
\Ext^{i}_{\gr \Lambda^{\rho}} (T^1 \otimes_k T^2, T^2 \otimes_k T^2 \langle (j,k) \rangle) = 0
\]
for $i \neq nl$, so $\Lambda$ is $n$-$T$-Koszul.
\end{proof}

Given a set of objects $\U \subseteq \D^b(\gr \Lambda)$, let $\Thicks(\U)$ denote the smallest thick subcategory of $\D^b(\gr \Lambda)$ which contains $\U$ and is closed under graded shift $\langle - \rangle$. Using that $\Lambda_0$ has finite global dimension and that $T$ is a tilting $\Lambda_0$-module, one can see that $T$ generates the entire bounded derived category of $\gr \Lambda$ whenever $\Lambda$ is a finite dimensional $n$-$T$-Koszul algebra. 

\begin{lem}\label{lem:T generates}
Let $\Lambda$ be a finite dimensional $n$-$T$-Koszul algebra. We then have \mbox{$\Thicks(T) = \D^b(\gr \Lambda)$}.
\end{lem}

\begin{proof}
Since $T$ is a tilting module over $\Lambda_0$, and $\Lambda_0 \langle i \rangle$ thus has a finite coresolution in $\add T\langle i \rangle$, we deduce that $\Lambda_0\langle i \rangle$ is in $\Thicks(T)$ for every $i \in \mathbb{Z}$.

Notice now that every simple graded $\Lambda$-module is concentrated in degree $i$ for some $i \in \mathbb{Z}$ and is hence necessarily contained in the thick subcategory generated by $\Lambda_0 \langle i \rangle$. To see this, apply $\langle i \rangle$ to a finite $\Lambda_0$-projective resolution of such a module, split up into short exact sequences and use that thick subcategories have the $2/3$-property on distinguished triangles. We can thus conclude that \mbox{$\Thick(\{\Lambda_0\langle i \rangle \}_{i \in \mathbb{Z}}) = \D^b(\gr \Lambda)$}, which finishes the proof.
\end{proof}

We are now ready to state and prove the main result of this section, namely to show that we obtain a higher Koszul duality equivalence. This recovers and strengthens \cite{Madsen 2011}*{Theorem 4.3.4} in the case where $n=1$ and is a version of \mbox{\cite{Beilinson-Ginzburg-Soergel}*{Theorem 2.12.6}} in the classical Koszul case. Note that following a suggestion of Bernhard Keller, we prove that $\Lambda^!$ has finite global dimension. We do hence not need to assume this as in \cite{Madsen 2011}*{Theorem 4.3.4}.

\begin{thm} \label{Existence of T-Koszul equivalence}
Let $\Lambda$ be a finite dimensional $n$-$T$-Koszul algebra. The following statements hold:
\begin{enumerate}
    \item $\U = \{T\langle i \rangle [ni] \mid i \in \Z \}$ is a tilting subcategory of $\D^b(\gr \Lambda)$, and we have $\D^b(\gr \Lambda) \simeq \D^{\perf}( \Gr \Lambda^!)$ as triangulated categories.
    \item $\Lambda^!$ has finite global dimension.  
    \item  If $\Lambda^!$ is also graded right coherent, there is a triangle equivalence 
    \[
    K \colon \D^b(\gr\Lambda) \xrightarrow{\simeq} \D^b(\gr\Lambda^!).
    \]
\end{enumerate}
\end{thm}

\begin{proof}[Proof of \cref{Existence of T-Koszul equivalence}]
Let us first show \textit{(1)}. We have
\begin{align*}
\Hom_{\D^b(\gr \Lambda)}(T \langle i \rangle [ni], T \langle j \rangle [nj + k]) 
& \simeq 
\Hom_{\D^b(\gr \Lambda)}(T , T \langle j - i \rangle [n(j-i) + k])\\
& \simeq \Ext_{\gr \Lambda}^{n(j-i) + k}(T,T \langle j - i \rangle) = 0
\end{align*}
for $k \neq 0$ as $T$ is graded $n\mathbb{Z}$-orthogonal. Moreover, \cref{lem:T generates} yields 
\[
\Thick (\U) = \Thicks(T) = \D^b(\gr\Lambda).
\]
Combining this shows that $\U$ is a tilting subcategory of $\D^b(\gr \Lambda)$. 

We next want to apply \cref{thm: keller's morita thm for algebraic triangulated categories}. Let $\grproj \Lambda$ denote the category of finitely generated projective graded $\Lambda$-modules with homogeneous morphisms of \mbox{degree $0$}. Recall that $C^{-,b}_{\dg}(\grproj \Lambda)$ is the the dg-category consisting of right bounded complexes over $\grproj \Lambda$ with bounded cohomology. In particular, its morphism spaces are given by all homogeneous maps of complexes that are also homogeneous of degree $0$ with respect to the grading of $\Lambda$. We now set $\mathcal{A}$ to be \mbox{$C^{-,b}_{\dg}(\grproj \Lambda)$} and $\mathcal{B}$ to be the full dg-subcategory of $\mathcal{A}$ given by $\{P \langle i \rangle [ni] \mid i \in \Z \}$, where $P$ is some graded projective resolution of $T$. 

Since $\U \simeq \Hm^0 \B$ is a tilting subcategory of $\D^b(\gr \Lambda$) by \textit{(1)}, we have a triangle equivalence $\D^b(\gr \Lambda) \simeq \D^{\perf}(\Hm^0 \B)$ by \cref{thm: keller's morita thm for algebraic triangulated categories}.
On the other hand, since 
\[
\bigoplus_{i \in \mathbb{Z}} \Hom_{\Hm^{0}\mathcal{B}}(P, P \langle i \rangle [ni])\simeq \bigoplus_{i \in \mathbb{Z}} \Ext_{\gr \Lambda}^{ni}(T,T\langle i \rangle) = \Lambda^!,
\]
we get an equivalence $\add \Hm^0\B \simeq \grproj \Lambda^!$ by using \cite{Hanihara}*{Proposition 4.2} with $\mathbb{Z}$ acting on $\add \Hm^0\B$ via the functor $(-)\langle 1 \rangle [n]$. 
This equivalence then induces a triangle equivalence $\D^{\perf}(\Hm^0 \B) \simeq \D^{\perf}(\Gr \Lambda^!)$.  
In combination, this yields that
\[
\D^b(\gr \Lambda) \simeq \D^{\perf}(\Hm^0 \B) \simeq \D^{\perf}( \Gr \Lambda^!),
\]
which finishes the proof of \textit{(1)}.

For \textit{(2)}, note that by \cite{Grant-Iyama}*{Theorem A.1} it is sufficient to demonstrate that $\Lambda^!$ has finite global dimension as a graded algebra. We do this by showing that $\Lambda^!$ is \textit{smooth} as a graded algebra, i.e.\ that $\Lambda^!$ is perfect as a $\Lambda^!$-bimodule.

Recall now that $\Lambda^{e} \coloneqq \Lambda^{\op} \otimes_k \Lambda$ is the enveloping algebra of $\Lambda$, and note that we grade $\Lambda^{e}$ over $\mathbb{Z}$ by setting $\Lambda^{e}_{i} \coloneqq \oplus_{j + k = i}\Lambda_{j} \otimes_k \Lambda_{k}$.
Moreover, recall that \mbox{($k$-symmetric)} $\mathbb{Z}$-graded $\Lambda$-bimodules correspond to $\mathbb{Z}$-graded $\Lambda^{e}$-modules.
Observe that $\Lambda^{e}$ is $n$-Koszul with respect to $DT \otimes_k T$ by \cref{prop: n-T-Koszul closed under tensor products} and that $\Lambda^{!,e} \coloneqq (\Lambda^!)^{e}  \simeq (\Lambda^!)^{e}$, where the 
isomorphism follows by the proof of \cref{prop: n-T-Koszul closed under tensor products}.
With this in mind, we see that \cref{prop: n-T-Koszul closed under tensor products} and \textit{(1)} together imply that 
\[
\mathcal{W} := \{DT \otimes_kT \langle i \rangle [ni] \mid i \in \Z \}
\]
is a tilting subcategory of $\D^b(\gr \Lambda^{e})$
and that 
\mbox{$\D^b(\gr \Lambda^{e}) \simeq \D^{\perf}( \Gr \Lambda^{!,e})$}.
Consulting the proof of \textit{(1)} above, we see that the functor $F$ giving this equivalence can be written as follows
\[
F(-) = \bigoplus_{i \in \mathbb{Z}}\RHom_{\Gr \Lambda^{e}}(DT \otimes_kT, -  \langle i \rangle [ni]).
\]
Using this, one can 
deduce that $\Lambda^{!}$ is in $\D^{\perf}(\Gr \Lambda^{!, e})$ via the following computation, which we note is a variation on an argument from Theorem 2.15 and Theorem 3.7 of \cite{LS}:
\begin{align*}
F(D(\Lambda)) 
& = \bigoplus_{i \in \mathbb{Z}}\RHom_{\Gr \Lambda^{e}}(DT \otimes_kT, D(\Lambda)  \langle i \rangle [ni]) \\
& \simeq 
\bigoplus_{i \in \mathbb{Z}}\RHom_{\Gr \Lambda}(T, \RHom_{\Gr \Lambda^{\op}} (DT, D(\Lambda)  \langle i \rangle [ni]))\\
& \simeq 
\bigoplus_{i \in \mathbb{Z}}\RHom_{\Gr \Lambda}(T, \RHom_{\Gr \Lambda} (\Lambda \langle -i \rangle [-ni], T))\\
& \simeq 
\bigoplus_{i \in \mathbb{Z}}\RHom_{\Gr \Lambda}(T, T \langle i \rangle [ni])\\
& \simeq \Lambda^!
\end{align*}

Consequently, we can conclude that $\Lambda^!$ has finite global dimension as a graded algebra by \cite{L}*{Lemma 3.6}, which finishes the proof of \textit{(2)}. 

In \textit{(3)} we have assumed that $\Lambda^!$ is graded right coherent. 
As $\Lambda^!$ has finite global dimension as a graded algebra, we thus have a triangle equivalence 
\[
\D^{\perf}(\Gr \Lambda^!) \xrightarrow{\simeq} \D^b(\gr \Lambda^!).
\]
Composing this with the equivalence established in \textit{(1)} finishes the construction of $K$ and the proof of \textit{(3)}.
\end{proof}

\begin{remark}\label{rem: another proof}
It is also possible to derive the equivalence in \cref{Existence of T-Koszul equivalence} \textit{(3)} by regrading the algebras involved as in \cref{rem: on regrading} and tracking our original (derived) categories of graded modules through the equivalence in \cite{Madsen 2011}*{Theorem 4.3.4}. Proceeding in this way, one can also recover generalized analogues of many of the results \mbox{in \cite{Madsen 2011}}.

To see how to get this alternative proof, we begin by noting that since the embedding in \cref{rem: on regrading} is exact, it induces a triangulated functor between the corresponding derived categories. By \cite{Stacks project}*{Lemma 13.17.4}, this functor yields an equivalence \mbox{$\D^b(\gr \Lambda) \xrightarrow{\simeq} \D^b_{\gr \Lambda}(\gr \Lambda^{\rho})$}, where $\D^b_{\gr \Lambda}(\gr \Lambda^{\rho})$ denotes the full subcategory of $\D^b(\gr \Lambda^{\rho})$ consisting of objects with cohomology in $\gr \Lambda$.

Using that $\Lambda^{\rho}$ is $T$-Koszul and noticing that $(\Lambda^!)^{\rho} \simeq (\Lambda^{\rho})^!$, we get by \cite{Madsen 2011}*{Theorem 4.3.4} the equivalence in the upper row of the diagram  
\[
  \begin{tikzcd}[column sep=22, row sep=20]
  & \D^b(\gr \Lambda^{\rho}) \arrow[r,"\simeq"] & \D^b(\gr (\Lambda^!)^{\rho})\\
  \D^b(\gr \Lambda) \arrow[r,"\simeq"] &
  \D^b_{\gr \Lambda}(\gr \Lambda^{\rho}) \arrow[r,dashed,"\simeq"] \arrow[hookrightarrow]{u} & \D^b_{\gr \Lambda^!}(\gr (\Lambda^!)^{\rho}) \arrow[hookrightarrow]{u} & \D^b(\gr \Lambda^!). \arrow{l}[swap]{\simeq}
  \end{tikzcd}
\]
In order to deduce \cref{Existence of T-Koszul equivalence} \textit{(3)}, it now suffices to observe that this equivalence restricts to an equivalence as indicated by the dashed arrow.
\end{remark}

Using \cref{rem: another proof}, we now obtain a generalization of a result from \cite{Madsen 2011}. 

\begin{thm}\label{prop: Koszul dual is n-DT-Koszul}
If $\Lambda$ is $n$-$T$-Koszul, then $\Lambda^!$ is $n$-$DT$-Koszul. 
\end{thm}

\begin{proof}
Since shifting by $1$ in $\gr \Lambda$ corresponds to shifting by $n$ in $\gr \Lambda^{\rho}$, the argument in \cref{rem: another proof} together with \cite{Madsen 2011}*{Theorem 4.2.1 (a)} is sufficient.
\end{proof}

Recall that we denote by \mbox{$K \colon \D^b(\gr \Lambda) \rightarrow \D^b(\gr \Lambda^!)$} the equivalence from \cref{Existence of T-Koszul equivalence}. 

\begin{prop}\label{K and int deg prop}
Let $\Lambda$ be a finite dimensional $n$-$T$-Koszul algebra and assume that $\Lambda^!$ is graded right coherent. The following statements hold:
\begin{enumerate}[(1)]
    \item We have $K(M\langle i \rangle) = K(M) \langle -i \rangle [-ni]$ for $M \in \D^{b}(\gr \Lambda)$.
    \item We have $K(D\Lambda) \simeq DT$, where the $\Lambda^!$-module structure on $DT$ is induced by $\Lambda^!_0 \simeq \End_{\gr \Lambda}(T) \simeq \End_{\Lambda_0}(T)$ acting on $T$ on the left by endomorphisms. 
\end{enumerate}
\end{prop}

\begin{proof}
Note that \textit{(1)} follows from \cite{Madsen 2011}*{Proposition 3.2.1 (c)} by using \cref{rem: another proof} as in \cref{prop: Koszul dual is n-DT-Koszul}.

Hence, we now show \textit{(2)}. 
For this, we assume the notation and setup used in the proof of \cref{Existence of T-Koszul equivalence}, that is, \mbox{$\mathcal{A} = C^{-,b}_{\dg}(\grproj \Lambda) \supset \mathcal{B} = \{P\langle i \rangle [ni] \, \lvert \, i \in \mathbb{Z}\}$}. Then the dg functor $\mathcal{A} \rightarrow C_{\dg}(\mathcal{B}), A \mapsto \mathcal{A}(\mathcal{B}, A)$ given in \cref{thm: keller's morita thm for algebraic triangulated categories} induces an equivalence
\[
F \colon \D^b(\gr \Lambda) \rightarrow \D^{\perf}(\B), ~X \mapsto \bigoplus_{i \in \mathbb{Z}}\mathcal{A}(P, pX\langle i \rangle [ni]),
\]
where $pX$ is a graded projective resolution of $X$. For each $j \in \mathbb{Z}$, we have $(\Hm^j \mathcal{A})(P, pX\langle i \rangle [ni]) \simeq \Ext^{ni + j}_{\gr \Lambda}(T, X\langle i \rangle)$. Thus, \mbox{$\mathcal{A}$$(P, p(D\Lambda)\langle i \rangle [ni])$} is acyclic if $i \neq 0$, and quasi-isomorphic to $\Hom_{\gr \Lambda}(T, D\Lambda) \simeq DT$ if \mbox{$i = 0$}. This yields $F(D\Lambda) \simeq DT$. 

Recall that in \cref{thm: keller's morita thm for algebraic triangulated categories}, one uses the zig-zag of dg-categories
\[
\begin{tikzcd}
\Hm^{0} \mathcal{B} & \arrow[two heads]{l} \tau_{\leq 0} \mathcal{B} \arrow[hookrightarrow]{r} & \mathcal{B}
\end{tikzcd}
\]
to induce the equivalence $\D^{\perf}(\Hm^{0} \mathcal{B}) \simeq \D^{\perf}(\mathcal{B})$ used in the construction of $K$. Chasing $F(D\Lambda)$ through the equivalences induced by the zig-zag above, we notice that this stalk complex has the $\Lambda^!$-action one expects, i.e.\ the action induced by \mbox{$\Lambda^!_0 \simeq \End_{\gr \Lambda}(T) \simeq \End_{\Lambda_0}(T)$} acting on $T$ on the left by endomorphisms. This shows \textit{(2)}, and we are done.
\end{proof}

We finish this section by showing that an analogue of the BGG-correspondence holds in our generality. Recall that $\qgr \Lambda^!$ is defined as the localization of $\gr \Lambda^!$ at the full subcategory of finite dimensional graded $\Lambda^!$-modules. We hence have a natural functor \mbox{$\D^b(\gr \Lambda^!) \rightarrow \D^b(\qgr \Lambda^!)$}. In the case where $\Lambda$ is graded Frobenius, there is a well-known equivalence \mbox{$\D^b(\gr \Lambda) / \D^{\perf}(\gr \Lambda) \simeq \stgrmodu \Lambda$} \cite{Buchweitz}*{Theorem 4.4.1}, \cite{Rickard}*{Theorem 2.1}. One consequently obtains a functor 
\[
\D^b(\gr \Lambda) \rightarrow \D^b(\gr \Lambda) / \D^{\perf}(\gr \Lambda) \xrightarrow{\simeq} \stgrmodu \Lambda.
\]
These two functors give the vertical arrows in the diagram in our proposition below.

\begin{prop} \label{answer motivating question}
Let $\Lambda$ be a finite dimensional $n$-$T$-Koszul algebra and assume that $\Lambda^!$ is graded right coherent. If $\Lambda$ is graded Frobenius, then the equivalence $K$ descends to yield \mbox{$\stgrmodu \Lambda \simeq \D^b(\qgr \Lambda^!)$}, as indicated in the following diagram
\[
    \begin{tikzcd}[column sep=22, row sep=20] \D^b(\gr \Lambda) \arrow[r,"K"] \arrow[d] & \D^b(\gr \Lambda^!) \arrow[d]  \\
    \stgrmodu \Lambda \arrow[r,dashed,"\simeq"] & \D^b(\qgr \Lambda^!).
    \end{tikzcd}
\]
\end{prop}

\begin{proof}
It suffices to prove that $K$ restricts to an equivalence between $\D^{\perf}(\gr \Lambda)$ and $\D^{\fd}(\gr \Lambda^!)$, where $\D^{\fd}(\gr \Lambda^!)$ denotes the full subcategory of $\D^b(\gr \Lambda^!)$ consisting of objects with finite dimensional total cohomology. 
In fact, we have equivalences
\[
\stgrmodu \Lambda \simeq \D^b(\gr \Lambda)/\D^{\perf}(\gr \Lambda)
\stackrel{K}{\simeq}
\D^b(\gr \Lambda^!)/\D^{\fd}(\gr \Lambda^!) 
\simeq
\D^b(\qgr \Lambda^!),
\]
where the first equivalence is \cite{Buchweitz}*{Theorem 4.4.1} and the last one is \ \cite{Stacks project}*{Lemma 13.17.3}. 

Hence, we now show that $\D^{\fd}(\gr \Lambda^{!}) = K(\D^{\perf}(\gr \Lambda))$. We begin by noting that \mbox{$\Thicks (D\Lambda) = \D^{\perf}(\gr \Lambda)$}. Since \mbox{$K(D \Lambda \langle i \rangle) \simeq DT\langle -i \rangle [-ni]$} by \cref{K and int deg prop}, we get that $K$ restricts to an equivalence \mbox{$\Thicks (D\Lambda) \xrightarrow{\simeq} \Thicks(DT)$}. As tilting theory implies that $DT$ is a tilting module over $\End_{\Lambda_0}(T)$, one deduces that \mbox{$\Thicks(DT) = \D^{\fd}(\gr \Lambda^!)$}, and hence we are done.
\end{proof}

\section{On $n$-hereditary algebras} \label{n-hereditary}
The class of \textit{$n$-hereditary algebras} was introduced in \cite{Herschend-Iyama-Oppermann} and consists of the disjoint union of $n$-representation finite and $n$-representation infinite algebras. In this section we recall some definitions and basic results from \cites{Iyama-Oppermann APR,Herschend-Iyama-Oppermann,Iyama-Oppermann}. This forms a necessary background for exploring connections between the notion of $n$-$T$-Koszulity and higher hereditary algebras, which is the topic our next two sections. Note that \cref{n-hereditary} does not contain any new results.

Throughout this section, let $A$ be a finite dimensional algebra. Recall that if $A$ has finite global dimension, then the derived Nakayama functor \mbox{$\nu(-) = - \otimes^{\mathbf{L}}_{A} DA$} is a Serre functor on $\D^b(\modu A)$. We use the notation \mbox{$\nu_n = \nu(-)[-n]$}. The algebra $A$ is called \textit{$n$-representation finite} if $\gldim A \leq n$ and $\modu A$ contains an $n$-cluster tilting object. We have the following criterion for $n$-representation finiteness in terms of the subcategory
\[
\U = \add\{\nu_n^i A \mid i \in \Z\} \subseteq \D^b(\modu A).
\]

\begin{thm}[See \cite{Iyama-Oppermann}*{Theorem 3.1}] \label{nrepfindef}
Assume $\gldim A \leq n$. The following are equivalent:
\begin{enumerate}
    \item $A$ is $n$-representation finite.
    \item $DA \in \U$.
    \item $\nu \U = \U$.
\end{enumerate}
\end{thm}

In particular, an algebra $A$ with $\gldim A \leq n$ is $n$-representation finite if and only if there for any indecomposable projective $A$-module $P_i$ is an integer \mbox{$m_i \geq 0$} such that $\nu_n^{-m_i}(P_i)$ is indecomposable injective. We will need the following well-known property of $n$-representation finite algebras. 

\begin{lem}[See \cite{Herschend-Iyama-Oppermann}*{Proposition 2.3}] \label{n-rep-fin}
Let $A$ be $n$-representation finite. For each indecomposable projective $A$-module $P_i$, we then have \mbox{$\Hm^l(\nu_n^{-m}(P_i))=0$} for \mbox{$l\neq 0$} and \mbox{$0 \leq m \leq m_i$}, where $m_i$ is given as above.
\end{lem}

Moving on to the second part of the $n$-hereditary dichotomy, recall that $A$ is called \textit{$n$-representation infinite} if $\gldim A \leq n$ and \mbox{$\Hm^{i}(\nu^{-j}_n(A)) = 0$} for \mbox{$i \neq 0$} and \mbox{$j \geq 0$}. The following well-known basic lemma will be needed in our next two sections. 

\begin{lem}\label{nrepinflem}
Let $\gldim A < \infty$. Then 
\[\gldim A = \max \{ i \geq 0 \,\, \lvert \,\, \Ext^i_{A}(DA, A) \neq 0 \}.\]
\end{lem}

We use this in combination with the fact that 
\[
\Hm^{i}(\nu^{-1}_n(A)) \simeq \Ext^{n+i}_A(DA,A),
\]
which holds since 
\[
\nu^{-1}(-) =\mathbf{R}\Hom_{A}( D A,-).
\]

Like in the classical theory of hereditary algebras, the class of $n$-hereditary algebras also has an appropriate version of (higher) preprojective algebras which is nicely behaved. Given an $n$-hereditary algebra $A$, we denote the \textit{$(n+1)$-preprojective algebra of $A$} by $\Pi_{n+1}A$. Recall from \cite{Iyama-Oppermann}*{Lemma 2.13} that
\[
\Pi_{n+1}A \simeq \bigoplus_{i \geq 0}\Hom_{D^b(A)}(A, \nu^{-i}_{n}(A)).
\]
If $A$ is $n$-representation finite, the associated $(n+1)$-preprojective is finite dimensional and self-injective, whereas in the $n$-representation infinite case, the $(n+1)$-preprojective is infinite dimensional graded bimodule $(n+1)$-Calabi--Yau of Gorenstein parameter $1$.

\begin{remark} 
Note that terminology related to the classes of algebras discussed in this section varies in the literature. For instance, an $n$-representation finite algebra is called `$n$-representation-finite $n$-hereditary' in \cite{Jasso-Kulshammer}. This terminology is very reasonable, but as we need to mention $n$-representation finite algebras frequently, we stick to the notion from \cite{Iyama-Oppermann APR} for brevity. 
\end{remark}

\section{Higher Koszul duality and $n$-representation infinite algebras} \label{T-Koszul and n-rep inf}

In this section we investigate connections between $n$-representation infinite algebras and the notion of higher Koszulity. Let us first present our standing assumptions.

\begin{setup}\label{setup}
Throughout the rest of this section, we make the following standing assumptions:
\begin{enumerate}
    \item Let $\Lambda=\oplus_{i \geq 0}\Lambda_i$ be a finite dimensional graded Frobenius algebra of highest degree $a \geq 1$ with $\gldim \Lambda_0 < \infty$.
    \item Let $T$ be a basic graded $\Lambda$-module which is concentrated in degree $0$ and a tilting module over $\Lambda_0$. We assume $T_\mu \simeq T$ as $\Lambda$-modules for the Nakayama automorphism $\mu$ of $\Lambda$.
\end{enumerate}
For our fixed positive integer $n$, we use the notation
\[
\widetilde{T}=\bigoplus_{i=0}^{a-1} \Omega^{-ni}T\langle i \rangle
\]
and denote the endomorphism algebra $\End_{\stgrmodu\Lambda}(\widetilde{T})$ by $B$. We write \mbox{$T \simeq \oplus^{t}_{i=1} T^{i}$} for the decomposition of $T$ into indecomposable summands. 
\end{setup}

One should note that in the classical case, where $T$ is given as the direct sum of all the simple modules, the assumption $T \simeq T_\mu$ is automatically satisfied. Observe moreover that as we have assumed $T \simeq T_\mu$, we immediately obtain \mbox{$\Omega T_{\mu}\langle -a \rangle \simeq \Omega T\langle -a \rangle$}. Additionally, it follows from the assumption $T \simeq T_\mu$ that twisting by the Nakayama automorphism $\mu$ of $\Lambda$ only permutes the indecomposable summands of $T$. This means that we have a permutation, for simplicity also denoted by $\mu$, on the set $\{1,\ldots,t\}$ such that $T^{i}_{\mu} \simeq T^{\mu(i)}$.

The following characterization of $n$-$T$-Koszul algebras satisfying the standing assumptions above is the main result of this section.

\begin{thm} \label{Characterization}
Assume \cref{setup}. The following statements are equivalent:
\begin{enumerate}
    \item $\Lambda$ is $n$-$T$-Koszul. 
    \item $\widetilde{T} = \oplus_{i=0}^{a-1}\Omega^{-ni}T \langle i \rangle$ is a tilting object in $\stgrmodu \Lambda$ and $B= \End_{\stgrmodu \Lambda}(\widetilde{T})$ is $(na-1)$-representation infinite.
\end{enumerate}
\end{thm}

It should be noted that in the classical case where $\Lambda$ is $n$-$T$-Koszul for $n=1$ and $T=\Lambda_0$, our tilting object $\widetilde{T} \in \stgrmodu \Lambda$ corresponds under the BGG-correspondence \mbox{$\stgrmodu \Lambda \simeq \D^b(\qgr \Lambda^!)$} to the tilting object \mbox{$\oplus_{i=0}^{a-1} \Lambda^{!} \langle -i\rangle\in \D^b(\qgr \Lambda^!)$} considered in \cite{Minamoto & Mori 2011}.
In fact, this follows by using \cref{K and int deg prop} and that $K(T) \simeq \Lambda^!$.

Before giving the proof of \cref{Characterization}, we show some useful lemmas. Our first aim is to describe the endomorphism algebra $B$ as an upper triangular matrix algebra of finite global dimension, see \cref{gldim lemma}. We start by recalling the following lemma.

\begin{lem}[See \cite{Fossum-Griffith-Reiten}*{Corollary 4.21 (4)}] \label{utriangular}
Let $A$ and $A'$ be finite dimensional algebras and $M$ an $A^{\op} \otimes_k A'$-module. Then the algebra 
\[
\begin{bmatrix}
A & M \\
0 & A'
\end{bmatrix}
\]
has finite global dimension if and only if both $A$ and $A'$ have finite global dimension.
\end{lem}

In \cref{gldim lemma} we describe $B$ as an upper triangular matrix algebra associated to the graded algebra $\Gamma = \oplus_{i \geq 0} \Ext^{ni}_{\gr\Lambda}(T,T\langle i \rangle)$. Notice that in the case where $\Lambda$ is $n$-$T$-Koszul, the algebra $\Gamma$ coincides with the $n$-$T$-Koszul dual $\Lambda^!$.

\begin{lem} \label{gldim lemma}
Assume \cref{setup}. The algebra $B = \End_{\stgrmodu \Lambda}(\widetilde{T})$ is isomorphic to the upper triangular matrix algebra
\[
B \simeq \begin{pmatrix}
\Gamma_{0} & \Gamma_{1} & \cdots & \Gamma_{a - 1} \\
0 & \Gamma_{0} & \cdots & \Gamma_{a - 2} \\
\vdots  & \vdots  & \ddots & \vdots  \\
0 & 0 & \cdots & \Gamma_{0} 
\end{pmatrix},
\]
where $\Gamma = \oplus_{i \geq 0} \Ext^{ni}_{\gr\Lambda}(T,T\langle i \rangle)$. In particular, the global dimension of $B$ is finite.
\end{lem}

\begin{proof}
For $0 \leq i,j \leq a-1$, we consider 
\[
\Hom_{\stgrmodu \Lambda}(\Omega^{-nj}T\langle j \rangle,\Omega^{-ni}T\langle i \rangle) \simeq \Hom_{\stgrmodu \Lambda}(T,\Omega^{-n(i-j)}T\langle i-j \rangle).
\]
In the case $i < j$, we note that $\lvert i - j \rvert \leq a-1$ and so \cref{lemma} \textit{(8)} applies. 
Consequently, 
\[
\Hom_{\stgrmodu \Lambda}(T,\Omega^{-n(i-j)}T\langle i-j \rangle) 
\simeq 
\Hom_{\gr \Lambda}(T,\Omega^{-n(i-j)}T\langle i-j \rangle) 
= 0.
\]

If $i=j$, one obtains $\End_{\stgrmodu \Lambda}(T)$, which is isomorphic to \mbox{$\End_{\gr \Lambda}(T) = \Gamma_0$} by \cref{lemma} \textit{(3)}. For $i > j$, we get
\[
\Hom_{\stgrmodu \Lambda}(T,\Omega^{-n(i-j)}T\langle i-j \rangle) \simeq \Ext_{\gr \Lambda}^{n(i-j)}(T,T \langle i-j \rangle) = \Gamma_{i-j}.
\]
Computing our matrix with respect to the decomposition 
\[
\widetilde{T}= \Omega^{-n(a-1)}T \langle a-1 \rangle \oplus \cdots \oplus \Omega^{-n} T \langle 1 \rangle \oplus T,
\]
this yields our desired description. 

To see that $B$ is of finite global dimension, notice that \mbox{$\Gamma_0 \simeq \End_{\Lambda_0}(T)$}. As $\End_{\Lambda_0}(T)$ is derived equivalent to $\Lambda_0$, which is of finite global dimension, \cref{utriangular} applies and the claim follows. \end{proof}

Note that we could also have deduced that $B$ is of finite global dimension from \cite{Yamaura 2013}*{Corollary 3.12}. 

Our next lemma provides an important step in the proof of \cref{Characterization}. Recall that given a graded $\Lambda$-module $M = \oplus_{i\in\Z}M_i$, each graded part $M_i$ is also a module over $\Lambda_0$.  On the other hand, every $\Lambda_0$-module is trivially a graded $\Lambda$-module concentrated in degree $0$. In the proof of \cref{genlem}, we repeatedly vary between thinking of graded $\Lambda$-modules concentrated in one degree and modules over the degree $0$ part. 

We use the notation $M_{\geq i}$ for the submodule of $M$ with
\[
(M_{\geq i})_j = \begin{cases}
      M_j & j \geq i \\
      0 & j < i,
   \end{cases}
\]
while the quotient module $\faktor{M}{M_{\geq i+1}}$ is denoted by $M_{\leq i}$. Note that $M_i$ is isomorphic to $\faktor{M_{\geq i}}{M_{\geq i+1}}$.

\begin{lem}\label{genlem} 
Assume \cref{setup}. The module $\widetilde{T}$ generates $\stgrmodu \Lambda$ as a thick subcategory, i.e.\ we have \mbox{$\Thick_{\stgrmodu \Lambda}(\widetilde{T}) = \stgrmodu \Lambda$}.
\end{lem}

\begin{proof}
We divide the proof into two steps. In the first part, we show that the set of objects $\{\Lambda_0\langle i \rangle \}_{i \in \mathbb{Z}}$ generates $\stgrmodu \Lambda$ as a thick subcategory. In the second part, we prove that this set is contained in $\Thick_{\stgrmodu \Lambda}(\widetilde{T})$, which yields our desired conclusion.

\proofpart{}
Notice first that every graded $\Lambda$-module which is concentrated in degree $i$ is necessarily contained in the thick subcategory generated by $\Lambda_0 \langle i \rangle$. To see this, apply $\langle i \rangle$ to a finite $\Lambda_0$-projective resolution of the module, split up into short exact sequences and use that thick subcategories have the $2/3$-property on distinguished triangles.

Let $M$ be an object in $\stgrmodu \Lambda$. Denote the highest and lowest degree of $M$ by $h$ and $l$, respectively. Observe that $M_{\geq h} = M_h$. By the argument above, we know that $M_j$ is in \mbox{$\Thick_{\stgrmodu \Lambda}(\{\Lambda_0\langle i \rangle \}_{i \in \mathbb{Z}})$} for every $j$. Considering the short exact sequences
\begin{equation} \label{ses}
    \begin{tikzcd} 
0 \rar &  M_{\geq j+1} \rar & M_{\geq j} \rar & M_{j} \rar & 0
\end{tikzcd}
\end{equation}
for $j = l, \ldots, h-1$, we can hence conclude that also $M_{\geq l} = M$ is in our subcategory. This proves that \mbox{$\Thick_{\stgrmodu \Lambda}(\{\Lambda_0\langle i \rangle \}_{i \in \mathbb{Z}}) = \stgrmodu \Lambda$}.

\proofpart{}
As thick subcategories are closed under direct summands and translation, we immediately observe that $T\langle i \rangle$ is in $\Thick_{\stgrmodu \Lambda}(\widetilde{T})$ for $i = 0,\dots,a-1$. Since $T$ is a tilting module over $\Lambda_0$, and $\Lambda_0 \langle i \rangle$ thus has a finite coresolution in $\add T\langle i \rangle$, this implies that $\Lambda_0\langle i \rangle$ is in $\Thick_{\stgrmodu \Lambda}(\widetilde{T})$ for $i = 0,\dots,a-1$. Note that by our argument in \textit{Part 1}, we hence know that every module which is concentrated in degree $i$ for some $i=0,\dots,a-1$, is contained in our subcategory.

Consider the short exact sequences (\ref{ses}) for $M=\Lambda$, and recall that the module $\Lambda_{\geq 0} = \Lambda$ is projective and hence zero in $\stgrmodu \Lambda$. By a similar argument as before, this yields that $\Lambda_a$ is contained in $\Thick_{\stgrmodu \Lambda}(\widetilde{T})$. We next explain why this entails that also $\Lambda_0 \langle a \rangle$ is in our subcategory. 

Since $\Lambda$ is graded Frobenius, we have $\Lambda \langle -a \rangle \simeq D\Lambda$ as graded right $\Lambda$-modules, and thus $D\Lambda_0 \simeq \Lambda_a$ as $\Lambda_0$-modules. As $\Lambda_0$ has finite global dimension, this implies that $\Lambda_0$ is contained in $\Thick_{\D^b(\Lambda_0)}  (\Lambda_a \langle -a \rangle)$. Composing the equivalence \mbox{$\D^b(\gr \Lambda) / \D^{\perf}(\gr \Lambda) \simeq \stgrmodu \Lambda$} from \cite{Buchweitz}*{Theorem 4.4.1} with the associated quotient functor, one obtains a triangulated functor \mbox{$Q \colon \D^b(\gr \Lambda) \rightarrow \stgrmodu \Lambda$}. From the chain of subcategories
\[
\Thick_{\D^b(\Lambda_0)} \Lambda_a\langle -a\rangle \subseteq \Thick_{\D^b(\gr \Lambda)} \Lambda_a\langle -a\rangle 
\subseteq Q^{-1}(\Thick_{\stgrmodu \Lambda} \Lambda_a \langle -a\rangle),
\] 
we see that $\Lambda_0\langle a \rangle$ is in $\Thick_{\stgrmodu \Lambda} (\Lambda_a)$, which again is contained in $\Thick_{\stgrmodu \Lambda}(\widetilde{T})$. 

Shifting the short exact sequences involved by positive integers and using the same argument as above, one obtains that $\Lambda_0 \langle i \rangle$ is in $\Thick_{\stgrmodu \Lambda}(\widetilde{T})$ for all $i \geq 0$. That $\Lambda_0 \langle i \rangle$ is in $\Thick_{\stgrmodu \Lambda}(\widetilde{T})$ for all $i < 0$ is shown similarly using the short exact sequences
\[
\begin{tikzcd}
0 \rar &  \Lambda_j \rar & \Lambda_{\leq j} \rar & \Lambda_{\leq j-1} \rar & 0
\end{tikzcd}
\]
for $j = 1, \ldots, a$. We can hence conclude that $\Lambda_0\langle i \rangle$  is in $\Thick_{\stgrmodu \Lambda}(\widetilde{T})$ for every integer $i$, which finishes our proof.
\end{proof}

We are now ready to prove the main result of this section.

\begin{proof}[Proof of \cref{Characterization}]
We begin by proving \textit{(1)} implies \textit{(2)}. To see that $\widetilde{T}$ is a tilting object, notice first that it generates $\stgrmodu \Lambda$ by \cref{genlem}. Thus, we need only check rigidity, i.e.\ that \mbox{$\Hom_{\stgrmodu \Lambda}(\widetilde{T},\Omega^{-l}\widetilde{T}) = 0$} whenever $l \neq 0$. Splitting up on summands of \mbox{$\widetilde{T}=\oplus_{i=0}^{a-1} \Omega^{-ni}T\langle i \rangle$} and reindexing appropriately, we see that it is enough to show
\begin{equation} \label{ets}
    \Hom_{\stgrmodu \Lambda}(T,\Omega^{-(nk+l)}T\langle k \rangle)=0 ~~\text{ for }~~ l \neq 0
\end{equation}
for any integer $k$ with $\lvert k \rvert \leq a-1$. For convenience, we instead argue that (\ref{ets}) holds for any $k \geq 1-a$, as this more general statement will be needed later in the proof.

Assume $nk+l=0$. Now $l \neq 0$ implies $k \neq 0$, so the condition above is satisfied as our morphisms are homogeneous of degree $0$. 

Let $nk+l > 0$. Now,
\[
\Hom_{\stgrmodu \Lambda}(T,\Omega^{-(nk+l)}T\langle k \rangle) \simeq \Ext_{\gr \Lambda}^{nk+l}(T,T\langle k \rangle), 
\]
which is zero for $l\neq 0$ as $\Lambda$ is $n$-$T$-Koszul.
 
It remains to verify (\ref{ets}) in the case where $nk+l < 0$. As $k \geq 1-a$, part \textit{(8)} of \cref{lemma} applies. We hence see that (\ref{ets}) is satisfied also in this case, which means that $\widetilde{T}$ is a tilting object in $\stgrmodu \Lambda$. 

Recall from \cref{gldim lemma} that $B$ has finite global dimension. To see that $B$ is $(na-1)$-representation infinite, we use that $\widetilde{T}$ is a tilting object in $\stgrmodu \Lambda$. Hence, the equivalence and correspondence of Serre functors described in \cref{tilting section} yields
 \begin{align} 
     \Hom_{\stgrmodu \Lambda}(\widetilde{T}, \Omega^{-(nai+l)}\widetilde{T}\langle ai \rangle) 
     & \simeq \Hom_{\D^b(B)}(B, \nu^{-i}(B)[nai - i+l]) \label{tilting isomorphism}\\
     & \simeq \Hom_{\D^b(B)}(B, \nu_{na-1}^{-i}(B)[l]) \nonumber \\
     & \simeq \Hm^l(\nu_{na-1}^{-i}(B)), \nonumber
\end{align}
where we have implicitly used that $T_{\mu} \simeq T$ and that the functors $\Omega^{\pm 1}(-)$, $\langle \pm 1 \rangle$ and $(-)_{\mu}$ commute.
 
Splitting up on summands of $\widetilde{T}$ and reindexing appropriately, we notice that \mbox{$\Hom_{\stgrmodu \Lambda}(\widetilde{T}, \Omega^{-(nai+l)}\widetilde{T}\langle ai \rangle) = 0$} for $l \neq 0$ and $i > 0$ if and only if (\ref{ets}) is satisfied for $k>0$. The latter holds as $a \geq 1$ and since we have already shown that (\ref{ets}) is satisfied for $k \geq 1-a$. We can thus conclude that \mbox{$\Hm^l(\nu_{na-1}^{-i}(B)) = 0$} for \mbox{$i>0$} and \mbox{$l \neq 0$}. Note that when $i = 0$ and $l \neq 0$, we have \mbox{$\Hm^l(\nu_{na-1}^{-i}(B)) = \Hm^l(B) = 0$}. Consequently, since 
\[
\Hm^l(\nu_{na-1}^{-1}(B)) \simeq \Ext^{na - 1 + l}_{B}(DB,B),
\] 
our algebra $B$ is $(na-1)$-representation infinite by \cref{nrepinflem}.

To show that \textit{(2)} implies \textit{(1)}, we verify that given any integer $k$, one obtains \mbox{$\Ext_{\gr \Lambda}^{nk+l}(T,T\langle k \rangle) = 0$} for $l \neq 0$. If $nk+l \leq 0$, this is immediately satisfied, so assume $nk+l > 0$. As before, we now have
\[
\Ext_{\gr \Lambda}^{nk+l}(T,T\langle k \rangle) \simeq \Hom_{\stgrmodu \Lambda}(T,\Omega^{-(nk+l)}T\langle k \rangle).
\]
If $k < 0$, this is zero by \cref{lemma} \textit{(7)}, so it remains to check the case where $k$ is non-negative. 

Observe that the isomorphism
\begin{align*}
     \Hom_{\stgrmodu \Lambda}(\widetilde{T}, \Omega^{-(nai+l)}\widetilde{T}\langle ai \rangle) & \simeq \Hm^l(\nu_{na-1}^{-i}(B))
\end{align*}
from (\ref{tilting isomorphism}) still holds, as $\widetilde{T}$ is assumed to be a tilting object in $\stgrmodu \Lambda$. As $B$ is $(na-1)$-representation infinite, we know that \mbox{$\Hm^l(\nu_{na-1}^{-i}(B)) = 0$} for $i \geq 0$ and $l \neq 0$. The isomorphism above hence yields that (\ref{ets}) is satisfied for $k \geq 0$.

This allows us to conclude that $T$ is graded $n\mathbb{Z}$-orthogonal. As $T$ is a tilting module over $\Lambda_0$ by our standing assumptions, we have hence shown that $\Lambda$ is $n$-$T$-Koszul.
\end{proof}

To illustrate our characterization result, we consider some examples. As can be seen below, we use diagrams to represent indecomposable modules. The reader should note that in general one cannot expect modules to be represented uniquely by such diagrams, but in the cases we look at, they determine indecomposable modules up to isomorphism.

\begin{example}\label{ex: section 6}
Let $A$ denote the path algebra of the quiver 
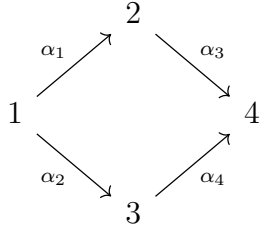
\begin{center}
\begin{tikzcd}
& 2 \arrow[dr, "\alpha_3"] & \\
1 \arrow[ur,"\alpha_1"] \arrow[dr,swap, "\alpha_2"] & & 4\\
& 3 \arrow[ur,swap, "\alpha_4"] &
\end{tikzcd}
\end{center}
modulo the ideal generated by paths of length two. The trivial extension $\Delta A$ is given by the quiver
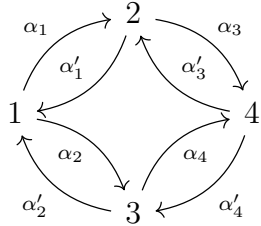
\begin{center}
\begin{tikzcd}
& 2 \arrow[dr, bend left, "\alpha_3"] \arrow[dl, swap, bend left, "\alpha_1'"] & \\
1 \arrow[ur, bend left, "\alpha_1"] \arrow[dr,swap, bend left, "\alpha_2"] & & 4 \arrow[ul, swap, bend left, "\alpha_3'"] \arrow[dl, bend left, "\alpha_4'"]\\
& 3 \arrow[ur,swap, bend left, "\alpha_4"] \arrow[ul, bend left, "\alpha_2'"] & 
\end{tikzcd}
\end{center}
with the trivial extension relations, i.e.\ all length two zero relations with the exception of $\alpha_i \alpha_i'$ and $\alpha_i' \alpha_i$. Instead, these latter paths satisfy all length two commutativity relations, i.e.\ \mbox{$\alpha_1 \alpha_1' - \alpha_2 \alpha_2'$}, \mbox{$\alpha_3 \alpha_3' - \alpha_1' \alpha_1$}, \mbox{$\alpha_4' \alpha_4 - \alpha_3' \alpha_3$}, and \mbox{$\alpha_2' \alpha_2 - \alpha_4 \alpha_4'$}. Moreover, we let $\Delta A$ be graded with the trivial extension grading. 

The indecomposable projective injectives for $\Delta A$ can be given as the diagrams 
\begin{center}
\begin{tikzpicture}
\node[scale=.75] (a) at (0,0){$
\arraycolsep=.8pt\def\arraystretch{0.8}
\begin{array}{ccc}
     &1_0 & \\
  3_0   & &2_0 \\
     & 1_1 &
\end{array}
$};
\node[scale=.75] (a) at (2,0){
$
\arraycolsep=.8pt\def\arraystretch{0.8}
\begin{array}{ccc}
     &2_0 & \\
  1_1   & &4_0 \\
     & 2_1 &
\end{array}
$
};
\node[scale=.75] (a) at (4,0){
$
\arraycolsep=.8pt\def\arraystretch{0.8}
\begin{array}{ccc}
     &3_0 & \\
  1_1   & &4_0 \\
     & 3_1 &
\end{array}
$

};
\node[scale=.75] (a) at (6,0){
$
\arraycolsep=.8pt\def\arraystretch{0.8}
\begin{array}{ccc}
     &4_0 & \\
  2_1   & &3_1, \\
     & 4_1 &
\end{array}
$
};
\end{tikzpicture}
\end{center}
where the (non-subscript) numbers represent elements of a basis for the module, each of which is annihilated by all the idempotents except for $e_i$ with $i$ equal to the number. The subscript numbers represent the degree of the basis element.

Let $T$ be the tilting $A$-module given by the direct sum of the following modules
\begin{center}
\begin{tikzpicture}
\node[scale=.75] (a) at (0,0){
$
\arraycolsep=.8pt\def\arraystretch{0.8}
\begin{array}{ccc}
& 1_0   &\\
3_0 & & 2_0\\
\end{array}
$
};
\node[scale=.75] (b) at (2,0){
$
\arraycolsep=.8pt\def\arraystretch{0.8}
\begin{array}{c}
2_0   \\
\end{array}
$
};
\node[scale=.75] (c) at (4,0){
$
\arraycolsep=.8pt\def\arraystretch{0.8}
\begin{array}{c}
3_0 \\
\end{array}
$
};
\node[scale=.75] (d) at (6,0){
$
\arraycolsep=.8pt\def\arraystretch{0.8}
\begin{array}{ccc}
2_0 &   & 3_0.\\
 & 4_0 &\\
\end{array}
$
};
\end{tikzpicture}
\end{center}
The initial two terms of the minimal injective $\Delta A$-resolution of the first summand of $T$ as well as the first two cosyzygies can be given as
\begin{center}
\begin{tikzpicture}
\node[scale=.75] (b) at (3,1.5){
$
\arraycolsep=.8pt\def\arraystretch{0.8}
\begin{array}{ccc}
& 3_{-1}   &\\
4_{-1} & & 1_0\\
& 3_0 &
\end{array}
\oplus
\arraycolsep=.8pt\def\arraystretch{0.8}
\begin{array}{ccc}
& 2_{-1}   &\\
1_0 & & 4_{-1}\\
& 2_0 &
\end{array}
$
};
\node[scale=.75] (c) at (6,0){
$
\arraycolsep=.8pt\def\arraystretch{0.8}
\begin{array}{ccccc}
& 3_{-1} & & 2_{-1} &\\
4_{-1} & & 1_0 & & 4_{-1}
\end{array}
$
};
\node[scale=.75] (d) at (9.725,1.5){
$
\arraycolsep=.8pt\def\arraystretch{0.8}
\begin{array}{ccc}
& 4_{-2}   &\\
2_{-1} & & 3_{-1}\\
& 4_{-1} &
\end{array}
\oplus
\arraycolsep=.8pt\def\arraystretch{0.8}
\begin{array}{ccc}
& 1_{-1}   &\\
3_{-1} & & 2_{-1}\\
& 1_{0} &
\end{array}
\oplus
\arraycolsep=.8pt\def\arraystretch{0.8}
\begin{array}{ccc}
& 4_{-2}   &\\
2_{-1} & & 3_{-1}\\
& 4_{-1} &
\end{array}
$
};
\node[scale=.75] (e) at (13,0){
$
\arraycolsep=.8pt\def\arraystretch{0.8}
\begin{array}{ccccccc}
& 4_{-2}   & & 1_{-1}   & & 4_{-2}   &\\
2_{-1} & & 3_{-1} & & 2_{-1} & & 3_{-1}.\\
\end{array}
$
};
\draw [->] (b) edge (c) (c) edge (d) (d) edge (e);
\end{tikzpicture}
\end{center}

Looking at this part of the resolution, it is not so obvious that $T$ is graded $2\mathbb{Z}$-orthogonal as a $\Delta A$ module, whereas by using the equivalence \mbox{$\D^b(\modu A) \simeq \stgrmodu \Delta A$} or by degree arguments as we have done before, it is immediate that \mbox{$\widetilde{T} \simeq T$} is a tilting object in $\stgrmodu \Delta A$. It is also easy to check that $\End_{\stgrmodu \Delta A}(T)$ is isomorphic to the hereditary algebra given by the path algebra of the quiver of $A$, which is representation infinite. Using \cref{Characterization}, we can hence conclude that the algebra $\Delta A$ is $2$-$T$-Koszul.

Note that this example also illustrates that, as has been remarked on in the literature before, one cannot always expect nice minimal resolutions of $T$ for (generalized) $T$-Koszul algebras. 
\end{example}

As a consequence of \cref{Characterization}, our next corollary shows that an algebra is $n$-representation infinite if and only if its trivial extension is $(n+1)$-Koszul with respect to its degree $0$ part. This result is inspired by connections between \mbox{$n$-representation} infinite algebras and graded bimodule $(n+1)$-Calabi--Yau algebras of Gorenstein parameter $1$, as studied in \cites{Herschend-Iyama-Oppermann,Keller 2011,Minamoto & Mori 2011,Amiot-Iyama-Reiten}. In some sense, the corollary below is a $T$-Koszul dual version of \cite{Herschend-Iyama-Oppermann}*{Theorem 4.36}.

\begin{cor} \label{motivating result}
Let $\Lambda = \Lambda_0 \oplus \Lambda_1$ be a finite dimensional graded Frobenius algebra of highest degree $1$ with $\gldim \Lambda_0 < \infty$.
Then $\Lambda$ is $(n+1)$-Koszul with respect to $T=\Lambda_0$ if and only if $\Lambda_0$ is $n$-representation infinite. In particular, we obtain a bijective correspondence 
\[
\left\{
\begin{tabular}{@{}l@{}}
    isomorphism classes \\
    of basic $n$-representa- \\ 
    tion infinite algebras
\end{tabular}
\right\}
{\rightleftarrows}
\left\{
\begin{tabular}{@{}l@{}}
    isomorphism classes of graded symmetric finite \\
    dimensional algebras of highest degree $1$ which are \\
    $(n+1)$-Koszul with respect to their degree $0$ part
\end{tabular}
\right\},
\]
where the maps are given by $ A \longmapsto \Delta A$ and $\Lambda_0 \longmapsfrom \Lambda$.
\end{cor}

\begin{proof}
Aside from the assumption that $(\Lambda_0)_\mu \simeq \Lambda_0$ as $\Lambda$-modules for the Nakayama automorphism $\mu$, all the other parts of \cref{setup} are trivially satisfied. 
Now, by \cref{graded frobenius lemma} \textit{(2)}, the Nakayama automorphism of $\Lambda$ is graded. It thus follows from \cref{prop: twist of algebra iso as one-sided modules to algebra} that $(\Lambda_0)_\mu \simeq \Lambda_0$ as $\Lambda$-modules.

Notice that $\End_{\stgrmodu \Lambda}(\Lambda_0) \simeq \End_{\gr \Lambda}(\Lambda_0) \simeq \Lambda_0$ by \cref{lemma} \textit{(3)}. Observe that \mbox{$\Hom_{\stgrmodu \Lambda}(\Lambda_0, \Omega^{-i}\Lambda_0) \simeq \Hom_{\stgrmodu \Lambda}(\Omega^{i}\Lambda_0, \Lambda_0) = 0$} for all $i \neq 0$. This follows by degree considerations similar to those used in the proof of \cref{lemma} and using the fact that syzygies of $\Lambda_0$ are generated in degrees greater or equal to $1$. Combining this with \cref{genlem}, one obtains that $\Lambda_0$ is a tilting object in $\stgrmodu\Lambda$, and consequently our first statement follows from \cref{Characterization}. 

We get the bijection as a special case of this, as $\Delta A$ is a graded symmetric finite dimensional algebra of highest degree $1$ and $\Lambda \simeq \Delta \Lambda_0$ as graded algebras in the case where $\Lambda$ is symmetric.
\end{proof}

Since the direction \textit{(2)} implies \textit{(1)} of \cref{Characterization} was most useful in \cref{ex: section 6}, we also include an example showing how the other direction can be utilised.

\begin{example}\label{ex: tensor product of Kroneckers}
Let $A$ denote the Kronecker algebra. The trivial extension $\Delta A$ is given by the quiver
\[
\begin{tikzcd}
1 
\rar[shift left = 0.5ex] \rar[shift right = 0.5ex] 
& 
2 
\lar[blue, shift left = 1.5ex] \lar[blue, shift right = 1.5ex]\\ 
\end{tikzcd}
\]
with the trivial extension relations, see e.g.\ \cite{Schroer}, \cite{Fernandez-Platzeck'02} or \cite{Fernandez-et-al'22}. 
The tensor product algebra \mbox{$\Lambda \coloneqq \Delta A \otimes_k \Delta A$} then has quiver given by
\[
\begin{tikzcd}
& (1,1) \dlar[shift left = 0.5ex] \dlar[shift right = 0.5ex] \drar[shift left = 0.5ex] \drar[shift right = 0.5ex]  & \\
(2,1) \drar[shift left = 0.5ex] \drar[shift right = 0.5ex] 
\urar[blue, shift left = 1.5ex] \urar[blue, shift right = 1.5ex]
& & 
(1,2) \dlar[shift left = 0.5ex] \dlar[shift right = 0.5ex]
\ular[blue, shift left = 1.5ex] \ular[blue, shift right = 1.5ex]
\\
& 
(2,2) 
\ular[blue, shift left = 1.5ex] \ular[blue, shift right = 1.5ex]
\urar[blue, shift left = 1.5ex] \urar[blue, shift right = 1.5ex]
&
\end{tikzcd}
\]
with the tensor product relations. We note that the relations of $\Lambda$ will not be explicitly used in the computations in this example, as we will instead employ general facts about tensor product algebras. We endow $\Delta A$ and $\Lambda$ with the gradings induced by putting the black arrows in degree $0$ and the blue arrows in \mbox{degree $1$}.

Since $\Delta A$ is graded symmetric of highest degree $1$, the algebra $\Lambda$ is graded symmetric of highest degree $a=2$. Putting \mbox{$T = \Lambda_0 = A \otimes_k A$}, we observe that the standing assumptions described in \cref{setup} are satisfied. 

As $A$ is $1$-representation infinite, it follows from \cref{motivating result} that $\Delta A$ is \mbox{$2$-Koszul} with respect to $A$. This implies that $\Lambda$ is $2$-Koszul with respect to $T$ by \cref{prop: n-T-Koszul closed under tensor products}. The direction \textit{(1)} implies \textit{(2)} of \cref{Characterization} now allows us to conclude that \mbox{$B = \End_{\stgrmodu \Lambda} (\widetilde{T})$} is $3$-representation infinite. 

We finish this example by giving the quiver of $B$. As \mbox{$\widetilde{T} = T \oplus \Omega^{-2} T \langle 1 \rangle$}, we have 
\[
B = \End_{\stgrmodu \Lambda} (\widetilde{T}) \simeq \End_{\stgrmodu \Lambda} (T)
\oplus 
\Hom_{\stgrmodu \Lambda} (T, \Omega^{-2} T \langle 1 \rangle)
\oplus
\End_{\stgrmodu \Lambda} (\Omega^{-2} T \langle 1 \rangle).
\]
Note that we have here used that \mbox{$\Hom_{\stgrmodu \Lambda} (\Omega^{-2} T \langle 1 \rangle,T) = 0$} by \mbox{\cref{lemma} \textit{(8)}}. The opposite of the quiver of $B$ can thus be given by
\[
\begin{tikzcd}
& 
(1,1)_0 
\dlar[shift left = 0.5ex] \dlar[shift right = 0.5ex] \drar[shift left = 0.5ex] \drar[shift right = 0.5ex]  
& \\
(2,1)_0 
\rar[blue, shift left = 0.5ex] \rar[blue, shift right = 0.5ex]
& 
(1,1)_1 
\dlar[shift left = 0.5ex] \dlar[shift right = 0.5ex] \drar[shift left = 0.5ex] \drar[shift right = 0.5ex] 
& 
(1,2)_0 
\dlar[crossing over, shift left = 0.5ex] \dlar[crossing over, shift right = 0.5ex]
\lar[blue, shift left = 0.5ex] \lar[blue, shift right = 0.5ex]
\\
(2,1)_1 
\drar[shift left = 0.5ex] \drar[shift right = 0.5ex]
& 
(2,2)_0
\lar[blue, shift left = 0.5ex] \lar[blue, shift right = 0.5ex]
\rar[blue, shift left = 0.5ex] \rar[blue, shift right = 0.5ex]
\arrow[from=ul, crossing over, shift left = 0.5ex] 
\arrow[from=ul, crossing over, shift right = 0.5ex]
&
(1,2)_1 
\dlar[shift left = 0.5ex] \dlar[shift right = 0.5ex]
\\
& 
(2,2)_1
&
\end{tikzcd},
\]
where the vertices with subscript $0$ and the black arrows between them correspond to \mbox{$\End_{\stgrmodu \Lambda} (T)$}, the vertices with subscript $1$ and the black arrows between them correspond to \mbox{$\End_{\stgrmodu \Lambda} (\Omega^{-2} T \langle 1 \rangle)$}, and the blue arrows correspond to \mbox{$\Hom_{\stgrmodu \Lambda} (T, \Omega^{-2} T \langle 1 \rangle)$}.
\end{example}

Our aim for the rest of this section is to use the theory we have developed to provide an affirmative answer to our motivating question from the introduction. As in the case of the generalized AS-regular algebras studied by Minamoto and Mori in \cite{Minamoto & Mori 2011}, the notion of quasi-Veronese algebras is relevant.

\begin{defin}
Let $\Gamma = \oplus_{i \in \mathbb{Z}}\Gamma_i$ be a $\mathbb{Z}$-graded algebra and $r$ a positive integer. The \textit{$r$-th quasi-Veronese algebra} of $\Gamma$ is a $\mathbb{Z}$-graded algebra defined by 
\[
\Gamma^{[r]} = \bigoplus_{i \in \mathbb{Z}}\begin{pmatrix}
\Gamma_{ri} & \Gamma_{ri + 1} & \cdots & \Gamma_{ri + r - 1} \\
\Gamma_{ri - 1} & \Gamma_{ri} & \cdots & \Gamma_{ri + r - 2} \\
\vdots  & \vdots  & \ddots & \vdots  \\
\Gamma_{ri - r + 1} & \Gamma_{ri - r + 2} & \cdots & \Gamma_{ri} 
\end{pmatrix}.
\]
\end{defin}

\begin{remark}
    We have chosen to use the terminology \textit{$r$-th quasi-Veronese algebra} to be consistent with \cite{Minamoto & Mori 2011}. However, it should be noted that the $r$-th quasi-Veronese algebra as defined above is indeed a \textit{$\mathbb{Z}/r\mathbb{Z}$-covering} in the sense of \cite{Bongartz-Gabriel}, which arises as a special type of \textit{smash product}, see e.g.\ \cite{CM}.
\end{remark}

In \cref{preprojective and Veronese} we show that if $\Lambda$ is $n$-$T$-Koszul, then the $na$-th preprojective algebra of $B = \End_{\stgrmodu \Lambda}(\widetilde{T})$ is isomorphic to a twist of the $a$-th quasi-Veronese of $\Lambda^!$. In order to make this precise, notice first that a graded algebra automorphism $\phi$ of a graded algebra $\Gamma$ induces a graded algebra automorphism $\phi^{[r]}$ of $\Gamma^{[r]}$ by letting \mbox{$\phi^{[r]}((\gamma_{j,k})) = (\phi(\gamma_{j,k}))$}. Here we use the notation $(\gamma_{j,k})$ for the matrix with $\gamma_{j,k}$ in position $(j,k)$. Recall also that we can define a possibly different graded algebra ${}_{\langle \phi \rangle}\Gamma$ with the same underlying vector space structure as $\Gamma$, but with multiplication \mbox{$\gamma \cdot \gamma' = \phi^{i}(\gamma)\gamma'$} for $\gamma'$ in $\Gamma_i$.  

Recall that $\mu$ is the Nakayama automorphism of $\Lambda$, and denote our chosen isomorphism $T_\mu \simeq T$ from before by $\psi$. Note that twisting by $\mu$ might non-trivially permute the summands of $T$. In the case where $\Lambda$ is $n$-$T$-Koszul, let $\overline{\mu}$ be the graded algebra automorphism of $\Lambda^!$ defined on the $i$-th component
\[
\Lambda^!_i = \Ext_{\gr \Lambda}^{ni}(T,T\langle i \rangle) \simeq \Hom_{\stgrmodu \Lambda}(T,\Omega^{-ni} T \langle i \rangle )
\]
by the composition
\[
\Hom_{\stgrmodu \Lambda}(T,\Omega^{-ni} T \langle i \rangle ) \xrightarrow{(-)_{\mu}} \Hom_{\stgrmodu \Lambda}(T_{\mu},\Omega^{-ni} T_{\mu} \langle i \rangle ) \xrightarrow{(-)^{\psi}} \Hom_{\stgrmodu \Lambda}(T,\Omega^{-ni} T \langle i \rangle ),
\]
where 
\[
(\gamma)^{\psi} = \Omega^{-ni}(\psi)\langle i\rangle \circ \gamma \circ \psi^{-1}
\] 
for $\gamma$ in $\Hom_{\stgrmodu \Lambda}(T_{\mu},\Omega^{-ni} T_{\mu} \langle i \rangle )$. 

Before showing \cref{preprojective and Veronese}, recall that a decomposition of $\widetilde{T}$ yields a decomposition of $B = \End_{\stgrmodu \Lambda}(\widetilde{T})$. In the proof below, we denote the summands of $\widetilde{T}$ by \mbox{$X^i = \Omega^{-ni}T \langle i \rangle$}, while $P^i$ is the projective $B$-module which is the preimage of $X^i$ under the equivalence \mbox{$\D^b(\modu B) \xrightarrow{\simeq} \stgrmodu \Lambda$} described in \cref{tilting section}.

\begin{prop} \label{preprojective and Veronese}
Assume \cref{setup} and let $\Lambda$ be $n$-$T$-Koszul. Then we have the isomorphism \mbox{$\Pi_{na}B \simeq {}_{\langle (\overline{\mu}^{-1})^{[a]} \rangle}(\Lambda^!)^{[a]}$} as graded algebras. In particular, we have \mbox{$\Pi_{na}B \simeq (\Lambda^!)^{[a]}$} in the case where $\Lambda$ is graded symmetric.
\end{prop}

\begin{proof}
As $\Lambda$ is $n$-$T$-Koszul, we know from \cref{Characterization} that $\widetilde{T}$ is a tilting object in $\stgrmodu \Lambda$ and that $B$ is $(na-1)$-representation infinite. The $i$-th component of the $na$-th preprojective algebra of $B$ is given by \mbox{$(\Pi_{na}B)_i = \Hom_{\D^b(B)}(B,\nu_{na-1}^{-i} B)$}. For \mbox{$0 \leq j,k \leq a-1$}, we hence consider
\begin{align*}
\Hom_{\D^b(B)}(P^k,\nu_{na-1}^{-i}P^j) &\simeq \Hom_{\stgrmodu \Lambda}(X^k,\Omega^{-(na-1)i-i}X^{j}_{\mu^{-i}} \langle ai \rangle) \\
&\simeq \Hom_{\stgrmodu \Lambda}(T,\Omega^{-n(ai+j-k)}T_{\mu^{-i}} \langle ai + j - k \rangle) \\
& \overset{(\ast)}{\simeq} \Ext_{\gr \Lambda}^{n(ai+j-k)}(T,T_{\mu^{-i}} \langle ai+j - k \rangle) \simeq  \Lambda^!_{ai+j-k}.
\end{align*}
Notice that the first isomorphism is a consequence of the equivalence and correspondence of Serre functors described in \cref{tilting section}, while $(\ast)$ is obtained by applying \cref{lemma} \textit{(3)} and \textit{(8)}. The last isomorphism follows from the assumption \mbox{$T_\mu \simeq T$}. 

Computing our matrix with respect to the decomposition 
\[
B \simeq P^{a-1}\oplus \cdots \oplus P^1 \oplus P^0,
\]
this yields
\begin{equation*}
(\Pi_{na}B)_i \simeq 
\begin{pmatrix}
\Lambda^!_{ai} & \Lambda^!_{ai + 1} & \cdots & \Lambda^!_{ai + a - 1} \\
\Lambda^!_{ai - 1} & \Lambda^!_{ai} & \cdots & \Lambda^!_{ai + a - 2} \\
\vdots  & \vdots  & \ddots & \vdots  \\
\Lambda^!_{ai - a + 1} & \Lambda^!_{ai - a + 2} & \cdots & \Lambda^!_{ai} 
\end{pmatrix},
\end{equation*}	
which shows that our two algebras are isomorphic as graded vector spaces.

In order to see that the multiplications agree, consider the diagram
\[
\begin{tikzcd}
(P^{j},\nu_{na-1}^{-i'}P^{j'}) 
\otimes (P^k,\nu_{na-1}^{-i}P^j) \arrow[d] \arrow[r] & (P^k,\nu_{na-1}^{-(i+i')}P^{j'}) \arrow[d]\\
(\nu_{na-1}^{-i}P^{j},\nu_{na-1}^{-(i + i')}P^{j'}) 
\otimes (P^k,\nu_{na-1}^{-i}P^j) \arrow[d] \arrow[r] & (P^k,\nu_{na-1}^{-(i+i')}P^{j'}) \arrow[d]\\
(X^{j}_{\mu^{-i}}(ai), X^{j'}_{\mu^{-(i+i')}}(a(i+i'))) \otimes (X^{k},X^{j}_{\mu^{-i}}(ai)) \arrow[d] \arrow[r] & (X^{k},X^{j'}_{\mu^{-(i+i')}}(a(i+i'))) \arrow[d]\\
\Lambda^!_{ai'+j'-j} \otimes \Lambda^!_{ai+j-k} \arrow[r] & \Lambda^!_{a(i+i')+j'-k.}
\end{tikzcd}
\] 
For simplicity, we have here suppressed the $\Hom$-notation and denoted \mbox{$\Omega^{-ni}(-)\langle i \rangle$} by $(-)(i)$. The horizontal maps are given by multiplication or composition, and the vertical maps give our isomorphism of graded algebras. In particular, the middle two horizontal maps are merely composition, whereas the top and bottom horizontal maps are the multiplication of
$\Pi_{na}B$ and \mbox{${}_{\langle (\overline{\mu}^{-1})^{[a]} \rangle}(\Lambda^!)^{[a]}$}, respectively. Moreover, the bottom vertical maps are given by 
\[f\otimes g \mapsto \prod_{l = 0}^{i'-1} \psi^{-1}_{\mu^{l-i'}} (ai'+j'-j) \circ f_{\mu^{i}}(-ai-j) \otimes \prod_{l = 0}^{i-1} \psi^{-1}_{\mu^{l-i}}(ai+j-k) \circ g(-k)\]
and 
\[f\circ g \mapsto \prod_{l = 0}^{i + i'-1} \psi^{-1}_{\mu^{l - i - i'}} (a(i + i')+j'-k) \circ (f\circ g)(-k).\]

As the diagram commutes, we can conclude that \mbox{$\Pi_{na}B \simeq {}_{\langle (\overline{\mu}^{-1})^{[a]}
\rangle}(\Lambda^!)^{[a]}$} as graded algebras. If $\Lambda$ is assumed to be graded symmetric, the Nakayama automorphism $\mu$ can be chosen to be trivial, so one obtains \mbox{$\Pi_{na}B \simeq (\Lambda^!)^{[a]}$}.
\end{proof}

In the corollary below, we show that the $(n+1)$-th preprojective of an $n$-representation infinite algebra is isomorphic to the $n$-$T$-Koszul dual of its trivial extension. 

\begin{cor} \label{cor: dual of triv.ext.}
If $A$ is basic $n$-representation infinite, then \mbox{$\Pi_{n+1}A \simeq (\Delta A)^!$} as graded algebras.
\end{cor}

\begin{proof}
Let $A$ be a basic $n$-representation infinite algebra. It then follows from \cref{motivating result} that $\Delta A$ is $(n+1)$-Koszul with respect to $A$. By \cref{lemma} part \textit{(3)}, one obtains \mbox{$\End_{\stgrmodu \Delta A}(A) \simeq \End_{\gr \Delta A}(A) \simeq A$}. Recall that $\Delta A$ is graded symmetric of highest degree $1$. Applying \cref{preprojective and Veronese} to $\Delta A$ hence yields our desired conclusion.
\end{proof}

\cref{cor: dual of triv.ext.} is a graded algebra analogue of \cite{Amiot}*{Lemma 4.13}, and can be regarded as a generalized Koszul dual version of \cite{Minamoto & Mori 2011}*{Proposition 4.20} and \cite{Grant-Iyama}. Note moreover that this corollary can be seen to follow from \cite{Keller}*{Lemma 4.4(b)}. 

We are now able to give an affirmative answer to our motivating question from the introduction, i.e.\ to deduce the equivalence \eqref{motivating equivalence} \mbox{$\stgrmodu (\Delta A) \simeq \D^b(\qgr \Pi_{n+1}A)$} as a consequence of higher Koszul duality.  

Recall that an $n$-representation infinite algebra $A$ is called \textit{$n$-representation tame} if the associated $(n+1)$-preprojective $\Pi_{n+1}A$ is a noetherian algebra over its center \cite{Herschend-Iyama-Oppermann}*{Definition 6.10}. Notice that a noetherian algebra is graded right coherent, so our result holds in this case.

\begin{cor} \label{n-rep tame}
Let $A$ be a basic $n$-representation infinite algebra with $\Pi_{n+1}A$ graded right coherent. There are then equivalences of triangulated categories as indicated in the commutative diagram
\[
    \begin{tikzcd}[column sep=22, row sep=20] 
   \D^b(\gr \Delta A) \arrow[r,"\simeq"] \arrow[d] & \D^b(\gr \Pi_{n+1}A)) \arrow[d]  \\
    \stgrmodu \Delta A \arrow[r,dashed,"\simeq"] & \D^b(\qgr \Pi_{n+1}A))
    \end{tikzcd}
\]
such that the equivalence on the top descends to the equivalence on the bottom. In particular, this holds if $A$ is $n$-representation tame.
\end{cor}

\begin{proof}
We get the equivalence $\D^b(\gr \Delta A) \simeq \D^b(\gr \Pi_{n+1}A)$ by \cref{Existence of T-Koszul equivalence} combined with \cref{motivating result} and \cref{cor: dual of triv.ext.}. By \cref{answer motivating question}, this equivalence descends to yield \mbox{$\stgrmodu (\Delta A) \simeq \D^b(\qgr \Pi_{n+1}A)$}. 
\end{proof}

\section{Higher almost Koszulity and $n$-representation finite algebras} \label{sec 7}

\label{section almost n-T}
In our previous section, we gave connections between higher Koszul duality and $n$-representation infinite algebras. Having developed our theory for one part of the higher hereditary dichotomy, it is natural to ask whether something similar holds in the $n$-representation finite case. To provide an answer to this question, we introduce the notion of higher almost Koszulity. As before, this should be formulated relative to a tilting module over the degree $0$ part of the algebra, which is itself assumed to be finite dimensional and of finite global dimension. Notice that after having presented the definitions and basic examples, we prove our results given the same standing assumptions as in \cref{T-Koszul and n-rep inf}, see \cref{setup}.

Our definition of an almost $n$-$T$-Koszul algebra is inspired by and generalizes the almost Koszul algebras of \cite{Brenner-Butler-King}. 

\begin{defin}[See \cite{Brenner-Butler-King}*{Definition 3.1}] \label{almost Koszul}
Assume that $\Lambda_0$ is semisimple. We say that $\Lambda$ is \textit{(right) almost Koszul} if there exist integers $p,q \geq 1$ such that the following conditions hold:
\begin{enumerate}
    \item $\Lambda_i = 0$ for all $i > p$.
    \item There is a graded complex 
        \[
        0 \rightarrow P^{-q} \rightarrow \cdots \rightarrow P^{-1} \rightarrow P^0 \rightarrow 0
        \]
        of projective right $\Lambda$-modules such that each $P^{-i}$ is generated by its component $P^{-i}_{i}$ and the only non-zero cohomology is $\Lambda_0$ in internal degree $0$ and \mbox{$P^{-q}_{q}\otimes_{\Lambda_0} \Lambda_p$} in internal degree $p+q$.
\end{enumerate}
If $\Lambda$ is almost Koszul for integers $p$ and $q$, one also says that $\Lambda$ is \textit{$(p,q)$-Koszul}. 
\end{defin}

Roughly speaking, if $\Lambda$ is almost Koszul, then taking tensor products over $\Lambda_0$ yields a somewhat periodic projective resolution of $\Lambda_0$ which is piecewise linear. This is similar to how the inverse Serre functor of an $n$-representation finite algebra acts on indecomposable projectives. However, for the latter, the periods of different indecomposable projectives need not be equal. Hence, unlike for almost Koszul algebras, we must allow the period of graded $n\mathbb{Z}$-orthogonality to be different for each of the individual indecomposable summands of our tilting module.

Motivated by our observations above, let us now define what it means for a module to be almost $(n,\underline{g},\underline{\ell})$-self-orthogonal. Recall that we consider a fixed decomposition \mbox{$T \simeq \oplus^t_{i = 1} T^{i}$} into indecomposable summands. 

\begin{defin} \label{periodisk ortogonal}
Let $T \simeq \oplus^t_{i = 1} T^{i}$ be a finitely generated basic graded $\Lambda$-module concentrated in degree $0$. Assume that for each \mbox{$i \in \{1,\dots,t\}$}, there exists an object \mbox{$T' \in \add T$} and positive integers $\ell_i$ and $g_i$ such that the following conditions hold:
\begin{enumerate}
    \item $\Omega^{-\ell_i}T^i \simeq T'\langle -g_i \rangle$.
    \item $\Ext^j_{\gr \Lambda}(T, T^i \langle k \rangle) = 0$ for $k \in \mathbb{Z}$ and $j \geq 0$ satisfying $j \neq nk$ and $j < \ell_i$.
\end{enumerate}
\sloppy We then say that $T$ is \textit{almost $(n,\underline{g},\underline{\ell})$-self-orthogonal} for \mbox{$\underline{g}=(g_1,\dots,g_t)$} and \mbox{$\underline{\ell}=(\ell_1,\dots,\ell_t)$}.
\end{defin}

This leads to our definition of what it means for an algebra to be almost $n$-$T$-Koszul, which is new even for $n = 1$.

\begin{defin}\label{almost n-T-Koszul}
	Assume $\gldim \Lambda_0 < \infty$ and let $T$ be a graded $\Lambda$-module concentrated in degree $0$. The highest degree of $\Lambda$ is denoted by $a$. We say that $\Lambda$ is \textit{almost $n$-$T$-Koszul} or \textit{almost $n$-Koszul with respect to $T$} if the following \mbox{conditions hold}:
	\begin{enumerate}
		\item $T$ is a tilting $\Lambda_0$-module.
		\item $T$ is almost $(n,\underline{g},\underline{\ell})$-self-orthogonal as a $\Lambda$-module.
        \item The parameters $\underline{g}$ and $\underline{\ell}$  satisfy $\ell_i = n(g_i - a) + 1$ for  $1 \leq i \leq t$. 
	\end{enumerate}
 An almost $n$-$T$-Koszul algebra is called \textit{minimally almost \mbox{$n$-$T$-Koszul}} or \textit{minimally almost $n$-Koszul with respect to $T$} if there exist no integers $0 < \ell_i' < \ell_i$ and $g_i'$ satisfying $\Omega^{-\ell_i'} T^{i} \simeq T'\langle - g_i' \rangle$ for $1 \leq i \leq t$ and $T' \in \add T$. By \cref{periodisk ortogonal} (2), such an isomorphism can only happen when $\ell_i' = ng_i'$.
\end{defin}

Whenever we work with an almost $n$-$T$-Koszul algebra, we use the notation $\ell_i$ and $g_i$ for integers given as in \cref{periodisk ortogonal}.

Let us now verify that \cref{almost n-T-Koszul} is indeed a generalization of \cref{almost Koszul}.

\begin{prop}\label{ex: almost-T-Koszul}
Let $\Lambda$ be a $(p,q)$-Koszul algebra. Then $\Lambda$ is of highest degree $p$ and is almost \mbox{$1$-$\Lambda_0$-Koszul} for the parameters given by $\ell_i = q + 1$ and $g_i = p + q$ for every \mbox{$i \in \{1,\ldots, t\}$}. Moreover, $\Lambda$ is minimally almost $1$-$\Lambda_0$-Koszul whenever $p > 1$.
\end{prop}
\begin{proof} 
We begin by showing that if $\Lambda$ is $(p,q)$-Koszul, then the highest degree $a$ of $\Lambda$ is equal to $p$. Indeed, note that by \cref{almost Koszul} (1), we have $a \leq p$, while \cref{almost Koszul} (2) states that $P^{-q}_q \otimes_{\Lambda_0}\Lambda_p \neq 0$, so $\Lambda_p \neq 0$.

We now continue by showing that $\Lambda$ is $1$-$\Lambda_0$-Koszul for the indicated parameters.
Since $\Lambda_0$ is semisimple, it is immediate that $\gldim \Lambda_0 < \infty$. 
Moreover, the module $T = \Lambda_0$ is concentrated in degree $0$ by definition. 
We thus proceed to check conditions (1) through (3) of \cref{almost n-T-Koszul}.

For (1), we note that it follows immediately from \cref{tilting module def} that $T = \Lambda_0$ is a $\Lambda_0$-tilting module. 
We next verify (2), i.e.\ that $\Lambda_0$ is almost $(1,\underline{g},\underline{\ell})$-self-orthogonal as a $\Lambda$-module.
For this, we use that an algebra is left $(p,q)$-Koszul if and only if it is right $(p,q)$-Koszul \cite{Brenner-Butler-King}*{Proposition 3.4}. Consequently, we get a left projective resolution of $\Lambda_0$, which can be dualized to yield a right injective resolution of $\Lambda_0$, allowing us to conclude that condition (2) of \cref{almost Koszul} implies conditions (1) and (2) of \cref{periodisk ortogonal}. This shows that part (2) of \cref{almost n-T-Koszul} holds. 
The parameters $\ell_i$ and $g_i$ then satisfy \mbox{\cref{almost n-T-Koszul} (3)} since 
$$1\cdot(g_i - a) + 1 = (p + q - p) + 1 = q + 1 = \ell_i,$$
and we have thus shown the first statement.

We now show the second statement in the proposition, i.e.\ that $\Lambda$ is minimally almost $1$-$\Lambda_0$-Koszul whenever $p > 1$. For this, observe that if $p > 1$, then $\ell_i$ cannot be smaller since $q$ is the length of the maximal linear part of the projective resolution of the simple modules.
\end{proof}

\begin{remark}
In light of this, we note that \cref{almost n-T-Koszul} (3) is a requirement that the highest degree $a$ and the parameters $\ell_i$ and $g_i$ of an almost $n$-$T$-Koszul algebra are related arithmetically in a way that generalizes parts of \cref{almost Koszul} (2). Indeed, as just shown in \cref{ex: almost-T-Koszul} above, any $(p,q)$-Koszul algebra $\Lambda$ is almost $1$-$\Lambda_0$-Koszul for the parameters obtained by setting $\ell_i = q + 1$ and $g_i = p + q$ for $i \in \{1, \ldots, t\}$.
Then \cref{almost Koszul} (2) implies that for every indecomposable $S \in \add \Lambda_0$, there exists an $S' \in \add \Lambda_0$ such that $\Omega^{q + 1} S \simeq S' \langle p + q \rangle$, which is precisely what yields that part (1) of \cref{periodisk ortogonal} holds for this choice of $\ell_i$ and $g_i$. In fact, as we saw, this is the only possible choice of $\ell_i$ and $g_i$ if $p > 1$.
Condition (3) in \cref{almost n-T-Koszul} should be interpreted as a higher analogue of this.
\end{remark}

Our main result in \cref{sec 7} is \cref{nrepfinchar}, which gives a characterization of an important class of minimally almost $n$-$T$-Koszul algebras. Before presenting this result, we provide an overview of some classes of examples of almost $n$-$T$-Koszul algebras.

Recall that a Dynkin quiver is said to have \textit{bipartite orientation} if every vertex is either a sink or a source. Just as in the study of almost Koszul algebras in \cite{Brenner-Butler-King}, trivial extensions of bipartite Dynkin quivers provide an important class of algebras which are minimally almost $n$-$T$-Koszul. See for instance \cite{Herschend-Iyama}*{Section 3.1} for an overview of the Coxeter numbers of different Dynkin quivers. 

\begin{prop}\label{prop: bipartite}
Let $Q$ be a bipartite Dynkin quiver with Coxeter number \mbox{$h \geq 4$}. Consider $\Lambda = \Delta kQ$ with grading given by putting arrows in degree $1$. Then $\Lambda$ is minimally almost $1$-$\Lambda_0$-Koszul. 
\end{prop}

\begin{proof}
As $Q$ is a bipartite Dynkin quiver and $h \geq 4$, it follows from \cite{Brenner-Butler-King}*{Proposition 3.11, Corollary 4.3} that $\Lambda$ is $(2,h-2)$-Koszul in the sense of \cref{almost Koszul}. Our conclusion now follows by \cref{ex: almost-T-Koszul}.
\end{proof}

We now illustrate the proposition above with a concrete example.

\begin{example}\label{ex: bipartite}
Let $Q$ be the quiver 
\[
\begin{tikzcd}
1 \rar["\alpha_0"] & 2 & 3 \lar[swap, "\alpha_1"].
\end{tikzcd}
\]
The trivial extension $\Lambda = \Delta kQ$ is given by 
\begin{center}
\begin{tikzcd}
1 \arrow[r, bend left, "\alpha_0"] & 2 \arrow[l, bend left, "\alpha_0'"] \arrow[r, bend right, swap, "\alpha_1'"] & 3 \arrow[l, swap, bend right, "\alpha_1"]
\end{tikzcd}
\end{center}
with relations $\alpha_{0}\alpha'_{1}$,  $\alpha_{1}\alpha'_{0}$, and $\alpha'_{0}\alpha_{0} - \alpha'_{1}\alpha_{1}$. This algebra is graded symmetric of highest degree $2$ with grading induced by letting the arrows be in degree $1$. As $Q$ is a bipartite Dynkin quiver, we have that $\Lambda$ is minimally almost $1$-$\Lambda_0$-Koszul by \cref{prop: bipartite}.
 
The indecomposable projective injectives of $\Lambda$ can be represented by the diagrams 
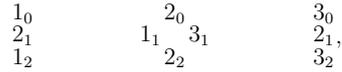
\begin{center}
\begin{tikzpicture}
\node[scale=.75] (a) at (0,0){$ 
\arraycolsep=.8pt\def\arraystretch{0.8}
\begin{array}{c}
1_0 \\
2_1\\
1_2
\end{array}
$};
\node[scale=.75] (b) at (2,0){$ 
\arraycolsep=.8pt\def\arraystretch{0.8}
\begin{array}{ccc}
& 2_0  & \\
1_1 & & 3_1 \\
& 2_2 & 
\end{array}
$};
\node[scale=.75] (c) at (4,0){$ 
\arraycolsep=.8pt\def\arraystretch{0.8}
\begin{array}{c}
3_0 \\
2_1 \\
3_2
\end{array},
$};
\end{tikzpicture}
\end{center}
where the subscripts indicate the degrees of the basis elements. One can verify directly that $\widetilde{\Lambda_0} =\Lambda_0\oplus\Omega^{-1}\Lambda_0 \langle 1 \rangle$ is a tilting object in $\stgrmodu \Lambda$ with $1$-representation finite endomorphism algebra. Note that this is a specific case of what we prove more generally in \cref{nrepfinchar}. Here, the endomorphism algebra of $\widetilde{\Lambda_0}$ in $\stgrmodu \Lambda$ decomposes as the direct product of the endomorphism algebras of 
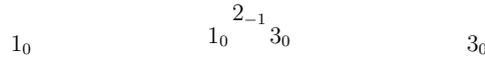
\begin{center}
\begin{tikzpicture}
\node[scale=.75] (a) at (-3,-.25){$ 
\arraycolsep=.8pt\def\arraystretch{0.8}
\begin{array}{c}
1_0
\end{array}
$};
\node[scale=.75] (a) at (0,0){$ 
\arraycolsep=.8pt\def\arraystretch{0.8}
\begin{array}{ccc} 
& 2_{-1}  & \\
1_0 & & 3_0\\
\end{array}
$};
\node[scale=.75] (a) at (3,-.25){$ 
\arraycolsep=.8pt\def\arraystretch{0.8}
\begin{array}{c} 
3_0
\end{array}
$};
\end{tikzpicture}
\end{center}
and 
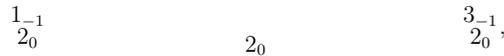
\begin{center}
\begin{tikzpicture}
\node[scale=.75] (b) at (0,-.25){$ 
\arraycolsep=.8pt\def\arraystretch{0.8}
\begin{array}{c} 
2_0
\end{array}
$};
\node[scale=.75] (a) at (-3,0){$ 
\arraycolsep=.8pt\def\arraystretch{0.8}
\begin{array}{c} 
1_{-1}  \\
2_0\\
\end{array}
$};
\node[scale=.75] (c) at (3,0){$ 
\arraycolsep=.8pt\def\arraystretch{0.8}
\begin{array}{c} 
3_{-1} \\
2_0\\
\end{array},
$};
\end{tikzpicture}
\end{center}
which are respectively isomorphic to the path algebras of the quivers
\begin{center}
\begin{tikzpicture}
\node[scale=1] (a) at (0,0){
\begin{tikzcd}
1 & 2 \lar \rar & 3
\end{tikzcd}
};
\end{tikzpicture}
\end{center}
and 
\begin{center}
\begin{tikzpicture}
\node[scale=1] (a) at (0,0){
\begin{tikzcd}
1 \rar & 2 & \lar 3.
\end{tikzcd}
};
\end{tikzpicture}
\end{center}
\end{example}

It is worth noting that algebras may often be almost $n$-$T$-Koszul without being $(p,q)$-Koszul.
As we show below, the trivial extension $\Delta(kQ)$ for a non-bipartite Dynkin quiver $Q$ is $(p,q)$-Koszul only in a limited number of cases. In contrast to this, \cref{nrepfinchar} shows that $\Delta(kQ)$ is always minimally almost \mbox{$2$-$T$-Koszul} for $T = kQ$ for a Dynkin quiver $Q$; see \cref{cor sec 7}.

\begin{prop}
Let $Q$ be a connected non-bipartite Dynkin quiver $Q$. Then $\Delta(kQ)$ is $(p,q)$-Koszul for some $(p,q)$ if and only if $Q$ is linearly oriented $A_n$ and $(p,q) = (n,1)$. 
\end{prop}
\begin{proof}
Since $Q$ is assumed non-bipartite, it follows by using results such as in \cite{Schroer}, \cite{Fernandez-Platzeck'02} or \cite{Fernandez-et-al'22} that $\Delta(kQ)$ is not quadratic. By e.g.\ \cite{Schroer}, this is clear in the case where $Q$ has three or more vertices, as the non-bipartite assumption then implies that $kQ$ has a maximal path of length greater than one. 
Moreover, the case where $Q$ has one vertex is excluded by the non-bipartite assumption, and the case of two vertices can be checked to yield relations given by all paths of length three, i.e.\ non-quadratic relations.
Consequently, \cite{Brenner-Butler-King}*{Proposition 3.7} implies that \mbox{$\Delta(kQ)$} is not $(p,q)$-Koszul for $q \geq 2$.

Now note that by \cite{Brenner-Butler-King}*{Proposition 3.5}, an algebra is $(p,1)$-Koszul if and only if it is a truncated algebra, i.e.\ a quotient of a path algebra by some power of its arrow ideal. 
One can check that $\Delta(kQ)$ is a truncated algebra if and only if $Q$ is linearly oriented $A_n$ by again using one of \cite{Schroer}, \cite{Fernandez-Platzeck'02} or \cite{Fernandez-et-al'22}.
Indeed, by using e.g.\ \cite{Schroer}, we see that $\Delta(kQ)$ has a non-monomial relation if and only if there are two or more distinct paths of maximal length in $Q$.
Moreover, by again using \cite{Schroer}, we get that $\Delta(kQ)$ is a truncated algebra with respect to the $(n+1)$-th power of the arrow ideal whenever $Q$ is linearly oriented $A_n$, and hence we are done by \cite{Brenner-Butler-King}*{Proposition 3.5}.
\end{proof}

In the following example, we consider a trivial extension of a representation finite hereditary algebra and verify that it is minimally almost $2$-$T$-Koszul for \mbox{$T = kQ$}. 

\begin{example}
Consider the algebra $\Lambda$ given by the quiver
\vspace{0.5em}
\[
\begin{tikzcd}
1 \rar["\alpha_1"] & 2 \rar["\alpha_2"] & 3 \arrow[ll, bend left, blue, "\rho"]
\end{tikzcd}
\]
\vspace{0.5em}
with relations $\alpha_1 \alpha_2 \rho \alpha_1, \alpha_2 \rho \alpha_1 \alpha_2$ and $\rho \alpha_1 \alpha_2 \rho$. Note that $\Lambda$ is the trivial extension of the path algebra of the quiver $A_3$, which is given as the black part of the quiver above. We equip $\Lambda$ with the trivial extension grading, meaning that we let the blue arrow be in degree $1$ and the rest in degree $0$. 

The indecomposable projective injectives of $\Lambda$ can be represented as
\vspace{0.5em}
\begin{center}
\begin{tikzpicture}
\node[scale=.75] (a) at (0,0){$ 
\arraycolsep=.8pt\def\arraystretch{0.8}
\begin{array}{c}
1_0 \\
2_0\\
3_0\\
1_1
\end{array}
$};
\node[scale=.75] (b) at (2,0){$ 
\arraycolsep=.8pt\def\arraystretch{0.8}
\begin{array}{c}
2_0 \\
3_0\\
1_1\\
2_1
\end{array}
$};
\node[scale=.75] (c) at (4,0){$ 
\arraycolsep=.8pt\def\arraystretch{0.8}
\begin{array}{c}
3_0 \\
1_1\\
2_1\\
3_1
\end{array},
$};
\end{tikzpicture}
\end{center}
\vspace{0.5em}
where the subscripts indicate the degrees of the basis elements. We claim that $\Lambda$ is minimally almost $2$-$\Lambda_0$-Koszul, and verify this using the graded minimal injective resolutions of the summands of $\Lambda_0$. The first parts of these resolutions are given as follows: 
\newpage 
\begin{center}
\begin{tikzpicture}
\node[scale=.75] (a) at (0,0){
$
\arraycolsep=.8pt\def\arraystretch{0.8}
\begin{array}{c}
1_0 \\
2_0 \\
3_0
\end{array}
$
};
\node[scale=.75] (b) at (1.25,1.5){
$
\arraycolsep=.8pt\def\arraystretch{0.8}
\begin{array}{c}
3_{-1} \\
1_0 \\
2_0 \\
3_0
\end{array}
$
};
\node[scale=.75] (c) at (2.5,0){
$
\arraycolsep=.8pt\def\arraystretch{0.8}
\begin{array}{ccc}
3_{-1}\\
\end{array}
$
};
\draw [->] (a) edge (b) (b) edge (c) ;
\end{tikzpicture}
\end{center}

\begin{center}
\begin{tikzpicture}
\node[scale=.75] (a) at (0,0){
$
\arraycolsep=.8pt\def\arraystretch{0.8}
\begin{array}{c}
2_0 \\
3_0 \\
\end{array}
$
};
\node[scale=.75] (b) at (1.25,1.5){
$
\arraycolsep=.8pt\def\arraystretch{0.8}
\begin{array}{c}
3_{-1} \\
1_{0} \\
2_0 \\
3_0 \\
\end{array}
$
};
\node[scale=.75] (c) at (2.5,0){
$
\arraycolsep=.8pt\def\arraystretch{0.8}
\begin{array}{ccc}
3_{-1}\\
1_{0}\\
\end{array}
$
};
\draw [->] (a) edge (b) (b) edge (c) ;
\node[scale=.75] (d) at (3.75,1.5){
$
\arraycolsep=.8pt\def\arraystretch{0.8}
\begin{array}{ccc}
1_{-1}\\
2_{-1}\\
3_{-1}\\
1_{0}\\
\end{array}
$
};
\node[scale=.75] (e) at (5,0){
$
\arraycolsep=.8pt\def\arraystretch{0.8}
\begin{array}{ccc}
1_{-1}\\
2_{-1}\\
\end{array}
$
};
\node[scale=.75] (f) at (6.25,1.5){
$
\arraycolsep=.8pt\def\arraystretch{0.8}
\begin{array}{ccc}
2_{-2}\\
3_{-2}\\
1_{-1}\\
2_{-1}\\
\end{array}
$
};
\node[scale=.75] (g) at (7.5,0){
$
\arraycolsep=.8pt\def\arraystretch{0.8}
\begin{array}{ccc}
2_{-2}\\
3_{-2}\\
\end{array}
$
};
\draw [->] (a) edge (b) (b) edge (c) (c) edge (d) (d) edge (e) (e) edge (f) (f) edge (g);
\end{tikzpicture}
\end{center}

\begin{center}
\begin{tikzpicture}
\node[scale=.75] (c) at (0,0){
$
\arraycolsep=.8pt\def\arraystretch{0.8}
\begin{array}{ccc}
3_{0}\\
\end{array}
$
};
\node[scale=.75] (d) at (1.25,1.5){
$
\arraycolsep=.8pt\def\arraystretch{0.8}
\begin{array}{c}
3_{-1} \\
1_{0} \\
2_{0} \\
3_{0}
\end{array}
$
};
\node[scale=.75] (e) at (2.5,0){
$
\begin{array}{c}
3_{-1} \\
1_{0} \\
2_{0} \\
\end{array}
$
};
\node[scale=.75] (f) at (3.75,1.5){
$
\begin{array}{c}
2_{-1} \\
3_{-1} \\
1_{0} \\
2_{0} \\
\end{array}
$
};
\node[scale=.75] (g) at (5,0){
$
\begin{array}{c}
2_{-1} \\
\end{array}
$
};
\node[scale=.75] (h) at (6.25,1.5){
$
\begin{array}{c}
2_{-2} \\
3_{-2} \\
1_{-1} \\
2_{-1} \\
\end{array}
$
};
\node[scale=.75] (i) at (7.5,0){
$
\begin{array}{c}
2_{-2} \\
3_{-2} \\
1_{-1} \\
\end{array}
$
};
\node[scale=.75] (j) at (8.75,1.5){
$
\begin{array}{c}
1_{-2} \\
2_{-2} \\
3_{-2} \\
1_{-1} \\
\end{array}
$
};
\node[scale=.75] (k) at (10,0){
$
\begin{array}{c}
1_{-2}
\end{array}
$
};
\node[scale=.75] (l) at (11.25,1.5){
$
\begin{array}{c}
1_{-3} \\
2_{-3} \\
3_{-3} \\
1_{-2} \\
\end{array}
$
};
\node[scale=.75] (m) at (12.5,0){
$
\begin{array}{c}
1_{-3} \\
2_{-3} \\
3_{-3} \\
\end{array}
$
};
\draw [->] (c) edge (d) (d) edge (e) (e) edge (f) (f) edge (g) (g) edge (h) (h) edge (i) (i) edge (j) (j) edge (k) (k) edge (l) (l) edge (m);
\end{tikzpicture}
\end{center}
We now choose $\underline{\ell} = (1,3,5)$ and $\underline{g} = (1, 2, 3)$. Since 
\[
(1,3,5) = 2\cdot (1 - 1,2 - 1,3 - 1) + (1,1,1),
\]
we see that condition (3) of \cref{almost n-T-Koszul} holds. Using the injective resolutions above, it is straightforward to check that $\Lambda_0$ is almost $(2,\underline{g},\underline{\ell})$-self-orthogonal and that the minimality condition of \cref{almost n-T-Koszul} holds. As in \cref{ex: bipartite}, one can verify that $\widetilde{\Lambda_0} = \Lambda_0$ is a tilting object in $\stgrmodu \Lambda$. In this case, its endomorphism algebra is isomorphic to $\Lambda_0$.  
\end{example}

Before we present \cref{nrepfinchar}, we give an example of algebras that are almost $1$-$T$-Koszul, but not $(p,q)$-Koszul. 

\begin{example} 
Let \mbox{$R \coloneqq k[x]/\langle x^{m + 1} \rangle$} for $m \geq 1$. Given a basic finite dimensional algebra $A$ of finite global dimension, we consider $\Lambda \coloneqq A \otimes_k R$ endowed with the grading induced by letting $A$ be a graded algebra concentrated in degree $0$ and $R$ have the grading given by letting $x$ be in degree $1$. We claim that $\Lambda$ is almost \mbox{$1$-$DA$-Koszul}, and that it is minimally almost $1$-$DA$-Koszul for $m > 1$. Note that if $m > 1$, we also have that $\Lambda$ cannot be $(p,q)$-Koszul for $q \geq 2$ as it is not quadratic; moreover, as long as $A$ is also not semisimple, it cannot be $(p,q)$-Koszul for any $p$ and $q$ as it $\Lambda$ is then neither quadratic nor a truncated algebra.

To show that $\Lambda$ is $1$-$DA$-Koszul, consider the minimal graded projective resolution of $A$ as a left $\Lambda$-module, for which the initial two terms can be written as
\[
\begin{tikzpicture}
\node[scale=.75] (a) at (0,0){
\begin{tikzcd}
& (A \oplus xA \oplus \cdots \oplus x^{m}A)\langle 1 \rangle \drar["- \cdot x"] & & A \oplus xA \oplus \cdots \oplus x^{m}A \drar & \\
A\langle m + 1 \rangle \urar["- \cdot x^{m}"] & &  xA \oplus \cdots \oplus x^{m}A \urar & & A.
\end{tikzcd}
};
\end{tikzpicture}
\]
Dualizing this yields a graded injective resolution of $DA$ as a right $\Lambda$-module. We set $\ell_i = 2$ and $g_i = m+1$ for all $i \in \{1,\ldots,t\}$ and observe that \mbox{condition (3)} of \cref{almost n-T-Koszul} holds as $a = m$. From the resolution above, we deduce that $\Ext^{1}_{\gr \Lambda}(DA, (DA)\langle i \rangle) \neq 0$ implies $i=1$, which shows that $DA$ is $(1,\underline{g},\underline{\ell})$-self-orthogonal. Noting that $DA$ is a tilting module over $\Lambda_0 = A$ since the global dimension of $A$ is finite, we can conclude that $\Lambda$ is $1$-$DA$-Koszul.

If $m > 1$, then $\Lambda$ is minimally almost $1$-$DA$-Koszul since \mbox{$D(xA \oplus \cdots \oplus x^{m}A)$} is not isomorphic as a $\Lambda$-module to an object in $\add DA$, entailing that we cannot choose $\ell_i$ smaller for any $i \in \{1, \ldots, t\}$. However, observe that if $m = 1$, then we can have $\ell_i = g_i = 1$ since $D(xA) \simeq D(A)\langle -1 \rangle$. It is straightforward to check that $\Lambda$ is then $1$-$DA$-Koszul for this choice of $\ell_i$ and $g_i$. 
\end{example}

From now on, we make the same standing assumptions as we did in order to develop our theory in \cref{T-Koszul and n-rep inf}, see \cref{setup}. This means that \mbox{$\Lambda=\oplus_{i \geq 0}\Lambda_i$} is assumed to be a finite dimensional graded Frobenius algebra of highest degree \mbox{$a \geq 1$} with $\gldim \Lambda_0 < \infty$. Moreover, we consider a basic graded $\Lambda$-module $T$ which is concentrated in degree $0$ and a tilting module over $\Lambda_0$. We assume \mbox{$T_\mu \simeq T$} as $\Lambda$-modules for the Nakayama automorphism $\mu$ of $\Lambda$.

We are now ready to state the main result of \cref{sec 7}, which is an almost $n$-$T$-Koszul analogue of \cref{Characterization}.

\begin{thm}\label{nrepfinchar}
Assume \cref{setup}. The following statements are equivalent:
\begin{enumerate} 
    \item $\Lambda$ is minimally almost $n$-$T$-Koszul.  
    \item $\widetilde{T} = \oplus_{i=0}^{a-1}\Omega^{-ni}T \langle i \rangle$ is a tilting object in $\stgrmodu \Lambda$ and $B = \End_{\stgrmodu \Lambda}(\widetilde{T})$ is $(na-1)$-representation finite.
\end{enumerate}
\end{thm}

We divide the proof of \cref{nrepfinchar} into a series of smaller steps, culminating in \cref{nrepfinchar2to1} and \cref{nrepfinchar1to2} which together prove the characterization. Note that along the way, we show that a minimally almost $n$-$T$-Koszul algebra $\Lambda$ determines combinatorial data including a permutation, and that this corresponds to similar combinatorial data \cite{Herschend-Iyama}*{Proposition 0.2} obtained from $B$ being $(na-1)$-representation finite.

Let us first show that a minimally almost $n$-$T$-Koszul algebra determines a permutation $\pi$ on the set $\{1,\dots,t\}$ in a natural way.

\begin{lem}\label{indeclem}
Assume \cref{setup} and let $\Lambda$ be minimally almost $n$-$T$-Koszul. There is then a permutation $\pi$ on the set $\{1,\dots,t\}$ such that 
    \[
    \Omega^{-\ell_i}T^i \simeq T^{\pi(i)}\langle -g_i \rangle
    \]
for each $i \in \{1,\dots,t\}$.
\end{lem}

\begin{proof}
Let $i \in \{1,\dots,t\}$. As $T$ is almost $(n,\underline{g},\underline{\ell})$-self-orthogonal, there exists an object $T' \in \add T$ such that 
\[
\Omega^{-\ell_i}T^i \simeq T'\langle -g_i \rangle.
\]
Recall that $T$ is concentrated in degree $0$ and that $a \geq 1$. Since it follows from \cref{graded frobenius lemma} that $\Soc \Lambda \subseteq \Lambda_a$, this implies that $T^i$ is not projective as a $\Lambda$-module by \cref{lemma} \textit{(2)}. As $\Omega^{-1}(-)$ is an equivalence on the stable category, the object $T'$ is indecomposable, and consequently $T' \simeq T^{i'}$ for some \mbox{$i' \in \{1,\dots,t\}$}. This allows us to define the map 
\[
\pi \colon \{1,\dots,t\} \rightarrow \{1,\dots,t\}
\]
by setting $\pi(i)=i'$. 

We next show that $\pi$ is injective and hence a permutation. Let \mbox{$\pi(i) = \pi(j)$} and assume $\ell_i \neq \ell_j$. Without loss of generality, we consider the case $\ell_i > \ell_j$. Our assumption yields 
\[
\Omega^{-(l_i-l_j)}T^i \simeq T^j \langle -(g_i - g_j) \rangle.
\]
Observe that the integers $\ell_i'=\ell_i-\ell_j$ and $g_i'=g_i-g_j$ hence satisfy \cref{periodisk ortogonal}. Note in particular that $0 <\ell_i' <\ell_i$ and that positivity of $\ell_i'$ combined with $T$ being almost $(n,\underline{g},\underline{\ell})$-self-orthogonal implies positivity of $g_i'$. This contradicts the minimality condition in \cref{almost n-T-Koszul}, so we must have $\ell_i=\ell_j$, which implies $T^i \simeq T^j$. As $T$ is basic, this means that $i=j$, which finishes our proof.
\end{proof}

Using our fixed decomposition $T \simeq \oplus^t_{i = 1} T^{i}$ together with the definition of $\widetilde{T}$, we see that the algebra $B = \End_{\stgrmodu \Lambda}(\widetilde{T})$ decomposes as
\[
B \simeq \bigoplus_{i=1}^t \bigoplus_{j=0}^{a-1} \Hom_{\stgrmodu \Lambda}(\widetilde{T},X^{i,j}),
\]
where $X^{i,j} = \Omega^{-nj}T^{i}\langle j \rangle$. Hence, the indecomposable projective $B$-modules $$P^{i,j} = \Hom_{\stgrmodu \Lambda}(\widetilde{T},X^{i,j})$$ 
are indexed by the set
\[
J = \{(i,j) \mid 1 \leq i \leq t ~~\text{ and }~~ 0 \leq j \leq a-1\}.
\]
Notice that if $\widetilde{T}$ is a tilting object in $\stgrmodu \Lambda$, then $X^{i,j}$ is the image of $P^{i,j}$ under the equivalence $\D^b(\modu B) \xrightarrow{\simeq} \stgrmodu \Lambda$ described in \cref{tilting section}.

Given a permutation $\sigma$ on the index set $J$, we let $\sigma^{L}_j$ and $\sigma^{R}_i$ be defined by 
\[
\sigma(i,j) = (\sigma^{L}_j(i),\sigma^{R}_i(j)).
\]

We are now ready to prove the first part of \cref{nrepfinchar}. 

\begin{thm}\label{nrepfinchar2to1}
Assume \cref{setup}. If $\widetilde{T}$ is a tilting object in $\stgrmodu \Lambda$ and the algebra \mbox{$B = \End_{\stgrmodu \Lambda}(\widetilde{T})$} is $(na-1)$-representation finite, then $\Lambda$ is minimally almost $n$-$T$-Koszul.
\end{thm}

\begin{proof}
By \cite{Herschend-Iyama}*{Proposition 0.2}, there is a permutation $\sigma$ on $J$ such that for every pair $(i,j)$ in $J$ there is an integer $m_{i,j} \geq 0$ with
\[
\nu_{na-1}^{-m_{i,j}}P^{i,j} \simeq I^{\sigma(i,j)}
\]
in $\D^b(\modu B)$ as $B$ is $(na-1)$-representation finite. Applying $\nu_{na-1}^{-1}$ on both sides, we get
\[
\nu_{na-1}^{-m_{i,j}-1}P^{i,j} \simeq P^{\sigma(i,j)}[na-1].
\]
Since $\widetilde{T}$ is a tilting object in $\stgrmodu \Lambda$, we have an equivalence \mbox{$\D^b(\modu B) \xrightarrow{\simeq} \stgrmodu \Lambda$} as described in \cref{tilting section}. Using that $X^{i,j} = \Omega^{-nj}T^{i}\langle j \rangle$ is the image of $P^{i,j}$ in $\stgrmodu \Lambda$ under this equivalence, combined with the correspondence of Serre functors described at the end of \cref{correspondence of Serre functors}, one obtains 
\[
\Omega^{-(na-1)(m_{i,j}+1) - (m_{i,j}+1)}X^{i,j}_{\mu^{-m_{i,j} -1}}\langle a(m_{i,j}+1) \rangle \simeq \Omega^{-(na-1)}X^{\sigma(i,j)}
\]
since $\Omega(-)_{\mu}\langle - a\rangle$ is the Serre functor of $\stgrmodu \Lambda$. This again yields
\begin{equation} \label{assumed condition}
\Omega^{-nam_{i,j} -1}X^{\mu^{-m_{i,j} -1}(i),j} \simeq X^{\sigma(i,j)} \langle -a(m_{i,j}+1) \rangle,
\end{equation}
as $(-)_\mu$ commutes with cosyzygies and graded shifts and permutes the summands of $T$. It follows that for each pair $(i,j)$ in $J$, we get 
\begin{equation} \label{for pi}
\Omega^{-nam_{i,j}-1 - n(j - \sigma^{R}_{i}(j))}T^{\mu^{-m_{i,j} -1}(i)} \simeq T^{\sigma^{L}_{j}(i)}\langle -a(m_{i,j} +1) + \sigma^{R}_{i}(j) - j  \rangle. 
\end{equation}
Twisting by $\mu^{m_{i,j}+1}$ and setting $j=0$, one obtains
\begin{equation} \label{condition}
    \Omega^{-(nam_{i,0} - n\sigma^{R}_{i}(0) + 1)}T^i \simeq T^{\mu^{m_{i,0}+1}(\sigma^{L}_{0}(i))}\langle -a(m_{i,0} +1) + \sigma^{R}_{i}(0) \rangle. 
\end{equation}

Letting $m_i \coloneqq m_{i,0}$ and $\sigma_i \coloneqq \sigma^{R}_i(0)$, we hence see that with
\mbox{$\ell_i = n(am_i - \sigma_i) + 1$} 
and 
$g_i = a(m_i + 1) - \sigma_i$, 
part (1) in the definition of being almost \mbox{$(n,\underline{g},\underline{\ell})$}-self-orthogonal is satisfied for $T$. Note that since $g_i$ of this form is always positive, so is $\ell_i$, as can be seen by applying \cref{lemma} \textit{(7)}. It is straightforward to check that part (3) of \cref{almost n-T-Koszul} holds for the choice of $\ell_i$ and $g_i$ given above.

In order to show that the minimality condition in \cref{almost n-T-Koszul} is satisfied, consider an integer $k$ satisfying $0 < nk <\ell_i$. Note that we can write \mbox{$k = qa - r$} with $q \geq 1$ and \mbox{$0 \leq r \leq a-1$}. Aiming for a contradiction, assume that there is an integer $j \in \{1,\dots,t\}$ with
\[
\Omega^{-n(qa - r)}T^{i} \simeq T^{j}\langle -(qa-r) \rangle.
\]
Twisting by $(-)_{\mu^{-q}}$ and using the equivalence $\D^b(\modu B) \simeq \stgrmodu \Lambda$ in a similar way as in the beginning of this proof, we obtain
\[
\nu_{na-1}^{-q}P^{i,0} \simeq P^{\mu^{-q}(j),r}.
\]
Applying $\nu_{na-1}$ on both sides yields
\begin{equation}\label{contradiction}
    \nu_{na-1}^{-(q-1)}P^{i,0} \simeq I^{\mu^{-q}(j),r}[-na+1].
\end{equation}
\sloppy From the assumption $nk<\ell_i$ along with the description of $\ell_i$, we deduce that \mbox{$0 \leq q-1 \leq m_i$}. As long as $na>1$, the expression (\ref{contradiction}) hence contradicts \cref{n-rep-fin}, so we can conclude that the minimality condition in \cref{almost n-T-Koszul} is satisfied. If $na=1$, the algebra $B$ is semisimple. In particular, this implies that $\ell_i=1$, so the condition is trivially satisfied in this case.

It remains to prove that $T$ satisfies part (2) of \cref{periodisk ortogonal}, i.e.\ that for each \mbox{$i \in \{1,\dots,t\}$}, we have $\Ext_{\gr \Lambda}^{nk+l}(T,T^i \langle k \rangle) = 0$ for $l \neq 0$ and $nk+l <\ell_i$. If \mbox{$nk+l \leq 0$}, this is immediately clear, so we can assume $nk+l > 0$. This yields
\[
\Ext_{\gr \Lambda}^{nk+l}(T,T^i \langle k \rangle) \simeq \Hom_{\stgrmodu \Lambda}(T,\Omega^{-(nk+l)}T^i\langle k \rangle).
\]
In the case $k < 0$, this is zero by \cref{lemma} \textit{(7)}, and we can thus assume \mbox{$k \geq 0$}. 

As $\widetilde{T}$ is a tilting object in $\stgrmodu \Lambda$, a similar argument as in the proof of \cref{Characterization} yields an isomorphism
\begin{equation} \label{isomorphism}
    \Hom_{\stgrmodu \Lambda}(\widetilde{T},\Omega^{-(nam+l)}X^{\mu^{-m}(i),j} \langle am \rangle) \simeq \Hm^l(\nu_{na-1}^{-m}(P^{i,j}))
\end{equation}
for every pair $(i,j)$ in $J$. By \cref{n-rep-fin}, we know that \mbox{$\Hm^l(\nu_{na-1}^{-m}(P^{i,j})) = 0$} for $l \neq 0$ and $0 \leq m \leq m_{i,j}$ as $B$ is $(na-1)$-representation finite. Using that $(-)_{\mu}$ is an equivalence on $\stgrmodu \Lambda$, that $\widetilde{T}_{\mu} \simeq \widetilde{T}$ and splitting up on summands of \mbox{$\widetilde{T} = \oplus_{s=0}^{a-1}\Omega^{-ns}T\langle s \rangle$}, this yields
\begin{equation} \label{know}
\Hom_{\stgrmodu \Lambda}(T,\Omega^{-(n(am-s+j)+l)}T^i \langle am - s + j\rangle) = 0
\end{equation}
for $l \neq 0$ and $0 \leq m \leq m_{i,j}$. We simplify this by letting $j = 0$. Hence, we have $m_{i,j} = m_i$. In the case $k \leq am_i$, we can write $k=am-s$ for appropriate values of $m$ and $s$, so (\ref{know}) implies our desired conclusion in this case. If \mbox{$k > am_{i}$}, we use the isomorphism $T^i \simeq \Omega^{\ell_i}T^{\pi(i)}\langle -g_i\rangle$ to rewrite 
\[ \Hom_{\stgrmodu \Lambda}(T,\Omega^{-(nk+l)}T^i \langle k\rangle) \simeq \Hom_{\stgrmodu \Lambda}(T,\Omega^{\ell_i-(nk+l)}T^{\pi(i)} \langle k - g_i\rangle).\]
When $nk+l <\ell_i$, this is $0$ by \cref{lemma} \textit{(8)}. To see this, notice that the assumption $k>am_i$ combined with the definition of $g_i$ yields \mbox{$k-g_i \geq 1-a$}. This finishes our proof.
\end{proof}

For convenience, we record the following observation from the proof above.

\begin{prop} \label{prop: argument at beginning of proof of 6.11}
Assume \cref{setup}. Let $\widetilde{T}$ be a tilting object in $\stgrmodu \Lambda$ and $m$ a non-negative integer. For $1 \leq i, i' \leq t$ and $0 \leq j, j' \leq a - 1$, we have that \mbox{$\nu_{na-1}^{-m}P^{i,j} \simeq I^{i',j'}$}
if and only if
\[
\Omega^{-nam -1}X^{\mu^{-m -1}(i),j} \simeq X^{i',j'} \langle -a(m+1) \rangle.
\]
\end{prop}

\begin{proof}
The statement follows from the argument given at the beginning of the proof of \cref{nrepfinchar2to1} up until (\ref{assumed condition}). Note in particular that the argument can be reversed and that it works in the generality stated here.
\end{proof}

Before giving a result which demonstrates how the aforementioned pieces of combinatorial data associated to higher almost Koszul algebras and higher representation finite algebras are related, we need the following lemma.

\begin{lem} \label{B is basic}
Assume \cref{setup}. If $\widetilde{T}$ is a tilting object in $\stgrmodu \Lambda$, then the algebra $B = \End_{\stgrmodu \Lambda}(\widetilde{T})$ is basic.
\end{lem}

\begin{proof}
As $\widetilde{T}$ is a tilting object in $\stgrmodu \Lambda$, it suffices to show that $\widetilde{T}$ is basic. Note that the indecomposable summands of $\widetilde{T}$ are of the form $\Omega^{-nj}T^i\langle j \rangle$ with \mbox{$0 \leq i \leq t$} and \mbox{$0 \leq j \leq a-1$}. 
Assume that we have isomorphic summands
\[
\Omega^{-nj}T^i\langle j \rangle \simeq \Omega^{-nl}T^k\langle l \rangle.
\]
If $j=l$, it follows that $i=k$ as $T$ is basic. Without loss of generality, we hence assume $j>l$. Consider now
\[
\Hom_{\gr \Lambda}(T^i,T^i) \simeq \Hom_{\gr \Lambda}(T^i,\Omega^{-n(l-j)}T^k\langle l-j \rangle),
\]
which is non-zero as $T^i \neq 0$. This contradicts \cref{lemma} \textit{(8)}, as \mbox{$l-j \geq 1-a$} and \mbox{$-n(l-j)>0$}, so we can conclude that $(i,j)=(k,l)$.
\end{proof}

To state and prove some of the results below, it is convenient to introduce some notation that will make the connection between almost \mbox{$n$-$T$-Koszul} algebras and the theory of higher representation finite algebras more transparent.

\begin{defin}\label{prop: equivalence of almost n-T-Koszul defs}
Let $\Lambda$ be an almost $n$-$T$-Koszul algebra with highest degree $a$ and parameters $\underline{g} = (g_1, \ldots, g_t)$ as in \cref{almost n-T-Koszul}.
Considering $g_i$ modulo $a$, we rewrite it as $g_i = a(m_i + 1) - \sigma_i$, where $m_i$ and $\sigma_i$ are non-negative integers such that $\sigma_i \leq a - 1$. 
Define the \textit{associated periodicity data} by \mbox{$\underline{m} = (m_1, \ldots, m_t)$} and $\underline{\sigma} = (\sigma_1, \ldots, \sigma_t)$.
\end{defin}

It is worth noting that the notation introduced above agrees with the form $\ell_i$ and $g_i$ took in the proof of \cref{nrepfinchar2to1}.
Moreover, note that if $\Lambda$ is minimally almost $n$-$T$-Koszul, then it follows by the division algorithm that the sequences of integers 
$\underline{m}=(m_1,\dots,m_t)$ and $\underline{\sigma}=(\sigma_1,\dots,\sigma_t)$ 
in \cref{prop: equivalence of almost n-T-Koszul defs} are uniquely determined.

Recall from \cite{Herschend-Iyama}*{Proposition 0.2} and the proof of \cref{nrepfinchar2to1} that when $B$ is $(na-1)$-representation finite, there is a permutation $\sigma$ on $J$ such that for every pair $(i,j)$ in $J$ there is an integer $m_{i,j} \geq 0$ with
\[
\nu_{na-1}^{-m_{i,j}}P^{i,j} \simeq I^{\sigma(i,j)}.
\]
As before, we use the notation 
\[
\sigma(i,j)=(\sigma_j^L(i),\sigma_i^R(j)).
\]

The proposition below provides information about how the permutation $\sigma$ and the integers $m_{i,j}$ associated to $B$ being $(na-1)$-representation finite are related to the associated periodicity data $\underline{m}$ and $\underline{\sigma}$ from \cref{prop: equivalence of almost n-T-Koszul defs}.

\begin{prop}\label{consistent parameters prop}
Assume \cref{setup}. If $\widetilde{T}$ is a tilting object in $\stgrmodu \Lambda$ and the algebra $B = \End_{\stgrmodu \Lambda}(\widetilde{T})$ is $(na-1)$-representation finite, then $\Lambda$ is minimally almost $n$-$T$-Koszul with associated periodicity data $\underline{m}$ and $\underline{\sigma}$ 
satisfying that 
\mbox{$m_i = m_{i,0}$} and \mbox{$\sigma_i = \sigma_i^R(0)$}. 
Moreover, we have
\[
{\sigma}_{i}^{R}(j) 
= \left\{
	\begin{array}{ll}
       \sigma_{i} + j & \mbox{if } \sigma_{i} + j \leq a-1 \\
	   \sigma_{i} + j - a & \mbox{if } \sigma_{i} + j > a-1
	\end{array}
\right. 
\]
and 
\[
m_{i,j} = 
\left\{
	\begin{array}{ll}
	m_i & \mbox{if } j \leq \sigma_{i}^{R}(j)\\
	m_i - 1 & \mbox{if } j > \sigma_{i}^{R}(j).
	\end{array}
\right.
\] 
Additionally, if $\pi$ is the permutation on $\{1,\dots,t\}$ induced by $\Lambda$ being minimally almost $n$-$T$-Koszul, we have
\[
\sigma^{L}_{j}(i) = \mu^{-m_{i,j} - 1}(\pi(i)).
\] 
\end{prop}

\begin{proof}
Recall first that $\Lambda$ is minimally almost $n$-$T$-Koszul by \cref{nrepfinchar2to1}. 
Moreover, we have $m_i = m_{i,0}$ and $\sigma_i = \sigma_i^R(0)$ by the proof of \cref{nrepfinchar2to1}. From now, consider a fixed integer $i \in \{1,\dots,t\}$ and let $0 \leq j \leq a-1$.

Our next aim is to verify the first two equations in the formulation of the proposition. Note that to get the desired expression for $\sigma_i^R(j)$, it is enough to show that
\[
{\sigma}_{i}^{R}(j) 
= \left\{
	\begin{array}{ll}
       {\sigma}_{i}^{R}(0) + j & \mbox{if } j \leq {\sigma}_{i}^{R}(j) \\
	   {\sigma}_{i}^{R}(0) + j - a & \mbox{if } j > {\sigma}_{i}^{R}(j).
	\end{array}
\right. 
\]
To see that this is sufficient, observe that given the expression above, one has $j \leq {\sigma}_{i}^{R}(j)$ if and only if $\sigma_i + j \leq a-1$. Indeed, if $j \leq {\sigma}_{i}^{R}(j)$, our formula gives
\[
{\sigma}_{i}^{R}(j) = {\sigma}_{i}^{R}(0) + j = \sigma_i + j,
\]
so $\sigma_i + j \leq a-1$. On the other hand, the assumption \mbox{$j > {\sigma}_{i}^{R}(j)$} yields 
\[
\sigma_i^R(j) = \sigma_i^R(0) + j - a = \sigma_i + j - a,
\]
which implies $\sigma_i + j > a-1$.

Assume $j \leq \sigma_i^R(j)$. Observe that one obtains
\[
\Omega^{-nam_{i,j} -1}X^{\mu^{-m_{i,j} -1}(i),0} \simeq X^{\sigma(i,j)-(0,j)} \langle -a(m_{i,j}+1) \rangle
\]
by applying $\Omega^{nj}(-)\langle -j\rangle$ to (\ref{assumed condition}). Our assumption yields \mbox{$0 \leq \sigma_i^R(j)-j \leq a-1$}, so we can apply \cref{prop: argument at beginning of proof of 6.11} to get
\[
\nu_{na-1}^{-m_{i,j}}P^{i,0} \simeq I^{\sigma(i,j)-(0,j)}.
\]
Recall that $\Hm^0(\nu_{na-1}^{-1}-) \simeq \tau_{na-1}^{-1}$ as endofunctors on $\modu B$, where $\tau_{na-1}^{-1}$ denotes the $(na-1)$-Auslander--Reiten translation. Note that the $\tau_{na-1}^{-1}$-orbit of a projective $B$-module contains precisely one injective \cite{Iyama}*{Proposition 1.3}. Compare our expression above with
\[
\nu_{na-1}^{-m_{i,0}}P^{i,0} \simeq I^{\sigma(i,0)}.
\]
If $na>1$, we deduce that $m_{i,j}=m_{i,0}$ and $I^{\sigma(i,j) - (0,j)} \simeq I^{\sigma(i,0)}$. If $na = 1$, then $B$ is semisimple. This implies $m_{i,j}=m_{i,0}=0$, and the same conclusion thus follows. In particular, this yields  
\[
\sigma(i,j) - (0,j) = \sigma(i,0)
\]
as $B$ is basic. Consequently, we obtain our desired expressions for $\sigma_i^R(j)$ and $m_{i,j}$ once we have made the substitutions $m_i = m_{i,0}$ and \mbox{$\sigma_i = \sigma_i^R(0)$}.

For the second case, assume $j > \sigma_i^R(j)$. Note that we now necessarily have $na > 1$ as $m_i=0$ implies $\sigma_i=0$. Apply \mbox{$\Omega^{-n(a - j)}(-)\langle a - j\rangle$} to (\ref{assumed condition}) to get
\[
\Omega^{-na(m_{i,j} + 1) -1}X^{\mu^{-(m_{i,j}+1)}(i),0} \simeq X^{\sigma(i,j) + (0,a - j)} \langle -a((m_{i,j}+1)+1)\rangle.
\]
Our assumption yields $0 < \sigma_i^R(j) + a - j \leq a-1$. Twisting by $(-)_{\mu^{-1}}$ and again applying \cref{prop: argument at beginning of proof of 6.11}, we hence obtain 
\[
\nu_{na-1}^{-(m_{i,j}+1)}P^{i,0} \simeq I^{\mu^{-1}(\sigma_j^L(i)),\sigma_i^R(j)+a-j}.
\]
Similarly as above, this leads to our desired expressions for $\sigma_i^R(j)$ and $m_{i,j}$.

It remains to check that $\sigma^{L}_{j}(i) = \mu^{-m_{i,j} - 1}(\pi(i))$. This follows by applying what we have shown so far to (\ref{for pi}).
\end{proof}

Our next aim is to prove the other direction of this section's main result, i.e.\ \cref{nrepfinchar}. Let us first give an overview of some useful observations.

\begin{lem}\label{almost lemma}
Assume \cref{setup} and let $\Lambda$ be minimally almost $n$-$T$-Koszul with associated periodicity data $\underline{m}$ and $\underline{\sigma}$. 
The following statements hold for \mbox{$1 \leq i \leq t$}:
\begin{enumerate}
    \item We have $\pi \circ \mu = \mu \circ \pi$, where $\pi$ is the permutation on $\{1,\dots,t\}$ induced by $\Lambda$ being minimally almost $n$-$T$-Koszul.
    \item The constants $\ell_i$ and $g_i$ satisfy $\ell_i =\ell_{\mu(i)}$ and $g_i = g_{\mu(i)}$.
    \item The constants $m_i$ and $\sigma_i$ satisfy $m_i = m_{\mu(i)}$ and $\sigma_i = \sigma_{\mu(i)}$.
    \item We have $g_i \geq a$. Moreover, if $m_i = 0$, then $\sigma_i = 0$.
\end{enumerate}
\end{lem}

\begin{proof}
For part \textit{(1)} and \textit{(2)}, recall that $\Omega^{\pm 1}(-)$ and $\langle \pm 1 \rangle$ both commute with $(-)_{\mu}$. This implies that $\Omega^{-\ell_i}T^{\mu(i)} \langle g_i \rangle \simeq T^{\mu(\pi(i))}$ and \mbox{$\Omega^{-\ell_{\mu(i)}}T^{\mu(i)}\langle g_{\mu(i)} \rangle \simeq T^{\pi(\mu(i))}$}, and hence this follows by the division algorithm. 

Comparing the expressions for $g_i$ and $g_{\mu(i)}$, we see that part \textit{(3)} follows from \textit{(2)} by a number theoretical argument.

Part \textit{(4)} is a consequence of the definition of $\ell_i$ and $g_i$. To be precise, it is clear that $m_i = 0$ implies $\sigma_i = 0$ as $\ell_i$ is positive. Using this, the assumption \mbox{$\sigma_i \leq a-1$} yields our first statement.
\end{proof}

Compared to what was the case for $n$-$T$-Koszul algebras, it is somewhat more involved to show that $\widetilde{T}$ is a tilting object in $\stgrmodu \Lambda$ whenever $\Lambda$ is minimally almost $n$-$T$-Koszul. We hence prove this as a separate result. 

\newpage
\begin{prop}\label{tilt obj lem}
Assume \cref{setup}. If $\Lambda$ is minimally almost $n$-$T$-Koszul, then $\widetilde{T}$ is a tilting object in $\stgrmodu \Lambda$.
\end{prop}

\begin{proof}
Since \cref{genlem} yields $\Thick_{\stgrmodu \Lambda}(\widetilde{T}) = \stgrmodu \Lambda$, we only need to check rigidity. As in the proof of \cref{Characterization}, it is enough to verify that
\[
\Hom_{\stgrmodu \Lambda}(T,\Omega^{-(nk+l)}T\langle k \rangle)=0 ~~\text{ for }~~ l \neq 0
\] 
for any integer $k$ with $\lvert k \rvert \leq a-1$. In the cases $nk+l = 0$ and $nk+l < 0$, the argument is exactly the same as in the proof of \cref{Characterization}, so assume \mbox{$nk+l > 0$}. For each summand $T^i$ of $T$, one now obtains 
\[
\Hom_{\stgrmodu \Lambda}(T,\Omega^{-(nk+l)}T^i\langle k \rangle) \simeq \Ext_{\gr \Lambda}^{nk+l}(T,T^i\langle k \rangle).
\]
In the case $nk+l <\ell_i$, this is zero for $l \neq 0$ as $T$ is almost $(n,\underline{g},\underline{\ell})$-self-orthogonal. Otherwise, we use the isomorphism $T^i \simeq \Omega^{\ell_i} T^{\pi(i)}\langle -g_i \rangle$, where $\pi$ is the permutation on $\{1,\dots,t\}$ induced by $\Lambda$ being minimally almost $n$-$T$-Koszul, to rewrite the expression above. In the case $nk+l =\ell_i$, we get
\[
\Hom_{\stgrmodu \Lambda}(T,\Omega^{-(nk+l-\ell_i)}T^{\pi(i)}\langle k - g_i \rangle) = \Hom_{\stgrmodu \Lambda}(T,T^{\pi(i)}\langle k - g_i \rangle).
\]
This is zero as $\lvert k \rvert \leq a-1$ together with \cref{almost lemma} \textit{(4)} 
yields \mbox{$k-g_i < 0$}. If \mbox{$nk+l >\ell_i$}, one obtains
\[
\Hom_{\stgrmodu \Lambda}(T,\Omega^{-(nk+l-\ell_i)}T^{\pi(i)}\langle k - g_i \rangle) \simeq \Ext_{\gr \Lambda}^{nk+l-\ell_i}(T,T^{\pi(i)}\langle k - g_i \rangle).
\]
As $nk+l-\ell_i > 0$ and $k-g_i<0$, the first expression cannot be written as an $n$-multiple of the second. If $nk+l-\ell_i <\ell_{\pi(i)}$, we are hence done. Otherwise, we iterate the argument until we reach our desired conclusion.
\end{proof}

We are now ready to show the other direction of \cref{nrepfinchar}.

\begin{thm}\label{nrepfinchar1to2}
Assume \cref{setup}.
If $\Lambda$ is minimally almost $n$-$T$-Koszul, then $\widetilde{T}$ is a tilting object in $\stgrmodu \Lambda$ and \mbox{$B = \End_{\stgrmodu \Lambda}(\widetilde{T})$} is $(na-1)$-representation finite.
\end{thm}

\begin{proof}
Since $\Lambda$ is minimally almost $n$-$T$-Koszul, the associated periodicity data $\underline{m} = (m_1, \ldots, m_t)$ and $\underline{\sigma} = (\sigma_1, \ldots, \sigma_t)$ is uniquely determined.

As $\widetilde{T}$ is a tilting object in $\stgrmodu \Lambda$ by \cref{tilt obj lem}, we only need to show that \mbox{$B = \End_{\stgrmodu \Lambda}(\widetilde{T})$} is $(na-1)$-representation finite. Let us first use the integers $m_i$ and $\sigma_i$ to define $\sigma_i^R(j)$, $m_{i,j}$ and $\sigma_j^L(i)$ for $(i,j)$ in $J$ by the formulas in the formulation of \cref{consistent parameters prop}. Note that this yields \mbox{$0 \leq \sigma_i^R(j) \leq a-1$}, as well as $1 \leq \sigma_j^L(i) \leq t$ and $m_{i,j}\geq 0$. The latter is a consequence of \cref{almost lemma} \textit{(4)}.

Using that $\Lambda$ is minimally almost $n$-$T$-Koszul, we see that (\ref{for pi}) is satisfied. Furthermore, we can apply \cref{prop: argument at beginning of proof of 6.11}, using that $\widetilde{T}$ is a tilting object in $\stgrmodu \Lambda$. Consequently, one obtains 
\[
\nu_{na-1}^{-m_{i,j}}P^{i,j} \simeq I^{\sigma(i,j)}
\]
for every indecomposable projective $B$-module $P^{i,j}$, where 
\[
\sigma(i,j) : = (\sigma^{L}_{j}(i), \sigma^{R}_{i}(j)).
\]

Our next aim is to show that $\sigma$ is a permutation on $J$. As $J$ is a finite set, it is enough to check injectivity. Recall that $\mu$ and $\pi$ are permutations, and hence injective. Combining this with \cref{almost lemma} \textit{(1)} and \textit{(3)}, notice that also $\sigma_0^L$ is injective.

Assume that $\sigma(i,j)=\sigma(k,l)$ for $(i,j)$ and $(k,l)$ in $J$. If $j \leq \sigma^{R}_{i}(j)$ and \mbox{$l \leq \sigma^{R}_{k}(l)$}, we see that 
\[
\sigma^{L}_0(i) = \sigma^{L}_j(i) = \sigma^{L}_l(k) = \sigma^{L}_0(k),
\]
so $i = k$ by injectivity of $\sigma_0^L$. As we in this case also have
\[
\sigma_i^R(0) + j = \sigma_i^R(j) = \sigma_k^R(l) = \sigma_k^R(0) + l,
\]
it follows that $j=l$, so $\sigma$ is injective. The argument in the case \mbox{$j > \sigma^{R}_{i}(j)$} and \mbox{$l > \sigma^{R}_{k}(l)$} is similar.

By symmetry, it remains to consider the case where $j \leq \sigma^{R}_{i}(j)$ and \mbox{$l > \sigma^{R}_{k}(l)$}. Here, the assumption $\sigma(i,j) = \sigma(k,l)$ yields
\[
\sigma^{L}_0(i) = \sigma^{L}_j(i) =  \sigma^{L}_l(k) = \mu(\sigma^{L}_0(k)).
\]
Consequently, \cref{almost lemma} \textit{(1)} and \textit{(3)} imply that $i = \mu(k)$ and \mbox{$\sigma_i^R(0) = \sigma_k^R(0)$}. As we in this case also have
\[
\sigma^{R}_i(0) + j = \sigma^{R}_i(j) = \sigma^{R}_k(l) = \sigma^{R}_k(0) + l - a,
\]
this means that $j = l - a$, contradicting the assumption \mbox{$0 \leq j,l \leq a - 1$}. Hence, this case is impossible, and we can conclude that $\sigma$ is a permutation.

It now follows that every indecomposable injective, and hence also $DB$, is contained in the subcategory
\[
\U = \add\{\nu_{na-1}^l B \mid l \in \Z\} \subseteq \D^b(\modu B).
\]
By \cref{nrepfindef}, it thus remains to prove that \mbox{$\gldim B \leq na-1$}. To show this, observe first that $B$ has finite global dimension by \cref{gldim lemma}. As $\widetilde{T}$ is a tilting object in $\stgrmodu \Lambda$, it follows from (\ref{isomorphism}) in the proof of \cref{nrepfinchar2to1} that we have
\begin{align*}
\Hm^l(\nu_{na-1}^{-1}P^{i,j}) & \simeq \Hom_{\stgrmodu \Lambda}(\widetilde{T},\Omega^{-(na+l)}X^{\mu^{-1}(i),j} \langle a \rangle) \\ & \simeq \bigoplus_{s=0}^{a-1} \Hom_{\stgrmodu \Lambda}(T,\Omega^{-(n(a+j-s)+l)}T^i \langle a + j - s \rangle)
\end{align*}
for every pair $(i,j)$ in $J$. We want to show that this is zero whenever \mbox{$l \not \in \{1 - na, 0\}$}. Note that the argument for this is similar to the proof of \cref{tilt obj lem}. In particular, it is enough to consider the case $n(a+j-s) +l \geq\ell_i$ for each $i$, since the remaining cases are covered by our previous proof. Using that \mbox{$\Omega^{-\ell_i}T^{i} \simeq T^{\pi(i)}\langle -g_i \rangle$}, the summands in our expression above can be rewritten as
\[
\Hom_{\stgrmodu \Lambda}(T,\Omega^{-n(\sigma_i+j-s - am_i) - (na - 1 +l)}T^{\pi(i)} \langle  \sigma_i + j - s -am_i \rangle). 
\]
If $n(\sigma_i+j-s - am_i) + na - 1 +l <\ell_{\pi(i)}$, this is non-zero only when $l$ is as claimed. Otherwise, \cref{almost lemma} \textit{(4)} implies that we get a negative graded shift in the next step of the iteration, and we are done by the same argument as in the proof of \cref{tilt obj lem}. From this, we see that the claim holds, and so 
\[
\bigoplus_{(i,j) \in J} \Hm^l(\nu_{na-1}^{-1}P^{i,j})
\simeq \Hm^l(\nu_{na-1}^{-1}B) 
\simeq \Ext^{na-1 + l}_B(DB,B) = 0
\]
for $l \not \in \{1-na, 0\}$. By \cref{nrepinflem}, it thus follows that \mbox{$\gldim B \leq na - 1$}. Applying \cref{nrepfindef}, we conclude that $B$ is $(na-1)$-representation finite, which finishes our proof.
\end{proof}

Altogether, combining \cref{nrepfinchar2to1} and \cref{nrepfinchar1to2}, we have now proved \cref{nrepfinchar}. We next present some consequences of this theorem, similar to the ones in \cref{T-Koszul and n-rep inf}. Notice that unlike the corresponding result for $n$-representation infinite algebras, the following corollary is not -- as far as we know -- an analogue of anything existing in the literature. Mutatis mutandis, the proof is the same as that of \cref{motivating result} and is hence omitted. 

\begin{cor}\label{cor sec 7} 
Let $\Lambda = \Lambda_0 \oplus \Lambda_1$ be a finite dimensional graded Frobenius algebra of highest degree $1$ with $\gldim \Lambda_0 < \infty$. Then $\Lambda$ is minimally almost $(n+1)$-Koszul with respect to $T=\Lambda_0$ if and only if $\Lambda_0$ is $n$-representation finite. In particular, we obtain a bijective correspondence 
\[
\left\{
\begin{tabular}{@{}l@{}}
    isomorphism classes of  \\
    basic $n$-representation \\ 
    finite algebras
\end{tabular}
\right\}
{\rightleftarrows}
\left\{
\begin{tabular}{@{}l@{}}
    isomorphism classes of graded symmetric finite \\
    dimensional algebras of highest degree $1$ which  \\
    are minimally almost $(n+1)$-Koszul with \\ 
    respect to their degree $0$ parts
\end{tabular}
\right\},
\]
where the maps are given by $ A \longmapsto \Delta A$ and $\Lambda_0 \longmapsfrom \Lambda$.
\end{cor}

Just like in \cref{T-Koszul and n-rep inf}, it is natural to consider the notion of an almost $n$-$T$-Koszul dual of a given almost $n$-$T$-Koszul algebra. 

\begin{defin}
Let $\Lambda$ be an almost $n$-$T$-Koszul algebra. The \textit{almost $n$-$T$-Koszul dual of $\Lambda$} is given by $\Lambda^! =  \oplus_{i \geq 0} \Ext_{\gr \Lambda}^{ni}(T,T\langle i \rangle)$.
\end{defin}

As before, note that while the notation $\Lambda^{!}$ is potentially ambiguous, it is for us always clear from context which structure the dual is computed with respect to. 

Our next proposition shows that if $\Lambda$ is minimally almost $n$-$T$-Koszul, then the $na$-th preprojective algebra of $B = \End_{\stgrmodu \Lambda}(\widetilde{T})$ is isomorphic to a twist of the $a$-th quasi-Veronese of $\Lambda^!$. The proof is exactly the same as that of the corresponding result in \cref{T-Koszul and n-rep inf}, namely \cref{preprojective and Veronese}.

\begin{prop} \label{preprojective and Veronese for n-rep fin}
Assume \cref{setup} and let $\Lambda$ be minimally almost $n$-$T$-Koszul. Then $\Pi_{na}B \simeq {}_{\langle (\overline{\mu}^{-1})^{[a]}\rangle}(\Lambda^{!})^{[a]}$ as graded algebras. In particular, we have \mbox{$\Pi_{na}B \simeq (\Lambda^{!})^{[a]}$} in the case where $\Lambda$ is graded symmetric.
\end{prop}

The proof of \cref{cor: dual sec. 7} below is similar to that of \cref{cor: dual of triv.ext.} and is hence omitted.

\begin{cor}\label{cor: dual sec. 7}
If $A$ is basic $n$-representation finite, then \mbox{$\Pi_{n+1}A \simeq (\Delta A)^!$} as graded algebras.
\end{cor}

We now illustrate the use of the main result from this section, i.e.\ the characterization given as \cref{nrepfinchar}, in a concrete example.

\begin{example}\label{ex:char7}
Let $A$ denote the path algebra of the quiver
\[
\begin{tikzcd}
    1 \arrow[r,"\alpha_1"] & 2 \arrow[r,"\alpha_2"] & 3 \arrow[r,"\alpha_3"] & 4.
\end{tikzcd}
\]
Its preprojective algebra $\Lambda = \Pi_2(A)$ has quiver given by
\[
\begin{tikzcd}
    1 \arrow[r,bend left,"\alpha_1"] & 2 \arrow[l,bend left,"\alpha'_1"] \arrow[r,bend left,"\alpha_2"] & 3 \arrow[l,bend left,"\alpha'_2"] \arrow[r,bend left,"\alpha_3"] & 4 \arrow[l,bend left,"\alpha'_3"]
\end{tikzcd}
\]
with relations $\alpha_1 \alpha'_1$, $\alpha'_1\alpha_1 - \alpha_2 \alpha'_2$, $\alpha'_2\alpha_2 - \alpha_3 \alpha'_3$ and $\alpha'_3\alpha_3$. We endow $\Lambda$ with the grading induced by putting all arrows in degree $1$. With this grading, the algebra $\Lambda$ is $(3,2)$-Koszul in the sense of \cref{almost Koszul} by \cite{Brenner-Butler-King}*{Corollary 4.3} (see also \cite{Grant-Iyama}*{Theorem B}). It thus follows from \cref{ex: almost-T-Koszul} that $\Lambda$ is minimally almost $1$-Koszul with respect to \mbox{$T=\Lambda_0=\Lambda / \Rad \Lambda$}. Note moreover that the highest degree of $\Lambda$ is $3$ and that the standing assumptions described in \cref{setup} are satisfied. By \cref{nrepfinchar}, we can thus conclude that \mbox{$\widetilde{T} = \oplus_{i=0}^{2}\Omega^{-i}T \langle i \rangle$} is a tilting object in $\stgrmodu \Lambda$ and that $B = \End_{\stgrmodu \Lambda}(\widetilde{T})$ is $2$-representation finite. 

We finish this example by explicitly computing $\widetilde{T}$ and $B$. The injectives of $\Lambda$ are
\begin{center}
\begin{tikzpicture}
\node[scale=.75] (a) at (0,0){$ 
\arraycolsep=.8pt\def\arraystretch{0.8}
\begin{array}{c}
1_{-3} \\
2_{-2} \\
3_{-1} \\
4_{0} \\
\end{array}
$};
\node[scale=.75] (b) at (2.5,0){$ 
\arraycolsep=.8pt\def\arraystretch{0.8}
\begin{array}{cccc}
& 2_{-3}  & & \\
1_{-2} & & 3_{-2} & \\
& 2_{-1} & & 4_{-1} \\
& & 3_{0} & \\
\end{array}
$};
\node[scale=.75] (c) at (5.7,0){$ 
\arraycolsep=.8pt\def\arraystretch{0.8}
\begin{array}{cccc}
& & 3_{-3} & \\
& 2_{-2} & & 4_{-2} \\
1_{-1} & & 3_{-1} & \\
& 2_0 & &\\
\end{array}
$};
\node[scale=.75] (d) at (8.2,0){$ 
\arraycolsep=.8pt\def\arraystretch{0.8}
\begin{array}{c}
4_{-3} \\
3_{-2} \\
2_{-1} \\
1_{0} \\
\end{array},
$};
\end{tikzpicture}
\end{center}
where the subscripts indicate the degree. The tilting object $\widetilde{T}$ has twelve summands, which are given as follows: 

\begin{center}
\begin{tikzpicture}
\node[scale=.75] (a) at (0,0){$ 
\arraycolsep=.8pt\def\arraystretch{0.8}
\begin{array}{c}
1_0 
\end{array}
$};

\node[scale=.75] (g) at (2.5,0){$ 
\arraycolsep=.8pt\def\arraystretch{0.8}
\begin{array}{c}
4_{-2} \\
3_{-1} \\
2_0 
\end{array}
$};

\node[scale=.75] (b) at (5,0){$ 
\arraycolsep=.8pt\def\arraystretch{0.8}
\begin{array}{c}
3_{-2} \\
2_{-1} \\
1_{0} 
\end{array}
$};

\node[scale=.75] (h) at (0,-1.5){$ 
\arraycolsep=.8pt\def\arraystretch{0.8}
\begin{array}{c}
2_0 
\end{array}
$};

\node[scale=.75] (c) at (2.5,-1.5){$ 
\arraycolsep=.8pt\def\arraystretch{0.8}
\begin{array}{cccc}
& & 3_{-2} & \\
& 2_{-1} & & 4_{-1} \\
1_{0} & & 3_{0} & 
\end{array}
$};

\node[scale=.75] (i) at (5,-1.5){$ 
\arraycolsep=.8pt\def\arraystretch{0.8}
\begin{array}{cccc}
& 2_{-2} &  & 4_{-2} \\
1_{-1} &  & 3_{-1} & \\
& 2_{0} & & 
\end{array}
$};

\node[scale=.75] (d) at (0,-3){$ 
\arraycolsep=.8pt\def\arraystretch{0.8}
\begin{array}{c}
3_0 
\end{array}
$};

\node[scale=.75] (j) at (2.5,-3){$ 
\arraycolsep=.8pt\def\arraystretch{0.8}
\begin{array}{cccc}
& 2_{-2} &  & \\
1_{-1} &  & 3_{-1} & \\
& 2_{0} & & 4_{0} 
\end{array}
$};

\node[scale=.75] (e) at (5,-3){$ 
\arraycolsep=.8pt\def\arraystretch{0.8}
\begin{array}{cccc}
1_{-2} & & 3_{-2} & \\
& 2_{-1} & & 4_{-1} \\
& & 3_0 & 
\end{array}
$};

\node[scale=.75] (k) at (0,-4.5){$ 
\arraycolsep=.8pt\def\arraystretch{0.8}
\begin{array}{c}
4_{0} 
\end{array}
$};

\node[scale=.75] (f) at (2.5,-4.5){$ 
\arraycolsep=.8pt\def\arraystretch{0.8}
\begin{array}{c}
1_{-2} \\
2_{-1} \\
3_{0} 
\end{array}
$};

\node[scale=.75] (b) at (5,-4.5){$ 
\arraycolsep=.8pt\def\arraystretch{0.8}
\begin{array}{c}
2_{-2} \\
3_{-1} \\
4_{0} 
\end{array}
$};
\end{tikzpicture}
\end{center}
Note that each row above consists of a simple module and degree shifts of its cosyzygies. 

Inspecting this, we observe that $\widetilde{T}$ decomposes as \mbox{$\widetilde{T} = \widetilde{T}_1 \oplus \widetilde{T}_2$}, where $\widetilde{T}_1$ and $\widetilde{T}_2$ are $\Hom_{\stgrmodu \Lambda}$-orthogonal to each other. The summands $\widetilde{T}_1$ and $\widetilde{T}_2$ has six indecomposable summands each, and we arrange the indecomposables of $\widetilde{T}_1$ as follows:
\vspace{0.001em}

\begin{center}
\begin{tikzpicture}
\node[scale=.75] (a) at (0,0){$ 
\arraycolsep=.8pt\def\arraystretch{0.8}
\begin{array}{c}
1_0 
\end{array}
$};

\node[scale=.75] (b) at (3,0){$ 
\arraycolsep=.8pt\def\arraystretch{0.8}
\begin{array}{c}
3_{-2} \\
2_{-1} \\
1_{0} 
\end{array}
$};

\node[scale=.75] (c) at (1.5,-1.25){$ 
\arraycolsep=.8pt\def\arraystretch{0.8}
\begin{array}{cccc}
& & 3_{-2} & \\
& 2_{-1} & & 4_{-1} \\
1_{0} & & 3_{0} & 
\end{array}
$};

\node[scale=.75] (d) at (0,-2.5){$ 
\arraycolsep=.8pt\def\arraystretch{0.8}
\begin{array}{c}
3_0 
\end{array}
$};

\node[scale=.75] (e) at (3,-2.5){$ 
\arraycolsep=.8pt\def\arraystretch{0.8}
\begin{array}{cccc}
1_{-2} & & 3_{-2} & \\
& 2_{-1} & & 4_{-1} \\
& & 3_0 & 
\end{array}
$};

\node[scale=.75] (f) at (1.5,-3.5){$ 
\arraycolsep=.8pt\def\arraystretch{0.8}
\begin{array}{c}
1_{-2} \\
2_{-1} \\
3_{0} \\
\end{array}
$};
\end{tikzpicture}
\end{center}
Using this, we see that $\End_{\stgrmodu \Lambda}(\widetilde{T}_1)$ is given by the quiver 
\[
\hspace{-1.5em}
\begin{tikzcd}
    1 & & 2\dlar["\beta_2"]\\
    & 3 \ular["\beta_1"] \dlar["\beta_3"] &\\
    4  & & 5 \ular["\beta_4"] \dlar["\beta_6"]\\
    & 6 \ular["\beta_5"] &
\end{tikzcd}
\]
\sloppy with relations $\beta_4 \beta_1, \beta_2 \beta_3$ and $\beta_4 \beta_3 - \beta_6 \beta_5$. One can recognize this as a $2$-representation finite algebra of type $A$, see \cite{Iyama-Oppermann APR}. The endomorphism algebra $\End_{\stgrmodu \Lambda}(\widetilde{T}_2)$ turns out to be isomorphic to $\End_{\stgrmodu \Lambda}(\widetilde{T}_1)$, and we have \mbox{$B \simeq \End_{\stgrmodu \Lambda}(\widetilde{T}_1) \times \End_{\stgrmodu \Lambda}(\widetilde{T}_2)$}. 
\end{example}

\cref{ex:char7} illustrates \cref{cor:application} below, which gives a general method for producing higher representation finite algebras from representation finite hereditary algebras. Note that although the $2$-representation finite algebra we obtain in \cref{ex:char7} is already described in the literature, we believe that applying \cref{cor:application} in cases where the algebra we start with is not of type $A$ will produce novel examples of $n$-representation finite algebras for any odd $n > 1$.

\begin{cor}\label{cor:application}
    Let $A=kQ$ for a Dynkin quiver $Q$ with Coxeter number $h \geq 4$. Consider $\Lambda = \Pi_2(A)$ with grading given by putting arrows in degree $1$, and let $T=\Lambda_0$. Then $B = \End_{\stgrmodu \Lambda}(\widetilde{T})$ is $(h - 3)$-representation finite.  
\end{cor}

\begin{proof}
    By \cite{Brenner-Butler-King}*{Corollary 4.3}, the algebra $\Lambda$ is $(h-2,2)$-Koszul in the sense of \cref{almost Koszul}. Note moreover that by \cite{Brenner-Butler-King}*{Theorem 4.8}, the standing assumptions described in \cref{setup} are satisfied. It thus follows from \cref{ex: almost-T-Koszul} that \cref{nrepfinchar} can be applied, and the claim follows. 
\end{proof}

\begin{acknowledgements}
The authors would like to thank Steffen Oppermann for helpful discussions and Louis-Philippe Thibault and {\O}yvind Solberg for careful reading and helpful suggestions on a previous version of this paper. They also thank Bernhard Keller for pointing out that $\Lambda^!$ has finite global dimension provided that $\Lambda$ is a finite dimensional $n$-$T$-Koszul algebra, simplifying the assumptions in \cref{Existence of T-Koszul equivalence}. The authors profited from use of the software QPA \cite{QPA} to compute examples which motivated parts of the paper.

Parts of this work were carried out while the first author participated in the Junior Trimester Program ``New Trends in Representation Theory'' at the Hausdorff Research Institute for Mathematics in Bonn. She would like to thank the Institute for excellent working conditions. 

The authors thank an anonymous referee for careful reading, insightful comments, and for suggesting new proofs of various results, including \cref{lemma} (1) and (4), and \cref{Existence of T-Koszul equivalence} (2).
\end{acknowledgements}


\begin{bibdiv}
\begin{biblist}

\bib{Amiot}{article}{
   author={Amiot, Claire},
   title={Cluster categories for algebras of global dimension 2 and quivers
   with potential},
   journal={Ann. Inst. Fourier (Grenoble)},
   volume={59},
   date={2009},
   number={6},
   pages={2525--2590},
}

\bib{Amiot-Iyama-Reiten}{article}{
   author={Amiot, Claire},
   author={Iyama, Osamu},
   author={Reiten, Idun},
   title={Stable categories of Cohen-Macaulay modules and cluster
   categories},
   journal={Amer. J. Math.},
   volume={137},
   date={2015},
   number={3},
   pages={813--857},
}

\bib{Auslander-Reiten-Smalo}{book}{
   author={Auslander, Maurice},
   author={Reiten, Idun},
   author={Smal\o , Sverre O.},
   title={Representation theory of Artin algebras},
   series={Cambridge Studies in Advanced Mathematics},
   volume={36},
   publisher={Cambridge University Press, Cambridge},
   date={1995},
   pages={xiv+423},
}

\bib{Beilinson-Ginzburg-Soergel}{article}{
   author={Beilinson, Alexander},
   author={Ginzburg, Victor},
   author={Soergel, Wolfgang},
   title={Koszul duality patterns in representation theory},
   journal={J. Amer. Math. Soc.},
   volume={9},
   date={1996},
   number={2},
   pages={473--527},
}

\bib{Bernstein-Gelfand-Gelfand}{article}{
    author = {Bern\v{s}te\u{\i}n, Joseph N.},
    author = {Gel\cprime fand, Izrail' M.}, 
    author = {Gel\cprime fand, Sergei I.},
     title = {Algebraic vector bundles on {${\bf P}\sp{n}$} and problems of
              linear algebra},
    journal = {Funktsional. Anal. i Prilozhen.},
    volume = {12},
    year = {1978},
    number = {3},
    pages = {66--67}
}

\bib{BP}{article}{
   author={Bondal, Alexey I.},
   author={Polishchuk, Alexander E.},
   title={Homological properties of associative algebras: the method of
   helices},
   journal={Izv. Ross. Akad. Nauk Ser. Mat.},
   volume={57},
   date={1993},
   number={2},
   pages={3--50},
   translation={
      journal={Russian Acad. Sci. Izv. Math.},
      volume={42},
      date={1994},
      number={2},
      pages={219--260},
      issn={1064-5632},
   }
}

\bib{Bongartz-Gabriel}{article}{
    author = {Bongartz, Klaus},
    author = {Gabriel, Peter},
     title = {Covering spaces in representation-theory},
   journal = {Invent. Math.},
    volume = {65},
      year = {1981/82},
    number = {3},
     pages = {331--378}
}

\bib{Brenner-Butler-King}{article}{
    author = {Brenner, Sheila},
    author = {Butler, Michael C. R.},
    author = {King, Alastair D.},
    title = {Periodic algebras which are almost {K}oszul},
    journal = {Algebr. Represent. Theory},
    volume = {5},
    year = {2002},
    number = {4},
    pages = {331--367},
}

\bib{Buchweitz}{book}{
   author={Buchweitz, Ragnar-Olaf},
   title={Maximal Cohen-Macaulay modules and Tate cohomology},
   series={Mathematical Surveys and Monographs},
   volume={262},
   note={With appendices and an introduction by Luchezar L. Avramov,
   Benjamin Briggs, Srikanth B. Iyengar and Janina C. Letz},
   publisher={American Mathematical Society, Providence, RI},
   date={2021},
   pages={xii+175},
}

\bib{CM}{article}{
   author={Cibils, Claude},
   author={Marcos, Eduardo N.},
   title={Skew category, Galois covering and smash product of a
   $k$-category},
   journal={Proc. Amer. Math. Soc.},
   volume={134},
   date={2006},
   number={1},
   pages={39--50},
}

\bib{Darpo-Iyama}{article}{
   author={Darp\"{o}, Erik},
   author={Iyama, Osamu},
   title={$d$-representation-finite self-injective algebras},
   journal={Adv. Math.},
   volume={362},
   date={2020},
   pages={106932, 50},
}

\bib{DJL}{article}{
   author={Dyckerhoff, Tobias},
   author={Jasso, Gustavo},
   author={Lekili, Yanki},
   title={The symplectic geometry of higher Auslander algebras: symmetric
   products of disks},
   journal={Forum Math. Sigma},
   volume={9},
   date={2021},
   pages={Paper No. e10, 49},
}

\bib{Evans-Pugh}{article}{
    author = {Evans, David E.}, 
    author = {Pugh, Mathew},
    title = {The {N}akayama automorphism of the almost {C}alabi-{Y}au
              algebras associated to {$SU(3)$} modular invariants},
    journal = {Comm. Math. Phys.},
    volume = {312},
    year = {2012},
    number = {1},
    pages = {179--222}
}

\bib{Fernandez-et-al'22}{article}{
      title={Characterisations of trivial extensions}, 
      author= {Elsa A. Fern\'{a}ndez},
      author = {Schroll, Sibylle},
      author = {Treffinger, Hipolito}, 
      author = {Trepode, Sonia}, 
      author = {Valdivieso, Yadira},
      year={2022},
      note={arXiv:2206.04581},
      eprint={2206.04581},
      url={https://arxiv.org/abs/2206.04581}
}


\bib{Fernandez-Platzeck'02}{article}{
    author = {Fern\'{a}ndez, Elsa A.},
    author = {Platzeck, Mar\'{\i}a In\'{e}s},
     title = {Presentations of trivial extensions of finite dimensional
              algebras and a theorem of {S}heila {B}renner},
   journal = {J. Algebra},
    volume = {249},
      year = {2002},
    number = {2},
     pages = {326--344},
      issn = {0021-8693},
}

\bib{Fossum-Griffith-Reiten}{book}{
   author={Fossum, Robert M.},
   author={Griffith, Phillip A.},
   author={Reiten, Idun},
   title={Trivial extensions of abelian categories},
   series={Lecture Notes in Mathematics, Vol. 456},
   publisher={Springer-Verlag, Berlin-New York},
   date={1975},
   pages={xi+122},
}

\bib{Grant-Iyama}{article}{
   author={Grant, Joseph},
   author={Iyama, Osamu},
   title={Higher preprojective algebras, Koszul algebras, and
   superpotentials},
   journal={Compos. Math.},
   volume={156},
   date={2020},
   number={12},
   pages={2588--2627},
}

\bib{Green-Reiten-Solberg}{article}{
   author={Green, Edward L.},
   author={Reiten, Idun},
   author={Solberg, \O yvind},
   title={Dualities on generalized Koszul algebras},
   journal={Mem. Amer. Math. Soc.},
   volume={159},
   date={2002},
   number={754},
   pages={xvi+67},
}

\bib{Hanihara}{article}{
    AUTHOR = {Hanihara, Norihiro},
     TITLE = {Auslander correspondence for triangulated categories},
   JOURNAL = {Algebra Number Theory},
    VOLUME = {14},
      YEAR = {2020},
    NUMBER = {8},
     PAGES = {2037--2058},
      ISSN = {1937-0652,1944-7833},
}

\bib{Happel}{article}{
    author = {Happel, Dieter},
     title = {On the derived category of a finite-dimensional algebra},
    journal = {Comment. Math. Helv.},
    volume = {62},
    year = {1987},
    number = {3},
    pages = {339--389},
}

\bib{Herschend-Iyama}{article}{
   author={Herschend, Martin},
   author={Iyama, Osamu},
   title={$n$-representation-finite algebras and twisted fractionally
   Calabi--Yau algebras},
   journal={Bull. Lond. Math. Soc.},
   volume={43},
   date={2011},
   number={3},
   pages={449--466},
}

\bib{Herschend-Iyama2}{article}{
   author={Herschend, Martin},
   author={Iyama, Osamu},
   title={Selfinjective quivers with potential and 2-representation-finite
   algebras},
   journal={Compos. Math.},
   volume={147},
   date={2011},
   number={6},
   pages={1885--1920},
}

\bib{Herschend-Iyama-Minamoto-Oppermann}{article}{
    title={Representation theory of Geigle-Lenzing complete intersections}, 
    author={Herschend, Martin},
    author = {Iyama, Osamu },
    author = {Minamoto, Hiroyuki},
    author = {Oppermann, Steffen},
    journal={to appear in Mem. Amer. Math. Soc.},
}

\bib{Herschend-Iyama-Oppermann}{article}{
   author={Herschend, Martin},
   author={Iyama, Osamu},
   author={Oppermann, Steffen},
   title={$n$-representation infinite algebras},
   journal={Adv. Math.},
   volume={252},
   date={2014},
   pages={292--342},
}

\bib{Iyama_2007}{article}{
   author={Iyama, Osamu},
   title={Auslander correspondence},
   journal={Adv. Math.},
   volume={210},
   date={2007},
   number={1},
   pages={51--82},
}

\bib{Iyama_2007_2}{article}{
   author={Iyama, Osamu},
   title={Higher-dimensional Auslander--Reiten theory on maximal orthogonal
   subcategories},
   journal={Adv. Math.},
   volume={210},
   date={2007},
   number={1},
   pages={22--50},
}

\bib{Iyama_2008}{article}{
   author={Iyama, Osamu},
   title={Auslander--Reiten theory revisited},
   conference={
      title={Trends in representation theory of algebras and related topics},
   },
   book={
      series={EMS Ser. Congr. Rep.},
      publisher={Eur. Math. Soc., Z\"{u}rich},
   },
   date={2008},
   pages={349--397},
}

\bib{Iyama}{article}{
   author={Iyama, Osamu},
   title={Cluster tilting for higher Auslander algebras},
   journal={Adv. Math.},
   volume={226},
   date={2011},
   number={1},
   pages={1--61},
}

\bib{Iyama-Oppermann APR}{article}{
   author={Iyama, Osamu},
   author={Oppermann, Steffen},
   title={$n$-representation-finite algebras and $n$-APR tilting},
   journal={Trans. Amer. Math. Soc.},
   volume={363},
   date={2011},
   number={12},
   pages={6575--6614},
}

\bib{Iyama-Oppermann}{article}{
   author={Iyama, Osamu},
   author={Oppermann, Steffen},
   title={Stable categories of higher preprojective algebras},
   journal={Adv. Math.},
   volume={244},
   date={2013},
   pages={23--68},
}

\bib{Iyama-Wemyss}{article}{
    author = {Iyama, Osamu},
    author = {Wemyss, Michael},
    title = {Maximal modifications and {A}uslander-{R}eiten duality for
              non-isolated singularities},
    journal = {Invent. Math.},
    volume = {197},
    year = {2014},
    number = {3},
    pages = {521--586}
}

\bib{Jasso-Kulshammer}{article}{
    author = {Jasso, Gustavo},
    author = {K\"{u}lshammer, Julian},
    title = {Higher {N}akayama algebras {I}: {C}onstruction},
    note = {With an appendix by K\"{u}lshammer and Chrysostomos Psaroudakis
              and an appendix by Sondre Kvamme},
    journal = {Adv. Math.},
    volume = {351},
    date = {2019},
    pages = {1139--1200},
}

\bib{Kvamme-Jasso}{article}{
   author={Jasso, Gustavo},
   author={Kvamme, Sondre},
   title={An introduction to higher Auslander--Reiten theory},
   journal={Bull. Lond. Math. Soc.},
   volume={51},
   date={2019},
   number={1},
   pages={1--24},
}

\bib{Keller}{article}{
   author={Keller, Bernhard},
   title={Deriving DG categories},
   journal={Ann. Sci. \'{E}cole Norm. Sup. (4)},
   volume={27},
   date={1994},
   number={1},
}

\bib{Keller 2011}{article}{
   author={Keller, Bernhard},
   title={Deformed Calabi--Yau completions},
   note={With an appendix by Michel Van den Bergh},
   journal={J. Reine Angew. Math.},
   volume={654},
   date={2011},
   pages={125--180},
}

\bib{Keller04}{article}{
    author = {Keller, Bernhard},
    title = {On differential graded categories},
    booktitle = {International {C}ongress of {M}athematicians. {V}ol. {II}},
    pages = {151--190},
    publisher = {Eur. Math. Soc., Z\"{u}rich},
    date = {2006}
}

\bib{L}{article}{
   author={Lunts, Valery A.},
   title={Categorical resolution of singularities},
   journal={J. Algebra},
   volume={323},
   date={2010},
   number={10},
   pages={2977--3003},
}

\bib{LS}{article}{
   author={Lunts, Valery A.},
   author={Schn\"{u}rer, Olaf M.},
   title={Smoothness of equivariant derived categories},
   journal={Proc. Lond. Math. Soc. (3)},
   volume={108},
   date={2014},
   number={5},
   pages={1226--1276},
}

\bib{Madsen 2011}{article}{
   author={Madsen, Dag Oskar},
   title={On a common generalization of Koszul duality and tilting
   equivalence},
   journal={Adv. Math.},
   volume={227},
   date={2011},
   number={6},
   pages={2327--2348},
}

\bib{Minamoto 2012}{article}{
    author = {Minamoto, Hiroyuki},
    title = {Ampleness of two-sided tilting complexes},
    journal = {Int. Math. Res. Not. IMRN},
    year = {2012},
    number = {1},
    pages = {67--101}
}

\bib{Minamoto & Mori 2011}{article}{
    author = {Minamoto, Hiroyuki},
    author = {Mori, Izuru},
    title = {The structure of {AS}-{G}orenstein algebras},
    journal = {Adv. Math.},
    volume = {226},
    year = {2011},
    number = {5},
    pages = {4061--4095},
}

\bib{Nastasescu-Van Oystaeyen}{book}{
   author={N\u{a}st\u{a}sescu, Constantin},
   author={Van Oystaeyen, Freddy},
   title={Methods of graded rings},
   series={Lecture Notes in Mathematics},
   volume={1836},
   publisher={Springer-Verlag, Berlin},
   date={2004},
   pages={xiv+304},
}

\bib{Oppermann-Thomas}{article}{
    author = {Oppermann, Steffen},
    author = {Thomas, Hugh},
    title = {Higher-dimensional cluster combinatorics and representation
              theory},
    journal = {J. Eur. Math. Soc. (JEMS)},
    volume = {14},
    year = {2012},
    number = {6},
    pages = {1679--1737}
}

\bib{Priddy}{article}{
    author = {Priddy, Stewart B.},
    title = {Koszul resolutions},
    journal = {Trans. Amer. Math. Soc.},
    volume = {152},
    year = {1970},
    pages = {39--60}
}

\bib{QPA}{webpage}{
  author={The QPA-team},
  title={QPA - Quivers, path algebras and representations},
  date={Version 1.31; 2018},
  myurl={https://folk.ntnu.no/oyvinso/QPA/},
}

\bib{Reiten-Van den Bergh}{article}{
   author={Reiten, Idun},
   author={Van den Bergh, Michel},
   title={Noetherian hereditary abelian categories satisfying Serre duality},
   journal={J. Amer. Math. Soc.},
   volume={15},
   date={2002},
   number={2},
   pages={295--366},
}

\bib{Rickard}{article}{
   author={Rickard, Jeremy},
   title={Derived categories and stable equivalence},
   journal={J. Pure Appl. Algebra},
   volume={61},
   date={1989},
   number={3},
   pages={303--317},
}

\bib{Schroer}{article}{
  author = {Schr\"{o}er, Jan},
  title = {On the quiver with relations of a repetitive algebra},
  journal = {Arch. Math. (Basel)},
  volume = {72},
  date = {1999},
  number = {6},
  pages = {426--432}
}

\bib{Stacks project}{webpage}{
  author={The Stacks Project Authors},
  title={Stacks Project},
  date={2018},
  myurl={https://stacks.math.columbia.edu},
}

\bib{Yamaura 2013}{article}{
    author = {Yamaura, Kota},
    title = {Realizing stable categories as derived categories},
    journal = {Adv. Math.},
    volume = {248},
    year = {2013},
    pages = {784--819},
}

\end{biblist}
\end{bibdiv}
\end{document}